\documentclass[aos,preprint]{imsart}{}

\RequirePackage[OT1]{fontenc}
\RequirePackage{amsthm,amsmath}
\RequirePackage[numbers]{natbib}
\RequirePackage[colorlinks,citecolor=blue,urlcolor=blue]{hyperref}
\usepackage{amsmath, amsthm, amsfonts,amssymb}
\usepackage{mathtools}
\usepackage{bm}
\usepackage{graphicx, color}
\usepackage{subfigure}
\usepackage{epstopdf}
\usepackage{float}
\usepackage{array}
\usepackage{booktabs}
\usepackage{multirow}
\usepackage{algorithm}
\usepackage{algpseudocode}

\newcommand{\head}[1]{\textnormal{\textbf{#1}}}
\newcommand{\normal}[1]{\multicolumn{1}{l}{#1}}
\DeclareMathOperator*{\argmin}{argmin}
\newcommand{\norm}[1]{\left\lVert#1\right\rVert}

\startlocaldefs
\numberwithin{equation}{section}
\theoremstyle{plain}
\newtheorem{thm}{Theorem}[section]
\theoremstyle{remark}
\newtheorem{exam}[thm]{Example}
\newtheorem{rem}[thm]{Remark}

\newtheorem{defn}[thm]{Definition}
\newtheorem{lem}[thm]{Lemma}
\newtheorem{assu}[thm]{Assumption}
\newtheorem{prop}[thm]{Proposition}

\newtheorem{cor}[thm]{Corollary}
\endlocaldefs

\begin{document}

\begin{frontmatter}
\title{Estimation and inference for precision matrices of non-stationary time series}
\runtitle{Time series precision matrix estimation}

\begin{aug}
\author{\fnms{Xiucai } \snm{Ding}\thanksref{m1}\ead[label=e1]{ xiucai.ding@mail.utoronto.ca}}, and 
\author{\fnms{Zhou} \snm{Zhou}\thanksref{m2}\ead[label=e2]{zhou@utstat.toronto.edu}}

\runauthor{X. Ding and Z. Zhou}

\affiliation{University of Toronto \thanksmark{m1}
\thanksmark{m2}}

\address{Department of Statistics \\
University of Toronto \\
100 St. George St. \\
Toronto, Ontario M5S 3G3 \\
Canada\\
\printead{e1}\\
\phantom{E-mail:\ }
\printead*{}}

\address{Department of Statistics \\
University of Toronto \\
100 St. George St. \\
Toronto, Ontario M5S 3G3 \\
Canada\\
\printead{e2}\\
\phantom{E-mail:\ }
\printead*{}}
\end{aug}

%
%
%
%

%

\begin{abstract}
We consider the estimation and inference of precision matrices of a rich class of locally stationary linear and nonlinear time series assuming that only one realization of the time series is observed.  Using a Cholesky decomposition technique, we show that the precision matrices can be directly estimated via a series of least squares linear regressions with smoothly time-varying coefficients. The method of sieves is utilized for the estimation and is shown to be  optimally adaptive in terms of estimation accuracy and efficient in terms of computational complexity. We establish an asymptotic theory for a class of ${\cal L}^2$ tests based on the nonparametric sieve estimators. The latter are used for testing whether the precision matrices are diagonal or banded. A Gaussian approximation result is established for a wide class of quadratic forms of non-stationary and possibly nonlinear processes of diverging dimensions, which is of interest by itself.
\end{abstract}

%
\begin{keyword}
\kwd{Non-stationary time series}
\kwd{Precision matrices}
\kwd{Cholesky decomposition}
\kwd{Sieve estimation}
\kwd{High dimensional Gaussian approximation}
\kwd{Random matrices}
\kwd{White noise and bandedness tests}
\end{keyword}

\end{frontmatter}

\section{Introduction}  
Consider a centered non-stationary time series $x_{1,n}$, $\cdots,$ $x_{n,n} \in \mathbb{R} $. Denote by $\Omega_n:=[\mbox{Cov}(x_{1,n},\cdots, x_{n,n})]^{-1}$ the precision matrix of the series.  Modelling, estimation and inference of $\Omega_n$ are of fundamental importance in a wide range of problems in time series analysis. For example, the ${\cal L}^2$ optimal linear forecast of $x_{n+1,n}$ based on $x_{1,n}$, $\cdots,$ $x_{n,n}$ is determined by $\Omega_n$ and the covariance between $x_{n+1,n}$ and $(x_{1,n},\cdots, x_{n,n})$ \cite{BD}. In time series regression with fixed regressors, the best linear unbiased estimator of the regression coefficient is a weighted least squares estimator with weights proportional to the square root of the precision matrix of the errors \cite{HD}. Furthermore, the precision matrix is a key part in Gaussian likelihood and quasi likelihood estimation and inference of time series \cite{BD, MN}. We shall omit the subscript $n$ in the sequel if no confusions arise.  

Observe that $\Omega$ is an $n\times n$ matrix. When the time series length $n$ is at least moderately large, it is generally not a good idea to first estimate the covariance matrix of $(x_{1},\cdots, x_{n})$ and then invert it to obtain an estimate of $\Omega$. One main reason is that small errors in the covariance matrix estimation may be amplified through inversion when $n$ is large, especially when the condition number of the covariance matrix is large. Also matrix inversion is not computationally efficient for large $n$. As a result it is desirable to directly estimate $\Omega$.  In this paper, we utilize a Cholesky decomposition technique to directly estimate $\Omega$ through a series of least squares linear regressions. Specifically, write
\begin{equation}\label{defn_bestlinear}
x_i=\sum_{j=1}^{i-1} \phi_{ij} x_{i-j}+\epsilon_i, \ i=2,\cdots, n
\end{equation}
where $\sum_{j=1}^{i-1} \phi_{ij} x_{i-j}:=\widehat{x}_i$ is the best linear forecast of $x_i$ based on $x_1,\cdots,x_{i-1}$ and $\epsilon_i$ is the forecast error. Let $\epsilon_1:=x_1$ and denote by $\sigma^2_i$ the variance of $\epsilon_i$, $i=1,2,\cdots,n$. Observe that $\epsilon_i$ are uncorrelated random variables. As a result it is straightforward to show that \cite{P}
\begin{equation} \label{estimationeq}
\Omega=\Phi^*\widetilde{\mathbf{D}}\Phi,
\end{equation}
where the diagonal matrix $\widetilde{\mathbf{D}}=\operatorname{diag}\{\sigma_1^{-2}, \cdots, \sigma_n^{-2}\}$, $\Phi$ is a lower triangular matrix having ones on its diagonal and $-\phi_{ij}$ at its $(i,i-j)-$th element  for $j<i$ and $*$ denotes matrix or vector transpose. The most significant advantage of the above Cholesky decomposition is structural simplification that transfers the difficult problem of precision matrix estimation to that of estimating a series of least squares regression coefficients and error variances.

However, the Cholesky decomposition idea is not directly applicable to precision matrix estimation of non-stationary time series. The reason is that there are in total $n(n+1)/2$ regression coefficients and error variances to be estimated in the Cholesky decomposition of $\Omega$. Meanwhile, observe that there are also $n(n+1)/2$ parameters to be estimated for the precision matrix of a general non-stationary time series. Hence Cholesky decomposition, though performs structural simplification, does not reduce the dimensionality of the parameter space. On the other hand, we only observe one realization of the time series with $n$ observations. As a result dimension reduction techniques with natural assumptions in non-stationary time series analysis are needed for the estimation of $\Omega$.

We adopt two natural and widely used assumptions in non-stationary time series for the dimension reduction. First such assumption is local stationarity which refers to slowly or smoothly time-varying underlying data generating mechanisms of the series. Utilizing the locally stationary framework in Zhou and Wu \cite{WZ1}, we show that, for a wide class of locally stationary nonlinear processes, each off-diagonal of the $\Phi$ matrix as well as the error variance series $\sigma^2_i$ can be well approximated by smooth functions on [0,1].  Specifically, we show that there exist smooth functions  $\phi_j(\cdot)$ and $g(\cdot)$ such that $\sup_{i>b}|\phi_{ij}-\phi_j(i/n)|=o(n^{-1/2})$, $j=1,2,\cdots,n$ and $\sup_{i>b}|\sigma^2_i-g(i/n)|=o(n^{-1/2})$, where $b=b_n$ diverges to infinity with $b/n\rightarrow 0$ whose specific value will be determined later in the article. To our knowledge, the latter is the first result on smooth approximation to general non-stationary precision matrices. From classic approximation theory \cite{SC}, a $d$ times continuously differentiable function can be well approximated by a basis expansion with $O((n/\log n)^{1/(2d+1)})$ parameters. Thanks to the local stationarity assumption, the number of parameters needed for estimating $\sigma^2_i$ is reduced from $n$ to $O((n/\log n)^{1/(2d+1)})$. Similar conclusion holds for each off-diagonal of $\Phi$. 

The second assumption we adopt is short range dependence which refers to fast decay of the dependence between $x_i$ and $x_{i+j}$ as $j$ diverges. Using the physical dependence measures introduced in Zhou and Wu \cite{WZ1}, modern operator spectral theory and approximation theory \cite{DMS, CT}, we show, as a theoretical contribution of the paper, that the off-diagonals of $\Phi$ decays fast to zeros for a general class of locally stationary short range dependent processes. Specifically, we show that $\phi_{ij}$ can be effectively treated as 0 whenever $j>b$. Hence the total number of parameters one needs to estimate is reduced to the order $b[b+(n/\log n)^{1/(2d+1)}]$ which is typically much smaller than the sample size $n$.

Now we utilize the method of sieves to estimate the smooth functions $\phi_j(\cdot)$ and $g(\cdot)$ mentioned above. The method of sieves refers to approximating an infinite dimensional space with a sequence of finer and finer finite dimensional subspaces. Typical examples include Fourier, wavelet and orthogonal polynomial approximations to smooth functions on compact intervals. We refer to \cite{CXH} by Chen for a thorough review of the subject. There are two major advantages of the sieve method when used for precision matrix estimation. First, many sieve estimators, such as the Fourier and wavelet methods mentioned above, do not have inferior performances at the boundary of the estimating interval. This is important as inaccurate estimates at the boundary may drastically lower the accuracy of the whole precision matrix estimation even though entries are well estimated in the interior. Second, the computation complexities of many sieve methods are both adaptive (to the smoothness of the functions of interest) and efficient. When estimating one smooth function of time, local methods such as the kernel estimation perform one regression at each time point. This could be computational inefficient when $n$ is large. On the contrary, the above mentioned three sieve methods only need to perform a single regression at the whole time interval with the number of covariates determined by the smoothness of the function of interest. In many cases this yields a much faster estimation. For instance, in the extreme case where the time series dependence is exponentially decaying and the functions are infinitely differentiable, the sieve method only needs $O(n\log^5 n)$ computation complexity to estimate $\Omega$. Under the same scenario, the kernel method is of $O(n^2k_n\log^2 n)$  computation complexity where $k_n$ is the bandwidth used for the regression and is typically of the order $n^{-1/5}$.

In this paper, we show that the sieve estimates of the functions $\phi_j(\cdot)$ achieve, uniformly over time and $j$, minimax rate  for nonparametric function estimation \cite{SC}. This extends previous convergence rate results on nonparametric sieve regression to the case of diverging number of covariates and non-stationary predictors and errors. Combining the latter result with modern random matrix theory
\cite{JT}, we show that the operator norm of the estimated precision matrix converges at a fast rate which is determined by the strength of time series dependence and smoothness of the underlying data generating mechanism. In the best scenario where the dependence is exponentially decaying and $\phi_j(\cdot)$ and $g(\cdot)$ are infinitely differentiable, the convergence rate is shown to be of the order $\log^3n/\sqrt{n}$, which is almost as fast as parametrically estimating a single parameter from i.i.d. samples. The sieve estimators have already been used to estimate the smooth conditional mean function in various settings. For instance, in \cite{BCCK}, the authors proved that the sieve least square estimators could achieve minimax rate in the sense of sup-norm loss for a fixed number of i.i.d regressors and errors with a general class of sieve basis functions; later Chen and Christensen \cite{CC} showed that the spline and wavelet sieve regression estimators attain the above global uniform convergence rate for a fixed number of weakly dependent and stationary regressors. In this article, we study nonparametric sieve estimates for locally stationary time series with diverging number of covariates under physical dependence and obtain the same minimax rate for the functions $\phi_j(\cdot).$

After estimating $\Omega$, one may want to perform various tests on its structure. In this paper, we focus on two such tests, one on whether $\{x_i\}_{i=1}^n$ is a non-stationary white noise and the other on whether $\Omega$ is banded. Two test statistics based on the ${\cal L}^2$ distances between the estimated and hypothesized $\Phi$ are proposed. These tests boil down to quadratic forms of the estimated sieve regression coefficients which are quadratic forms of non-stationary, dependent vectors of diverging dimensionality. To our knowledge, there have been no previous works on ${\cal L}^2$ inference of nonparametric sieve estimators as well as the inference of high dimensional quadratic forms of non-stationary nonlinear time series. Here we utilize Stein's method together with an $m$-dependence approximation technique and prove that the laws of a large class of quadratic forms of non-stationary nonlinear processes with diverging dimensionality can be well approximated by those of quadratic forms of diverging dimensional Gaussian processes. Consequently asymptotic normality can be established for those high dimensional quadratic forms. The latter Gaussian approximation result is of separate interest and may of wider applicability in non-stationary time series analysis. In \cite{XZW}, Xu, Zhang and Wu derived the $\mathcal{L}^2$ asymptotics for the quadratic form $\overline{X}^*\overline{X},$ where $\overline{X}$ is the sample mean of $n$ i.i.d. random vectors and $\overline{X}^*$ is its transpose. In the present paper, we prove new and much more general $\mathcal{L}^2$ asymptotics for quadratic forms $\overline{\mathbf{Z}}^* E \overline{\mathbf{Z}}$ for any bounded positive semi-definite matrix $E$ using Stein's method \cite{AR, ZC}, where $\overline{\mathbf{Z}}$ is the sample mean of a high dimensional, non-stationary and dependent process. It is very interesting that similar ideas have been used in proving the universality of random matrix theory \cite{DY, EYY, KY1, TV}. 

 We point out that the idea of Cholesky decomposition has been used in time series analysis under some different settings when multiple replicates of the vector of interest are available. Assuming a longitudinal setup where multiple realizations can be observed, Wu and Pourahmadi \cite{WP} studied the estimation of covariance matrices using nonparametric smoothing techniques. Bickel and Levina \cite{BL} considered estimating large covariance and precision matrices by either banding or tapering the sample covariance matrix and its inverse assuming that multiple independent samples can be observed. On the other hand, we assume that only one realization of the time series is observed which is the case in many real applications. Hence none of the aforementioned results can be applied under this scenario. 

Finally, we mention that estimating large dimensional covariance and precision matrices has attracted much attention in the last two decades.  One main research line is to assume that we can observe $n$ i.i.d copies of a $p$-dimensional random vector. When $p$ is comparable or larger than $n,$ it is well-known that sample covariance and precision matrices  are inconsistent estimators \cite{DXC1709, DP}. To overcome the difficulty from high dimensionality, researchers usually impose two main structural assumptions in order to consistently estimate the covariance and precision matrices: sparsity structure and factor model structure. Various families of covariance matrices and regularization methods have been introduced assuming some types of sparsity, this includes the bandable covariance matrices \cite{BL, CRZ, WP}, sparse covariance matrices \cite{CZ, LF, ZZM} and sparse precision matrices \cite{yuan2010, yuan2007}.  
 On the other hand, factor models in the high dimensional setting have been used in a range of applications
in finance and economics.
For a comprehensive review on factor model based methods, we refer to \cite{FLW}. 
 Although high dimensional covariance and precision matrix estimation has witnessed unprecedented development, statistical inference of high dimensional and non-stationary time series remains largely untouched so far. Under stationarity, \cite{MP, XW1} considers thresholding and banding techniques for estimating the covariance matrix with only one realization of the series. Under sparsity assumptions, \cite{CXW} estimates {\it marginal} covariance and precision matrices of high-dimensional stationary and locally stationary time series using thresholding and Lasso techniques. Note that when estimating marginal covariance or precision matrices of a $p$ dimensional time series of length $n$, the series can be viewed as $n$ dependent replicates of the vector of interest which is completely different than the situation considered in this article.

The rest of the paper is organized as follows. In Section \ref{sec_localstation}, we introduce a rich class of non-stationary (locally stationary) and nonlinear time series and study the theoretical properties of its covariance and precision matrices. In Section \ref{sec_est}, we consistently estimate the precision matrices and provide convergent rates for these estimators. In Section \ref{sec_test}, we propose two efficient testings using some simple statistics from our estimation procedure. In Appendix \ref{sec_examsimu}, we give Monte Carlo simulations to illustrate our results. Technical proofs are left to Appendix \ref{sec_proofs} and \ref{appendix_a}. Some auxiliary lemmas are provided in Appendix \ref{app_d}. 



\section{Locally stationary time series}\label{sec_localstation}
Consider a locally stationary time series \cite{ ZZ, WZ1, WZ2} 
\begin{equation}\label{defn_model}
x_i=G(\frac{i}{n}, \mathcal{F}_i),
\end{equation}
where $\mathcal{F}_i=(\cdots, \eta_{i-1}, \eta_i)$ and $\eta_i, \ i  \in \mathbb{Z}$ are i.i.d  random variables, and $G:[0,1] \times \mathbb{R}^{\infty} \rightarrow \mathbb{R}$ is a measurable function such that $\xi_i(t):=G(t, \mathcal{F}_i)$ is a properly defined random variable for all $t \in [0,1].$ The above represents a wide class of locally stationary linear and nonlinear processes. We refer to Zhou and Wu \cite{WW, WZ1, WZ2} for detailed discussions and examples. And following \cite{WW, WZ1, WZ2}, we introduce the following dependence measure to quantify the temporal dependence of (\ref{defn_model}).
\begin{defn}\label{defn_physical} Let $\{\eta_i^{\prime}\}$ be an i.i.d. copy of $\{\eta_i\}.$ Assuming that for some $q>0,\ ||x_i||_q<\infty, $ where $|| \cdot ||_q=[\mathbb{E} |\cdot|^q ]^{1/q}$ is the $\mathcal{L}_q$ norm of a random variable.  For $j \geq 0,$ we define the physical dependence measure by 
\begin{equation}\label{eq_phyoriginal}
\delta(j,q):=\sup_{t \in [0,1]} \max_{i} \left|\left|G(t, \mathcal{F}_i)-G(t, \mathcal{F}_{i,j})\right|\right|_q,
\end{equation}
where $\mathcal{F}_{i,j}:=(\mathcal{F}_{i-j-1}, \eta^{\prime}_{i-j},\eta_{i-j+1}, \cdots, \eta_i).$ 
\end{defn}   
The measure $\delta(j, q)$ quantifies the changes in the system's output when the input of the system
$j$ steps ahead is changed to an i.i.d. copy. If the change is small, then we have short-range
dependence. It is notable that $\delta(j,q)$ is related to the data generating mechanism and can be easily computed. We refer the readers to  \cite[Section 4]{WZ1} for examples of such computation. 

In the present paper, we impose the following assumptions on (\ref{defn_model}) and the physical dependence measure to control the temporal dependence of the non-stationary time series. 
\begin{assu}\label{assu_phylip} There exists a  constant $\tau>10$ and $q>4$, for some constant $C>0$, we have that
\begin{equation}\label{assum_phy}
\delta(j,q) \leq Cj^{-\tau},  \ j \geq 1.
\end{equation}
Furthermore,  $G$ defined in (\ref{defn_model}) satisfies the property of stochastic Lipschitz continuity,  for any $t_1, t_2 \in [0,1]$, we have
\begin{equation}\label{assum_lip}
\left| \left| G(t_1, \mathcal{F}_{i})-G(t_2,\mathcal{F}_i) \right|\right|_q \leq C|t_1-t_2|.
\end{equation}
We also assume that
\begin{equation}\label{assum_moment}
\sup_t \max_i ||G(t,\mathcal{F}_i) ||_q<\infty.
\end{equation}
\end{assu}
(\ref{assum_phy}) indicates that the time series has short-range dependence. (\ref{assum_lip}) implies that $G(\cdot, \cdot)$ changes smoothly over time and ensures local stationarity. Furthermore, for each fixed $t \in [0,1],$ denote \begin{equation}
\gamma(t,j)=\mathbb{E}(G(t, \mathcal{F}_0), G(t, \mathcal{F}_{j})),
\end{equation}
(\ref{assum_lip}) and (\ref{assum_moment}) imply that $\gamma(t,j)$ is Lipschitz continuous in $t$. Furthermore, we need the following mild assumption on the smoothness of $\gamma(t,j).$
\begin{assu}\label{assu_smmothness} For any $j \geq 0,$ we assume that $\gamma(t,j) \in C^d([0,1]), d>0$ is some integer, where $C^d([0,1])$ is the function space on $[0,1]$ of continuous functions that have continuous first $d$ derivatives. 
\end{assu}

 \subsection{Examples} 
In this subsection, we list a few examples of locally stationary processes satisfying  Assumption \ref{assu_phylip} and \ref{assum_phy}.
We first consider two linear processes.  
\begin{exam}[Nonstationary linear processes]\label{exam_linear} Let $\{\epsilon_i\}$ be i.i.d random variables, let $a_j(\cdot), j=0,1,\cdots$ be $C^d([0,1])$ functions such that 
\begin{equation*}
G(t, \mathcal{F}_i)=\sum_{k=0}^{\infty} a_j(t) \epsilon_{i-k}.
\end{equation*}   
The above model is studied in \cite[Section 4.1]{WZ1}.  By \cite[Proposition 2]{WZ1}, we find that Assumption \ref{assu_phylip} will be satisfied if 
\begin{equation*}
 \sup_{t \in [0,1]} |a_j(t)|^{\min(2,q)} \leq C j^{-\tau}, \ j \geq 1; \ \sum_{j=0}^{\infty} \sup_{t \in [0,1]} |a_j'(t)|^{\min(2,q)}<\infty, 
\end{equation*} 
for some constant $C>0.$ Furthermore, by the rule of term by term differentiation,  Assumption \ref{assu_smmothness} will be satisfied if 
\begin{equation*}
 \sup_{t \in [0,1]} |a_j^{(d)}(t)|^{\min(2,q)}\leq  C j^{-\tau}, \ j \geq 1.
\end{equation*} 

A concrete example is the time-varying MA($q$) process. Since the trigonometric functions are $C^{\infty}$, it is easy to check that 
\begin{equation*}
\sum_{k=0}^q a_j(t) \epsilon_{i-k},  \ q>0 \ \text{is a fixed constant},
\end{equation*} 
with $a_{j}(t)=\alpha_j \cos(2 \pi t)$ or $\beta_j\sin(2 \pi t), \ |\alpha_j|<1, |\beta_j|<1$ satisfy such assumptions. 

\end{exam}

%

\begin{exam}[Nonstationary nonlinear process] Let $\{\epsilon_i\}$ be i.i.d random variables. We now consider a process of the following form 
\begin{equation}\label{eq_nonlinear}
\xi_i(t)=R(t, \xi_{i-1}(t), \epsilon_i), 
\end{equation}
where $R$ is some (possibly nonlinear) measurable function. This process has been studied in \cite[Section 4.2]{WZ1}. Suppose that for some $x_0,$ we have $\sup_{t \in [0,1]} \norm{R(t, x_0, \epsilon_i)}_q <\infty.$ Denote 
\begin{equation*}
\chi:=\sup_{t \in [0,1]} L(t), \ \text{where} \ L(t)=\sup_{x \neq y} \frac{\norm{R(t, x, \epsilon_0)-R(t, y, \epsilon_0)}_q}{|x-y|}
\end{equation*}
It is known from \cite[Theorem 6]{WZ1} that  if $\chi<1,$ then (\ref{eq_nonlinear}) admits a unique locally stationary solution with $\xi_i(t)=G(t, \mathcal{F}_i)$ and the physical dependence measure satisfies that 
$\delta(j,q) \leq  C \chi^j. $
Hence, the temporal dependence is of  exponentially decay (see equation (\ref{eq_phyexp})) which is much faster than (\ref{assu_phylip}). Furthermore,  we conclude from \cite[Proposition 4]{WZ1} that (\ref{assum_lip}) holds true if 
\begin{equation*}
\sup_{t \in [0,1]} \norm{M(G(t, \mathcal{F}_0))}_q <\infty, \  \text{where} \ M(x)=\sup_{0 \leq t<s \leq 1 } \frac{\norm{R(t,x,\epsilon_0)-R(s, x, \epsilon_0)}_q}{|t-s|}.
\end{equation*}
To verify Assumption \ref{assu_smmothness}, we assume that $G(t, \mathcal{F}_i)$ admits the following Volterra expansion \cite{WW} 
\begin{equation*}
G(t,\mathcal{F}_i)=\sum_{k=1}^{\infty} \sum_{u_1, \cdots, u_k=0}^{\infty} g_k(t,u_1,\cdots, u_k)\epsilon_{i-u_1} \cdots \epsilon_{i-u_k},
\end{equation*}
where $g_k's$ are the Volterra kernels. Suppose $g_k \in C^d[0,1]$ for $t$ and $$\sup_{t \in [0,1]} \sum_{k=1}^{\infty} \sum_{u_1, \cdots, u_k=0}^{\infty} (g^{(d)}_k(t,u_1,\cdots, u_k))^2<\infty,$$
by the rule of term by term differentiation, we can easily see that Assumption \ref{assu_smmothness} holds. 

A concrete example is the time-varying threshold autoregressive (TVTAR) model (see \cite[Example 1]{WZ1}) where (\ref{eq_nonlinear}) has the following form
\begin{equation*}
\xi_i(t)=a(t)[\xi_{i-1}(t)]^++b(t)[-\xi_{i-1}(t)]^++\epsilon_i.
\end{equation*}
We can see that Assumption \ref{assu_phylip} and \ref{assu_smmothness} are satisfied if $a(t), b(t) \in C^d[0,1]$ and $\sup_{t \in [0,1]} [|a(t)|+|b(t)|]<1.$

\end{exam}

\subsection{Theoretical properties of locally stationary properties} Many important consequences can be derived due to  Assumption \ref{assu_phylip} and \ref{assu_smmothness}. We list the most useful ones in this section and put their proofs into Appendix \ref{appendix_a}.  The first one is the following control on $\gamma(t,j).$ 
\begin{lem}\label{lem_phy} Under  Assumption \ref{assu_phylip} and \ref{assu_smmothness}, there exists some constant $C>0,$ such that  
\begin{equation*}
\sup_{t} |\gamma(t,j)| \leq Cj^{-\tau}, \ j \geq 1. 
\end{equation*}
\end{lem}

Our first important conclusion is that the coefficients defined in (\ref{defn_bestlinear}) is of polynomial decay. Hence, when $i>b$ is large, where $b=O(n^{2/\tau}),$ we only need to focus on autoregressive fit of order $b$ instead of $i-1.$ Recall (\ref{defn_bestlinear}), denote $\bm{\phi}_i=(\phi_{i1}, \cdots, \phi_{i,i-1})^*$. Then we have 
\begin{equation}\label{defn_bestcoeff}
\bm{\phi}_i=\Omega_i \bm{\gamma}_i,
\end{equation}
where $\Omega_i$ and $\bm{\gamma_i}$ are defined as $\Omega_i=[\operatorname{Cov}(\mathbf{x}^i_{i-1}, \mathbf{x}^i_{i-1})]^{-1}, \ \bm{\gamma}_i=\operatorname{Cov}(\mathbf{x}^i_{i-1}, x_i), $
with $\mathbf{x}^i_{i-1}=(x_{i-1},\cdots, x_1)^*.$  The above claims are formally summarized in the following proposition. 
\begin{prop} \label{lem_borderapproximation} Under Assumption \ref{assu_phylip} and letting $b=O(n^{2/\tau})$, there exists some constant $C>0,$ such that  
\begin{equation} \label{eq_phijbound}
\sup_{i>b}|\phi_{ij}| \leq
\begin{cases}
 \max \{ Cn^{-4+5/\tau}, Cj^{-\tau} \}, & i \geq b^2; \\
 \max\{ Cn^{-2+3/\tau}, Cj^{-\tau}\}, & b<i<b^2. 
 \end{cases}
\end{equation}
Furthermore, when $i>b,$ denote $\bm{\phi}_i^b=(\phi_{i1}, \cdots, \phi_{ib}), $ and $\widetilde{\bm{\phi}}_i^b=\Omega_i^b \gamma_i^b$ with entries $(\widetilde{\phi}_{i1}, \cdots,\widetilde{\phi}_{ib}),$ where $\Omega_i^b=[\operatorname{Cov}(\mathbf{x}_{i}, \mathbf{x}_{i})]^{-1}, \gamma_i^b=\mathbb{E}(\mathbf{x}_{i} x_i), \ \mathbf{x}_{i}=(x_{i-1}, \cdots, x_{i-b}),$ we have 
\begin{equation*}
\sup_i \left|\left|\bm{\phi}_i^b-\widetilde{\bm{\phi}}_i^b\right|\right| \leq Cn^{-2+1/\tau}.
\end{equation*}
\end{prop}

To our knowledge, Proposition \ref{lem_borderapproximation} is the first result on the decay rate of best linear forecasting under nonstationarity. It serves the first dimension reduction for our parameter space. It states that we can treat $\phi_{ij}=0$ for $j>b.$ Hence, the number of coefficients needed for the Cholesky decomposition reduces from $O(n^2)$ to $O(nb).$ Finally, denote $\bm{\phi}^b(\frac{i}{n}):=(\phi_1(\frac{i}{n}), \cdots, \phi_b(\frac{i}{n}))$ by
\begin{equation} \label{eq_phigeqb}
\bm{\phi}^b(\frac{i}{n})=\widetilde{\Omega}^b_i \widetilde{\bm{\gamma}}^b_i,
\end{equation} 
where $\widetilde{\Omega}^b_i$ and $\widetilde{\bm{\gamma}}^b_i$ are defined as 
\begin{equation*}
\widetilde{\Omega}^b_i=[\operatorname{Cov}(\widetilde{\mathbf{x}}_{i},\widetilde{\mathbf{x}}_{i})]^{-1}, \ \widetilde{\bm{\gamma}}_i=\operatorname{Cov}(\widetilde{\mathbf{x}}_{i}, x_i),
\end{equation*}
with $\widetilde{\mathbf{x}}_{i,k}=G(\frac{i}{n},\mathcal{F}_{i-k}), \ k=1,2,\cdots,b.$  The following lemma shows that $\bm{\phi}^b_i$ can be well approximated by $\bm{\phi}^b(\frac{i}{n})$ when $i>b.$
\begin{lem}\label{lem_approxphi} Under Assumption \ref{assu_phylip}, there exists some constant $C>0,$ such that for all $j \leq b,$ 
\begin{equation*}
  \sup_{i>b} \left|\phi_{ij}-\phi_{j}(\frac{i}{n}) \right| \leq   Cn^{-1+2/\tau}.
\end{equation*} 
\end{lem} 

Lemma \ref{lem_approxphi} claims that each off-diagonal $\{\phi_{ij}\}_{i=b}^n$ can be well-approximated by a smooth function $\phi_j(\cdot), $ it provides the second dimension reduction. Due to the smoothness of $\phi_j(\cdot),$ it can be well approximated by a sieve expansion of order $c,$ where $c \ll n$. This will further reduce the dimension of the parameter space  from $O(nb)$ to $O(bc)$. Throughout of the rest of the paper, unless otherwise specified, we will always use $b=O(n^{2/\tau}).$Recall that $\epsilon_i$ is the prediction error with variance $\sigma_i^2,$ 
\begin{equation} \label{defn_epsilon}
\epsilon_i=x_i-\widehat{x}_i.
\end{equation} 
We define $\widetilde{\epsilon}_i:=x_i-\sum_{j=1}^{\min(b,i-1)} \phi_{ij} x_{i-j}.$ First of all, we deduce from Proposition \ref{lem_borderapproximation} and Assumption  \ref{assu_phylip} that 
\begin{equation}\label{eq_reduceepsilon}
\max_{1 \leq i \leq n}|\epsilon_i-\widetilde{\epsilon}_i|=o(n^{-3}) \ \text{in probability}. 
\end{equation} 

 Next we summarize the basic properties of $\widetilde{\epsilon}_i$. Denote $\widetilde{\sigma}_i^2$ as the variance of $\{\widetilde{\epsilon}_i\}.$ 
\begin{lem}\label{lem_epsilon} 
We have $\sup_i \widetilde{\sigma}_i^2<\infty.$  Furthermore, denote the physical dependence measure of $\{\widetilde{\epsilon}_i\}$ as $\delta^{\epsilon}(j,q),$ then there exists some constant $C>0,$ such that for $\delta^{\epsilon}(j,q) \leq Cj^{-\tau}.$
\end{lem}

\begin{rem} In this paper, we focus on the discussion when the physical dependence measure is of polynomial decay, i.e. (\ref{assum_phy}) holds true. However, all our results can be extended to the case when the short-range dependence is of exponential decay, i.e. 
\begin{equation}\label{eq_phyexp}
\delta(j,q) \leq C a^{j}, \ 0<a<1.
\end{equation}  
In detail, Lemma \ref{lem_phy} can be changed to 
\begin{equation*}
\sup_t |\gamma(t,j)|<Ca^{j}, \ j \geq 1.
\end{equation*}
Therefore, we only need to choose $b=O(\log n).$  As a consequence, Proposition \ref{lem_borderapproximation} can be updated to 
\begin{equation*} 
\sup_{i>b}|\phi_{ij}| \leq \max\{Cn^{-C}, \ Ca^{j} \}, \
\sup_i \left|\left|\bm{\phi}_i^b-\widetilde{\bm{\phi}}_i^b\right|\right| \leq Cn^{-C},
\end{equation*} 
where $C>1$ is some constant. Similarly, Lemma \ref{lem_approxphi} can be modified to 
\begin{equation*}
  \sup_{i>b} \left|\phi_{ij}-\phi_{j}(\frac{i}{n}) \right| \leq \frac{C\log n}{n} ,  \ \text{for all} \ j \leq b. 
\end{equation*} 
Finally, the analog of Lemma \ref{lem_epsilon} is  $\delta^{\epsilon}(j,q) \leq C \max\{ a^{j}, n^{-C}\}.$ 
\end{rem}

\section{Estimation of precision matrices}\label{sec_est} As shown in (\ref{estimationeq}), Proposition \ref{lem_borderapproximation} and Lemma  \ref{lem_approxphi}, in order to estimate $\Omega,$ it suffices to estimate $\phi_{ij}, i \leq b$, $\phi_j(\frac{i}{n}), i>b \geq j$ and the 
 variances of the residuals. When $i>b,$ by (\ref{assum_moment}) and Proposition \ref{lem_borderapproximation}, it is easy to see that 
 \begin{equation*}
 \sup_{i}\left|\sum_{j=b+1}^{i-1} \phi_{ij}x_{i-j} \right|=o(n^{-1}) \ \text{in probability}. 
 \end{equation*}
Therefore, we now simply write 
\begin{equation}\label{defn_cho_borderapp}
x_i=\sum_{j=1}^{b} \phi_{ij}x_{i-j}+\epsilon_i+o_{u}(n^{-1}), \ i=b+1, \cdots, n, 
\end{equation}
where $X_i=o_{u}(n^{-1})$ means $nX_i$ converges to zero in probability uniformly for $i>b.$

\subsection{Estimating $\phi_{ij}$ for  $i > b$}\label{sec_3sub1} We first estimate the time-varying coefficients $\phi_j(\frac{i}{n})$ using the method of sieves \cite{BCCK,CXH,CC} when $i > b$. 
 We first observe the following result, whose proof will be put into Appendix \ref{appendix_a}.
\begin{lem} \label{lem_smoothphijt}
Under Assumption \ref{assu_phylip} and \ref{assu_smmothness}, for any $j\leq b$, we have that $\phi_j(t) \in C^d([0,1]).$ 
\end{lem}

 Based on the above lemma, we use
\begin{equation}\label{sieve_eq}
\theta_j(\frac{i}{n}):=\sum_{k=1}^c a_{jk} \alpha_k(\frac{i}{n}), \ j \leq b,
\end{equation} 
to approximate $\phi_j(\frac{i}{n}),$ where $\{\alpha_k(t)\}$ is a set of pre-chosen orthogonal basis functions on $[0,1]$ and $c \equiv c(n)$ stands for the number of basis functions.  In the present paper, unless otherwise specified,  we always set $c=O(n^{\alpha_1}).$  The estimate of $\theta_j(t)$ boils down to the estimation of the $a_{jk}'s.$ Next, the results of the convergent rate on the approximation (\ref{sieve_eq}) can be found in \cite[Section 2.3]{CXH} for the commonly used basis functions, where we summarize it in the following lemma.  
\begin{lem} \label{lem_sievebasisbound}
Denote the sup-norm with respect to Lebesgue measure as 
\begin{equation*}
 \mathcal{L}_{\infty}:=\sup_{t \in [0,1]}\left| \phi_j(t)-\theta_j(t) \right|. 
\end{equation*}
We then have that $\mathcal{L}_{\infty}=O(c^{-d})$ for the orthogonal polynomials, trigonometric polynomials, spline series with order $r$ when $r \geq d+1,$ and orthogonal wavelets with degree $m$ when $m>d. $
\end{lem}

Then we impose the following regularity condition on  the basis functions. 
\begin{assu} \label{assu_basisregularity}
Let $\otimes$ be the Kronecker product.
For any $k=1,2,\cdots, b, $ denote $\Sigma^k(t) \in \mathbb{R}^{k \times k}$ via $\Sigma^k_{ij}(t)=\gamma(t, |i-j|),$ we assume that the eigenvalues of 
$$ \int_0^1 \Sigma^k(t) \otimes \left( \mathbf{b}(t) \mathbf{b}^*(t) \right)dt,$$
are bounded above and also away from zero by a constant $\kappa>0$,  
where $\mathbf{b}(t)=(\alpha_1(t),\cdots, \alpha_c(t))^* \in \mathbb{R}^c.$ 
\end{assu} 
Since $\Sigma^k(t) \otimes (\mathbf{b}(t) \mathbf{b}^*(t))$ is positive semidefinite for any $t \in [0,1],$ Assumption \ref{assu_basisregularity} is mild.  
 We next provide some comments on how to check Assumption \ref{assu_basisregularity}. It is clear that when $x_i$ is a stationary process, the assumption will hold immediately due to the orthonormality of the basis functions. We next consider locally stationary MA($q$) process of the form  
 \begin{equation} \label{eqmaq}
G(t, \mathcal{F}_i)=\sum_{j=1}^q a_j(t) \epsilon_{i-j}+\epsilon_i, \   1 \leq q \leq \infty,
\end{equation}
where $\epsilon_i$ are i.i.d centered random variables with variance $1.$ The following lemma shows that under suitable conditions, Assumption \ref{assu_basisregularity} holds true for (\ref{eqmaq}). We leave its proof to Appendix \ref{appendix_a}.
\begin{lem}\label{lem_checkinginvert} Suppose that  the assumptions of Examples \ref{exam_linear} hold for (\ref{eqmaq}) and 
\begin{equation}\label{eq_mainvert}
\sup_t \sum_{i=1}^q |a_j(t)|<1.
\end{equation}
Then Assumption \ref{assu_basisregularity} holds  for (\ref{eqmaq}) and any orthonormal basis functions. 
\end{lem} 
 Finally, since locally stationary process has an locally MA(q) approximation, we are able to check the above assumption by studying its MA approximation.

Next we impose the following mild assumption on the parameters. 

\begin{assu}\label{assu_parameter}
We assume that for $\tau$ defined in (\ref{assum_phy}), $d$ defined in Assumption \ref{assu_smmothness} and $\alpha_1$, there exists a constant $C>4,$ such that
\begin{equation*}
\frac{C}{\tau}+d\alpha_1<1.
\end{equation*}
\end{assu} 

Note that the above assumption can be easily satisfied by choosing $C<\tau$ and $\alpha_1$ accordingly. In the case when the physical dependence is of exponentially decay, we only need $d \alpha_1<1.$

We now estimate $\phi_{ij}.$ Under Assumption \ref{assu_parameter}, by  (\ref{defn_cho_borderapp}), (\ref{sieve_eq}) and Lemma \ref{lem_sievebasisbound}, we can now write
\begin{equation}\label{defn_forest}
x_i=\sum_{j=1}^b \sum_{k=1}^c a_{jk} z_{kj}(\frac{i}{n})+\epsilon_i+o_{u}(n^{-1}), \ i=b+1,\cdots, n, 
\end{equation}
where $z_{kj}(\frac{i}{n}):=\alpha_k(\frac{i}{n})x_{i-j}.$ 
In view of (\ref{defn_forest}), we can use the ordinary least square (OLS) method to estimate the coefficients $a_{jk}.$ Denote the vector $\bm{\beta} \in \mathbb{R}^{bc}$ with $\bm{\beta}_s=a_{j_s,k_s},$ where $j_s=\lfloor \frac{s}{c} \rfloor+1, \ k_s=s-\lfloor \frac{s}{c}\rfloor \times c.$ Similarly, we define $\mathbf{y}_i \in \mathbb{R}^{bc}$ by letting $\mathbf{y}_{is}=z_{k_s,j_s}(\frac{i}{n}).$ Furthermore, we denote $Y^*$ as the $bc \times (n-b)$ design matrix of (\ref{defn_forest}) whose columns are $\mathbf{y}_i, \ i=b+1, \cdots,n.$ We also define by $\mathbf{x} \in \mathbb{R}^{n-b}$ the vector of $x_{b+1}, \cdots, x_n.$ Hence, the OLS estimator for $\bm{\beta}$ can be written as 
\begin{equation*}
\widehat{\bm{\beta}}=(Y^*Y)^{-1} Y^* \mathbf{x}.
\end{equation*}

Moreover, recall $\mathbf{x}_i=(x_{i-1}, \cdots, x_{i-b})^* \in \mathbb{R}^b,$ denote $X=(\mathbf{x}_{b+1}, \cdots, \mathbf{x}_n) \in \mathbb{R}^{b \times (n-b)}$ and the matrices $E_i \in \mathbb{R}^{(n-b) \times (n-b)}$ such that $(E_i)_{st}=1,$ when $s=t=i-b$ and $(E_i)_{st}=0$ otherwise. As a consequence, we can write 
\begin{equation}\label{y_kronecker}
Y^*=\sum_{i=b+1}^n \left (X \otimes \mathbf{b}(\frac{i}{n}) \right) E_i,
\end{equation} 
Observe that 
\begin{equation}\label{beta_est}
\widehat{\bm{\beta}}=\bm{\beta}+\left(\frac{Y^*Y}{n} \right)^{-1}\frac{Y^* \bm{\epsilon}}{n}+o_{\mathbb{P}}(n^{-1}),
\end{equation}
where $\bm{\epsilon} \in \mathbb{R}^{n-b}$ consists of $\epsilon_{b+1}, \cdots, \epsilon_n$ and the error is entrywise. We decompose $\bm{\beta}$ into $b$ blocks by denoting $\bm{\beta}=(\bm{\beta}^*_1, \cdots, \bm{\beta}_b^*)^*,$ where each $\bm{\beta}_i \in \mathbb{R}^c.$ Similarly, we can decompose $\widehat{\bm{\beta}}.$ Therefore, our sieve estimator can be written as $\widehat{\phi}_j(\frac{i}{n})=\bm{\widehat{\beta}}_j^* \mathbf{b}(\frac{i}{n})$ and it satisfies that 
\begin{equation}\label{phi_diff}
\phi_j(\frac{i}{n})-\widehat{\phi}_j(\frac{i}{n})=(\bm{\beta}_j-\widehat{\bm{\beta}}_j)^*\mathbf{b}(\frac{i}{n}).
\end{equation}

We impose the following assumption on the derivative of the basis functions, which is also used in \cite[Assumption 4]{CC}. 
\begin{assu} \label{assu_derivativebasis} There exist $\omega_1, \omega_2 \geq 0,$ we have 
\begin{equation*}
\sup_t || \nabla \mathbf{b}(t) || \leq C n^{\omega_1} c^{\omega_2}, \ C>0 \ \text{is some constant}. 
\end{equation*}
\end{assu}
The above assumption is a mild regularity condition on the sieve basis functions and is satisfied by many of the widely used basis functions. For instance, we can choose $\omega_1=0, \ \omega_2=\frac{1}{2}$ for trigonometric polynomials, spline series, orthogonal wavelets and weighted Chebyshev polynomials. For more examples of basis functions satisfying this assumption, we refer to \cite[Section 2.1]{CC}. Finally, we impose the following mild assumption on the parameters.

\begin{thm} \label{thm_timevaryingcoeff} Under Assumption \ref{assu_phylip}, \ref{assu_smmothness}, \ref{assu_basisregularity}, \ref{assu_parameter} and \ref{assu_derivativebasis}, we have
\begin{equation*}
\sup_{i>b,j \leq b} \left| \phi_j(\frac{i}{n})-\widehat{\phi}_j(\frac{i}{n}) \right|=O_{\mathbb{P}}\left(\zeta_c \sqrt{\frac{\log n}{n}}+n^{-d\alpha_1}\right). 
\end{equation*}
\end{thm}
By carefully choosing $c=O(n^{\alpha_1}),$ we show that $\widehat{\phi}_j(\frac{i}{n})$ are consistent estimators for $\phi_j(\frac{i}{n})$ uniformly in $i$ for all $j \leq b$ in Theorem \ref{thm_timevaryingcoeff}.  Denote $\zeta_c:=\sup_i || \mathbf{b}(\frac{i}{n})||,$ as discussed in \cite[Section 3]{BCCK}, we can write
$\zeta_c=n^{\alpha_1^*},$
where $\alpha_1^*=\frac{1}{2} \alpha_1$ for trigonometric polynomials, spline series, orthogonal wavelets and weighted orthogonal Chebyshev polynomials. And $\alpha^*_1=\alpha_1$ for Legendre orthogonal polynomials.
Furthermore, by using basis functions with $\alpha_1^*=\frac{1}{2} \alpha_1,$ we can show that our estimators  attain the optimal minimax convergent rate $\left(n/(\log n) \right)^{-d/(2d+1)}$ for the nonparametric regression proposed by Stone in \cite{SC}.
\begin{cor}\label{cor_minimax}
Under Assumption \ref{assu_phylip}, \ref{assu_smmothness}, \ref{assu_basisregularity}, \ref{assu_parameter} and \ref{assu_derivativebasis}, using the trigonometric polynomials, spline series, orthogonal wavelets and weighted orthogonal Chebyshev polynomials, when $c=O((n/(\log n))^{1/2d+1})$, we have  
\begin{equation*}
\sup_{i>b,j \leq b} \left| \phi_j(\frac{i}{n})-\widehat{\phi}_j(\frac{i}{n}) \right|=O_{\mathbb{P}}\left(\left(n/(\log n) \right)^{-d/(2d+1)}\right). 
\end{equation*}
\end{cor}



\subsection{Estimating $\phi_{ij}$ for $i \leq b$}\label{sec3_sub2} It is notable that by Lemma \ref{lem_approxphi}, when $i,j$ are less or equal to $b$, we cannot use the estimators derived from Section \ref{sec_3sub1}. Instead, a different series of least squares linear regressions should be used.  For instance, in order to estimate $\phi_{21},$ we use the following regression equations
\begin{equation*}
x_{k}=\phi_{k1}x_{k-1}+\xi_{k,2}, \ k=2,3,\cdots,n.
\end{equation*}
Note that $\xi_{2,2}=\epsilon_2.$ Due to the local stationarity assumption, there exists a smooth function $f_{21},$ such that $\phi_{k1} \approx f_{21}(\frac{k}{n}),$ $k=2,3,\cdots,n$. Here $f_{21}$ can be efficiently estimated using the sieve method as described by the previous discussions and $\phi_{21}$ can be estimated by $\widehat{f}_{21}(2/n)$. Generally, for each fixed $i \leq b,$ to estimate $\bm{\phi}_i,$ we make use of the following predictions:
\begin{equation}\label{defn_noisexi}
x_k=\sum_{j=1}^{i-1} \lambda_{ij}^k x_{k-j}+\xi_{k,i}, \ k=i,i+1, \cdots, n,
\end{equation} 
where $\bm{\lambda}^k_i=(\lambda_{i1}^k, \cdots, \lambda_{i,i-1}^k)$ are the coefficients of the best linear prediction using the $i-1$ predecessors. Note that $\bm{\lambda}_i^i=\bm{\phi}_i.$ Using Yule-Walker equation, we find 
\begin{equation*}
\bm{\lambda}_i^k=\Omega_i^k \bm{\gamma}_i^k,
\end{equation*}
where $\Omega_i^{k}=[\operatorname{Cov}(\mathbf{x}^k_{i}, \mathbf{x}^k_{i})]^{-1}, \bm{\gamma}_i^k=\operatorname{Cov}(\mathbf{x}_i^k,x_k)$ and $\mathbf{x}^k_{i}=(x_{k-1},\cdots, x_{k-i+1}).$ Due to Assumption \ref{assu_smmothness}, we define $\mathbf{f}_i^k=(f_{1}^i(\frac{k}{n}), \cdots, f^{i}_{i-1}(\frac{k}{n}))$ by 
\begin{equation}\label{defn_illeqbfun}
\mathbf{f}_i^k=\widetilde{\Omega}_i^k \widetilde{\gamma}_i^k,
\end{equation}
with $\widetilde{\Omega}_i^k, \bm{\widetilde{\gamma}}_i^k$ 
\begin{equation*}
\widetilde{\Omega}_i^k=[\operatorname{Cov}(\widetilde{\mathbf{x}}_i^k,\widetilde{\mathbf{x}}_i^k)]^{-1}, \ \widetilde{\bm{\gamma}}_i^k=\operatorname{Cov}(\widetilde{\mathbf{x}}_i^k, \widetilde{x}_k),
\end{equation*}
where $\widetilde{\mathbf{x}}_{i,j}^k=G(\frac{k}{n}, \mathcal{F}_{i-j}).$ The following lemma shows that $\lambda_{i,j}^k$ can be well-approximated by a smooth function $f_j^i(t).$
\begin{lem}\label{lem_coeffsmalli} Under Assumption \ref{assu_phylip} and \ref{assu_smmothness}, for each fixed $i \leq b$ and  for any $j \leq i-1,$ $f^i_{j}(t)$ are $C^d$ functions on $[0,1].$ Furthermore, for some constant $C>0,$ we have 
\begin{equation*}
\sup_{k \geq i} \left| \lambda_{ij}^k-f_j^i(\frac{k}{n}) \right| \leq C\Big( n^{-1+2/\tau} +n^{-d\alpha_1} \Big ).
\end{equation*}
In particular, when $k=i,$ we have
\begin{equation}\label{eq_phiapproileqb}
\left|\phi_{ij}-f^i_{j}(\frac{i}{n})\right| \leq C\Big( n^{-1+2/\tau} +n^{-d\alpha_1} \Big ), \ j < i \leq b.
\end{equation}
\end{lem}

Therefore, the rest of the work leaves to estimate the functions $f^i_{j}(t), \ j<i \leq b$ using sieve approximation by denoting
\begin{equation*}
f^i_{j}(t)=\sum_{k=1}^c d_{jk} \alpha_k(t)+O(c^{-d}),
\end{equation*}
where we recall Lemma \ref{lem_sievebasisbound}. Then the above sieve expansion is plugged into (\ref{defn_noisexi}). An OLS regression is then used to estimate the $d_{jk}'s.$ 
We denote the OLS estimator of $f_j^i(\frac{i}{n})$ as $\widehat{f}^i_j(\frac{i}{n})=\sum_{k=1}^c \widehat{d}_{jk} \alpha_k(\frac{i}{n}).$ 
\begin{thm}\label{thm_phiileqb}  
Under Assumption \ref{assu_phylip}, \ref{assu_smmothness}, \ref{assu_basisregularity}, \ref{assu_parameter} and \ref{assu_derivativebasis},  we have 
\begin{equation*}
\sup_{i \leq b, j<i} \left| f^i_j(\frac{i}{n})-\widehat{f}^i_j(\frac{i}{n}) \right|=O_{\mathbb{P}} \left( \zeta_c \sqrt{\frac{\log n}{n}}+n^{-d\alpha_1} \right).
\end{equation*}
\end{thm}

Similar to the discussion of Corollary \ref{cor_minimax},  using the  trigonometric polynomials, spline series, orthogonal wavelets and weighted orthogonal Chebyshev polynomials and setting $c=O((n/(\log n))^{1/2d+1}),$ we obtain the optimal minimax convergent rate from Theorem \ref{thm_phiileqb}.

\subsection{Sieve estimation for noise variances}\label{sec3_sub3} This subsection is devoted to the estimation of $\{\sigma_i^2\}_{i=1}^n.$ We discuss the case for $i>b$ and $i \leq b$ separately.  For $i >b,$ denote $\epsilon_i^b=x_i-\sum_{j=1}^b \phi_{ij} x_{i-j}$ and $(\sigma_i^b)^2=\mathbb{E}(\epsilon_i^b)^2.$ $\sigma_i$ can be well approximated using $\sigma_i^b$ by the following lemma, whose proof will be put into Appendix \ref{appendix_a}. 
\begin{lem} \label{lem_sigmaigeqb} Under Assumption \ref{assu_phylip} and \ref{assu_smmothness}, for $i>b$ and some constant $C>0,$ we have 
\begin{equation*}
\sup_{i>b} \left|\sigma^2_i-(\sigma^b_i)^2\right| \leq Cn^{-2 +2/\tau}.
\end{equation*}
Furthermore, denote $g(\frac{i}{n})=\mathbb{E} \left( x_i-\sum_{j=1}^b \phi_{ij} G(\frac{i}{n}, \mathcal{F}_{i-j}) \right)^2,$ we then have 
\begin{equation*}
\sup_{i>b} \left| (\sigma_i^b)^2-g(\frac{i}{n}) \right| \leq Cn^{-1+4/\tau}.
\end{equation*}
Finally, $g(\frac{i}{n}) \in C^d([0,1]).$ 
\end{lem}

Lemma \ref{lem_sigmaigeqb} indicates that $\{\sigma_i^2\}_{i \geq b}$ can be well approximated by a $C^d$ function $g(\cdot).$
Denote $r_i^b=(\epsilon_i^b)^2,$ it is notable that $r_i^b$ can not be observed directly.  Instead, we use $\widehat{r}_i^b=\widehat{\epsilon}_i^2$, where 
\begin{equation}\label{eq_hatepsilon}
  \widehat{\epsilon}_i=x_i-\sum_{j=1}^{b} \sum_{k=1}^c \widehat{a}_{jk} z_{kj}(\frac{i}{n}), \ i=b+1, \cdots,n.  
\end{equation}
By Theorem \ref{thm_timevaryingcoeff} and Assumption \ref{assu_parameter}, we conclude that
\begin{equation}\label{eq_rapprox}
\sup_{i>b}|r_i^b-\widehat{r}_i^b|=O_{\mathbb{P}}\Big( n^{2/\tau}\Big(\zeta_c \sqrt{\frac{\log n}{n}}+n^{-d\alpha_1} \Big) \Big). 
\end{equation}
Invoking Lemma \ref{lem_sievebasisbound} and Assumption \ref{assu_parameter}, for $i>b,$ we can therefore utilize the method of sieves and write
\begin{equation}\label{eq_variancesieveigeqb}
\widehat{r}_i^b=\sum_{k=1}^{c} d_k \alpha_k(\frac{i}{n})+\omega_i^b+O_{\mathbb{P}}\Big(n^{2/\tau}\Big(\zeta_c \sqrt{\frac{\log n}{n}}+n^{-d\alpha_1} \Big) \Big).
\end{equation}
The coefficients $d_k's$ are then estimated via OLS. Similar to Lemma \ref{lem_epsilon}, we can show that the physical dependence measure of $\omega_i^b$ is also of polynomial decay. Therefore, the OLS estimator for $\bm{\alpha}=(d_1,\cdots, d_c)^*$ can be written as $\widehat{\bm{\alpha}}=(W^*W)^{-1}W^* \widehat{\mathbf{r}}, $
where $W^*$ is an $c \times (n-b)$ matrix whose $i$-th column is $(\alpha_1(\frac{i+b}{n}), \cdots, \alpha_c(\frac{i+b}{n}))^*$, $\ i=1,2,\cdots, n-b,$ and $\mathbf{\widehat{r}}$ is an $\mathbb{R}^{n-b}$ containing $\widehat{r}^b_{b+1}, \cdots, \widehat{r}^b_{n}.$ 
We have the following consistency result. 
\begin{thm}\label{thm_varaincefunction}
Under Assumption \ref{assu_phylip}, \ref{assu_smmothness}, \ref{assu_basisregularity}, \ref{assu_parameter} and \ref{assu_derivativebasis}, we have 
\begin{equation}\label{eq_variance1}
\sup_{i>b} \left| \widehat{g}(\frac{i}{n})-g(\frac{i}{n})\right| =O_{\mathbb{P}}\Big( n^{2/\tau}\Big(\zeta_c \sqrt{\frac{\log n}{n}}+n^{-d\alpha_1} \Big) \Big).
\end{equation}  
\end{thm}

Finally, we study the estimation of $\sigma_i^2, \ i=1,2,\cdots, b,$ which enjoys the same discussion as in Section \ref{sec3_sub2}. Recall  $\xi_{k,i}$ defined in (\ref{defn_noisexi}), denote $(\sigma_{k,i}(\xi))^2=\mathbb{E}(\xi_{k,i})^2,$ using a similar discussion to Lemma \ref{lem_sigmaigeqb}, we can find a smooth function $g^i,$ such that $\sup_k \sup_{i \leq b} |(\sigma_{k,i}(\xi))^2-g^i(\frac{k}{n})| \leq O(n^{-1+4\tau}),$ especially we can use $g^i(\frac{i}{n})$ to estimate $\sigma_i^2.$  When $i=1, $ we need to estimate the variance function of $x_1$. 

The rest of the work leaves to estimate $g^i(t)$ using sieve method similar to (\ref{eq_variancesieveigeqb}) for $i \leq b$, where we replace the errors with $\widehat{r}_k^i, \ k=i,\cdots,n. $ Here $\widehat{r}^i_k$ is defined as
\begin{equation}\label{eq_residualileqb}
\widehat{r}^i_k:=\left( x_i-\sum_{j=1}^{i-1} \widehat{f}^i_j(\frac{k}{n})x_{i-j} \right)^2, \ k=i, i+1, \cdots, n. 
\end{equation}
Then for $i \leq b,$ we can estimate $\widehat{g}^{i}(\frac{i}{n})$ using the method of sieves similarly, except that the dimension of  $W^*$ is  $c \times (n+1-i).$ The results are summarized in the following theorem. 

\begin{thm} \label{thm_variancefunctionileqb}
Under Assumption \ref{assu_phylip}, \ref{assu_smmothness}, \ref{assu_basisregularity}, \ref{assu_parameter} and \ref{assu_derivativebasis}, we have 
\begin{equation}\label{eq_varaince2}
\sup_{i \leq b} \left| \widehat{g}^{i}(\frac{i}{n})-g^{i}(\frac{i}{n})\right|=O_{\mathbb{P}}\Big( n^{2/\tau}\Big(\zeta_c \sqrt{\frac{\log n}{n}}+n^{-d\alpha_1} \Big) \Big).
\end{equation}  
\end{thm}

 In the finite sample case, for positiveness,  we suggest simply choose 
\begin{equation}\label{eq_positive}
\widehat{\sigma}_i^*=
\begin{cases}
\widehat{\sigma}_i, & \ \text{if} \ \widehat{\sigma}_i>0; \\
\frac{1}{n}, & \ \text{if} \ \widehat{\sigma}_i \leq 0.
\end{cases}
\end{equation}
where $\widehat{\sigma}_i=\widehat{g}(\frac{i}{n})$ for $i>b$ and $\widehat{\sigma}_i=\widehat{g}^i(\frac{i}{n})$ when $i \leq b.$
Since $n^{-1}$ is much smaller than the right-hand side of (\ref{eq_variance1}) and (\ref{eq_varaince2}), this modified estimator will not influence the results in Theorem \ref{thm_varaincefunction} and \ref{thm_variancefunctionileqb}.

\subsection{Precision matrix estimation}
From (\ref{estimationeq}), it is natural to choose 
\begin{equation*}
\widehat{\Omega}:=\widehat{\Phi}^{*} \mathbf{\widehat{\widetilde{D}}}\widehat{\Phi}
\end{equation*}
as our estimator for the precision matrix.
As we discussed in the previous sections, here $\widehat{\Phi}$ is a lower triangular matrix whose diagonal entries are all ones. For the off-diagonal entries, when $i>b$ and $j\le b$, its $(i,i-j)$-th entry is $-\widehat{\phi}_j(\frac{i}{n})$ defined in Section \ref{sec_3sub1}. And when $i \leq b,$ $\Phi_{i,i-j}$ is estimated using  $-\widehat{f}^{i}_j(\frac{i}{n})$ from Section \ref{sec3_sub2}. All other entries of $\widehat{\Phi}$ are set to be zeros. Finally, $\mathbf{\widehat{\widetilde{D}}}$ is a diagonal matrix with entries  $\{(\widehat{\sigma}_i^*)^{-2}\}$ estimated from (\ref{eq_positive}). Observe that $\widehat{\Omega}$ is always positive definite.

We now discuss the computational complexity of estimating $\Omega.$  It is easy to see that when $i>b,$ the number of  regressors is $bc$ and length of observation is $n-b$. Hence the computational complexity of the least squares regression is $O(n(bc)^2).$ Similar discussion can be applied for $i \leq b$ and we hence conclude that the computational complexity for estimating $\widehat{\Omega}$ is of the order $O(nb^3 c^2).$ As a result the computation complexity of our estimation is adaptive to the smoothness of the underlying data generating mechanism and the decay rate of temporal dependence. In the best scenario when assumption (\ref{eq_phyexp}) holds and $\gamma(t,j) \in C^{\infty}([0,1]),$ our procedure only requires $O(n\log^5 n)$ computation complexity. 

In the following, we shall control the estimation error between $\Omega$ and $\widehat{\Omega}.$ We first observe that, as $\det(\Phi \Phi^* )=\det(\widehat{\Phi}\widehat{\Phi}^*)=1,$ combining with Assumption \ref{assu_phylip}, there exist some constants $C_1, C_2>0,$ such that 
\begin{equation*}
C_1 \leq \lambda_{\min}(\Phi \Phi^*) \leq \lambda_{\max} (\Phi \Phi^*) \leq C_2. 
\end{equation*} 
Similar results hold for $\widehat{\Phi} \widehat{\Phi}^*.$ 

\begin{thm}\label{prop_covest} Under Assumption \ref{assu_phylip}, \ref{assu_smmothness}, \ref{assu_basisregularity}, \ref{assu_derivativebasis} and \ref{assu_parameter}, we have 
\begin{equation}\label{eq_errorr}
\left| \left| \Omega-\widehat{\Omega} \right| \right|=O_{\mathbb{P}} \Big( n^{4/\tau} \Big( n^{-d\alpha_1}+ \zeta_c \sqrt{\frac{\log n}{n}} \Big) \Big). 
\end{equation}
\end{thm}
Recall that $\left| \left| \cdot  \right| \right|$ denotes the operator norm of a matrix. It can be seen from the above theorem that  the estimation accuracy of precision matrices depends on the decay rate of dependence and the smoothness of the covariance functions. The estimation accuracy gets higher for time series with more smooth covariance functions and faster decay speed of dependence.

\begin{rem} Under assumption (\ref{eq_phyexp}),  when we apply Lemma \ref{lem_disc} for our proof, we only need $O(\log n)$ matrix entries to bound the error terms. Hence, we can change (\ref{eq_errorr}) to 
\begin{equation*}
\left| \left| \Omega-\widehat{\Omega} \right| \right|=O_{\mathbb{P}} \left( \log^2 n \Big(n^{-d\alpha_1}+\zeta_c \sqrt{\frac{\log n}{n}}  \Big) \right). 
\end{equation*}
\end{rem} 
In the best scenario where the dependence is exponentially decaying and $\phi_j(\cdot)$ and $g(\cdot)$ are infinitely differentiable, following the same arguments as those in the proof of Theorem \ref{prop_covest}, it is easy to show  the convergence rate of $\widehat{\Omega}$ is of the order $\log^3n/\sqrt{n}$, which is almost as fast as parametrically estimating a single parameter from i.i.d. samples.

%

\section{Testing the structure of the precision matrices}\label{sec_test}
An important advantage of our methodology is that we can test many structural assumptions of the precision matrices using some simple statistics in terms of the entries of $\widehat{\Phi}$.
\subsection{Test statistics}\label{sec_subtest} In this subsection, we focus on discussing two fundamental tests in non-stationary time series analysis.
One of those is to test whether the observed samples are from a non-stationary white noise process $\{x_i\}$ in the sense that $\text{Cov}(x_i, x_j)=\delta_{ij} \sigma_i^2,$ where $\delta_{ij}$ is the Dirac delta function such that $\delta_{ij}=1$ when  $i=j$ and $\delta_{ij}=0$ otherwise. 
Note that we allow heteroscedasticity by assuming that the variance of $x_i$ changes over time. Formally, we would like to test 
\begin{equation*}
\mathbf{H}_0^1: \{x_i\} \ \text{is a non-stationary white noise process}.
\end{equation*}
Under $\mathbf{H}_0^1,$ recall (\ref{eq_phigeqb}), we shall have that $\phi_j(\frac{i}{n})$ are all zeros. Therefore, our estimation $\widehat{\phi}_j(\frac{i}{n})$ should be small enough for all pairs $i,j$, $i\neq j$. We hence use the following statistic:
\begin{equation}\label{eq_t1oridef}
T^*_1=\sum_{j=1}^{b} \int_0^1 \widehat{\phi}^2_j(t)dt.
\end{equation}

The second hypothesis of interest is whether the precision matrices are banded. In our setup,  the Cholesky decomposition provides a convenient way to test the bandedness.  Formally, for any $k_0 \equiv k_0(n) < b,$ we are interested in testing the following hypothesis:
\begin{equation*}
\mathbf{H}_0^2: \text{The precision matrix of}  \ \{x_i\} \ \text{is $k_0$-banded}. 
\end{equation*}
Due to (\ref{estimationeq}), as $\Omega$ is strictly positive definite, the Cholesky decomposition is unique. Therefore, we conclude that $\Phi$ is also $k_0$-banded using the discussion in \cite[Section 2]{RH}. Furthermore, under $\mathbf{H}_0^2,$   we have that $\phi_j(\frac{i}{n})=0,$ for $j > k_0.$  Therefore, it is natural for us to use the following  statistic
\begin{equation*}
T^*_2=\sum_{j=k_0+1}^b \int_{0}^1 \widehat{\phi}_j^2(t)dt.
\end{equation*}

It is notable that both of the test statistics $T^*_1$ and $T^*_2$ can be  written into summations of quadratic forms under the null hypothesis. For instance,  for $T^*_1$ under $\mathbf{H}_0^1,$ we have that 
\begin{equation*}
\widehat{\phi}^2_j(t)=\left( \widehat{\phi}_j(t)-\phi_j(t) \right)^2.
\end{equation*}
 For any fixed $ j \leq b$, we have 
\begin{equation*}
 \int_0^1 \left( \phi_j(t)-\widehat{\phi}_j(t) \right)^2dt =\sum_{k=1}^c (\widehat{a}_{jk}-a_{jk})^2+O(n^{-d\alpha_1}).
\end{equation*}

It can be seen from the above equation that the order of smoothness and number of basis functions are important to our analysis. Under Assumption \ref{assu_parameter},  we can see that the error $O(n^{-d\alpha_1})$ is negligible. 
Then recall (\ref{beta_est}), it is easy to see that for $\Sigma$ defined in (\ref{eq_sigma}),  we have that
\begin{equation} \label{squarecon}
\sum_{j=1}^b \int_0^1 \left( \phi_j(t)-\widehat{\phi}_j(t) \right)^2 dt=\frac{\bm{\epsilon}^* Y}{n}  \Sigma^{-1} \sum_{j=1}^b A_j^*A_j \Sigma^{-1} \frac{Y^* \bm{\epsilon}}{n}+o_{\mathbb{P}}(1),
\end{equation} 
where $A_j \in \mathbb{R}^{bc}$ is a diagonal block matrix whose $j$-th diagonal block being the identity matrix and zeros otherwise. Therefore, the investigation of $T_1^*$ boils down to the analysis of quadratic forms of a $bc$  dimensional locally stationary time series $\{Y^*\bm{\epsilon}\}.$ 
%

\subsection{Diverging dimensional Gaussian approximation}\label{sec_gaussian}
As we have seen from the previous subsection, both test statistics are involved with high dimensional quadratic forms. Observe that the distribution of quadratic forms of Gaussian vectors can be derived using Lindeberg's central limit theorem. Hence our case can be tackled if  we could establish a Gaussian approximation of the quadratic form (\ref{squarecon}) of general non-stationary time series. In this subsection, we will prove a Gaussian approximation result for the quadratic form $\mathbf{Z}^*E \mathbf{Z}, $ where $\mathbf{Z}:=\frac{\bm{\epsilon}^* Y}{\sqrt{n}} \in \mathbb{R}^{bc}$ and $E$ is a bounded positive semi-definite matrix. Denote $p=bc$ and $\mathbf{z}_i=(z_{i1}, \cdots, z_{ip})^*,$ where 
\begin{equation} \label{eq_zis}
z_{is}=x_{i-\bar{s}-1}\epsilon_{i} \alpha_{s^{\prime}}(\frac{i}{n}), \ \bar{s}=\lfloor \frac{s}{c}\rfloor , \  s^{\prime}=s- \bar{s} c, \ i \geq b+1.
\end{equation}
As a consequence, we can write 
$\mathbf{Z}:=(\mathbf{Z}_1, \cdots, \mathbf{Z}_p)=\frac{1}{\sqrt{n}} \sum_{i=b+1}^n \mathbf{z}_i. $
Denote $\mathbf{U}=\frac{1}{\sqrt{n}}\sum_{i=b+1}^n \mathbf{u}_i,$ where $\{\mathbf{u}_{i}\}_{i=b+1}^n$ are centered Gaussian random vectors independent of $\{\mathbf{z}_{i}\}_{i=b+1}^n$ and preserve their covariance structure.
Our task is to control the following Kolmogorov distance  
\begin{equation} \label{eq_rhoij}
\rho:=\sup_{x \in \mathbb{R}}\left|P(R^z \leq x)-P(R^u \leq x)\right|,
\end{equation}
where $R^z=\mathbf{Z}^* E \mathbf{Z}, \  R^u=\mathbf{U}^*E\mathbf{U}.$

We have the following result on the high dimensional Gaussian approximation. Define $\xi_c:=\sup_{i,t}|\alpha_i(t)|.$
It is notable that $\xi_c$ can be well-controlled for the commonly used basis functions. For instance, for the trigonometric polynomials and the weighted Chebyshev  polynomials of the first kind, $\xi_c=O(1);$ and for orthogonal wavelet, $\xi_c=O(\sqrt{c}).$ The following theorem establishes the Gaussian approximation for high dimensional quadratic forms under physical dependence. 
\begin{thm}\label{thm_gaussian}
Under Assumption \ref{assu_phylip}, \ref{assu_smmothness}, \ref{assu_basisregularity}, \ref{assu_derivativebasis} and \ref{assu_parameter}, for some constant $C>0,$ we have 
\begin{equation*}
 \rho \leq Cl(n),  
\end{equation*}
where $l(n) $ is defined as
\begin{align*}
l(n)=\psi^{-1/2}+& \xi_c p\psi^{\frac{q}{q+1}}M^{\frac{q(-\tau+1)}{q+1}}+\xi_c M_x^{-1}\psi^2 p^4+\frac{M^2}{\sqrt{n}}\psi^3 p^6 \\
& +p \psi  \Big(\frac{\xi_c^{1/2}}{M_x^{5/6}}+\frac{\sqrt{M}}{M_x^3} \Big)  \sqrt{\log \frac{p}{\gamma}}+\gamma,
\end{align*}
where $M_x, \psi, M \rightarrow \infty$ and $\gamma \rightarrow 0$ when $n \rightarrow \infty.$  
\end{thm}

\subsection{Asymptotic normality of test statistics}\label{section_43} With the above preparation, we now derive the distributions for the test statistics $T^*_1$ and $T^*_2$ defined in Section \ref{sec_subtest}. First of all, under $\mathbf{H}_0^1,$ we have
\begin{align}\label{eq_nt1star}
nT_1^*=\sum_{j=1}^b \sum_{k=1}^c \widehat{a}^2_{jk}=\widehat{\bm{\beta}}^* \widehat{\bm{\beta}} =\frac{\bm{\epsilon}^* Y}{\sqrt{n}} \Sigma^{-2} \frac{Y^* \bm{\epsilon}}{\sqrt{n}}+o_{\mathbb{P}}(1), 
\end{align}
where we recall (\ref{beta_est}). We can analyze $T^*_2$ in the same way using
\begin{align*}
nT_2^*=\sum_{j=k_0+1}^b \sum_{k=1}^c \widehat{a}^2_{jk}&=(A \widehat{\bm{\beta}})^*(A \widehat{\bm{\beta}}) \\ 
&= \frac{\bm{\epsilon}^* Y}{\sqrt{n}} \Sigma^{-1}A^* A \Sigma^{-1} \frac{Y^* \bm{\epsilon}}{\sqrt{n}}+o_{\mathbb{P}}(1),
\end{align*}
where $A \in \mathbb{R}^{bc \times bc}$ is a block diagonal matrix with the non-zero block being the lower $(b-k_0)c \times (b-k_0)c$ major part. 

Note that $\frac{1}{\sqrt{n}} Y^* \bm{\epsilon} \in \mathbb{R}^{p}$ is a block vector with size $c,$ where the $j$-th entry of the $i$-block is $\frac{1}{\sqrt{n}}\sum_{k=b+1}^n x_{k-i}\epsilon_{k} \alpha_j(\frac{k}{n}).$ We can therefore rewrite it as 
\begin{equation*}
\frac{1}{\sqrt{n}}Y^* \bm{\epsilon}=\frac{1}{\sqrt{n}} \sum_{i=b+1}^n \mathbf{h}_i \otimes \mathbf{b}(\frac{i}{n}), 
\end{equation*}  
where $ \mathbf{h}_i=\mathbf{x}_i \epsilon_i.$ For $i>b,$ $\mathbf{h}_i$ can be regarded as a locally stationary time series, i.e. $\mathbf{h}_i=\mathbf{U}(\frac{i}{n}, \mathcal{F}_i).$ Denote the long-run covariance matrix of $\{\mathbf{h}_i\}$ as 
\begin{equation*}
\bar{\Delta}(t)=\sum_{j=-\infty}^{\infty} \text{Cov} \Big( \mathbf{U}(t, \mathcal{F}_j), \mathbf{U}(t, \mathcal{F}_0) \Big),
\end{equation*}
and we further define
\begin{equation}\label{eq_deltadef}
\Delta=\int_0^1 \bar{\Delta}(t) \otimes (\mathbf{b}(t) \mathbf{b}^*(t)) dt.
\end{equation}
For $k \in \mathbb{N},$ denote  $$ f_k=\left(\text{Tr}[ (\Delta^{1/2} \Sigma^{-2} \Delta^{1/2})^k] \right)^{1/k}, \ g_k=\left(\text{Tr}[ (\Delta^{1/2} \Sigma^{-1}A^*A \Sigma^{-1}  \Delta^{1/2})^k] \right)^{1/k}. $$

The limiting distributions of $T_1^*$ and $T_2^*$ are summarized in the following theorem. 

\begin{thm}\label{prop_quadist} Under Assumption \ref{assu_phylip}, \ref{assu_smmothness}, \ref{assu_basisregularity}, \ref{assu_derivativebasis} and \ref{assu_parameter}, when $l_n\rightarrow 0,$ we have 
\begin{itemize}
\item[(1).] Under $\mathbf{H}_0^1,$ we have 
\begin{equation*}
\frac{nT_1^*-f_1}{f_2} \Rightarrow \mathcal{N}(0,2).
\end{equation*}
Furthermore, there exist some positive constants $c_i, C_i, i=1,2,$ such that 
\begin{equation*}
c_1 \leq \frac{f_1}{bc} \leq C_1, \ c_2 \leq \frac{f_2}{\sqrt{bc}} \leq C_2.  
\end{equation*}
\item[(2).] Under  $\mathbf{H}_0^2$, we have 
\begin{equation*}
\frac{nT_2^*-g_1}{g_2} \Rightarrow \mathcal{N}(0,2).
\end{equation*}
Furthermore, there exist some positive constants $w_i, W_i, i=1,2,$ such that 
\begin{equation*}
w_1 \leq \frac{g_1}{(b-k_0)c} \leq W_1, \ w_2 \leq \frac{g_2}{\sqrt{(b-k_0)c}} \leq W_2.  
\end{equation*}
\end{itemize}
\end{thm}

Finally, we discuss the local power of our tests. We will only focus on the white noise test and similar discussion can be applied to the bandedness test.   Consider the alternative    
\begin{equation*}
\mathbf{H}_a:   \frac{n \sum_{j=1}^{\infty} \int_0^1 \gamma^2(t, j)dt}{\sqrt{bc}} \rightarrow \infty.
\end{equation*}
The following proposition states that under $\mathbf{H}_a,$ the power of our test will asymptotically be 1. 
%

\begin{prop}\label{pro_power}
Under Assumption \ref{assu_phylip}, \ref{assu_smmothness}, \ref{assu_basisregularity}, \ref{assu_derivativebasis} and \ref{assu_parameter}, when the alternative hypothesis $\mathbf{H}_a$ holds true, for any given significant level $\alpha,$ we have
\begin{equation*}
\mathbb{P} \left( \left| \frac{nT_1^*-f_1}{f_2} \right| \geq \sqrt{2} \mathcal{Z}_{1-\alpha} \right) \rightarrow 1, \ n \rightarrow \infty,
\end{equation*}
where $\mathcal{Z}_{1-\alpha}$ is the $(1-\alpha) \%$ quantile of the standard normal distribution. 
\end{prop}
Proposition \ref{pro_power} states that the white noise test has asymptotic power 1 whenever $\sum_{j=1}^{\infty} \int_0^1 \gamma^2(t, j)dt\gg \sqrt{bc}/n$. In an interesting special case when $\int_0^1 \gamma^2(t, j_i)dt\gg \sqrt{bc}/(nk)$, $i=1,2,\cdots,k$, $T_1^*$ achieves asymptotic power 1. Note that if $k$ here is large, then we conclude that alternatives consists of many very small deviations from the null can be picked up by the ${\cal L}^2$ test $T_1^*$. On the contrary, maximum deviation or ${\cal L}^\infty$ norm based tests will not be sensitive to such alternatives.  

%
%

\subsection{Practical implementation}\label{sec4_subtest}
It can been seen from Theorem \ref{prop_quadist} that the key to implement the tests is to estimate the covariance matrix of the high dimensional vector $\{\mathbf{x}_i \epsilon_i\}$. A disadvantage of using (\ref{eq_nt1star}) is that the basis functions are mixed with the time series. In the present subsection, we provide a practical implementation by representing $nT_1^*$ and $nT_2^*$ into different forms in order to separate the data and the basis functions. We focus our discussion on $nT_1^*.$ 

For $i>b, j \leq b,$ denote the vector $\mathbb{B}_j(\frac{i}{n}) \in \mathbb{R}^{bc}$ with $b$-blocks, where the $j$-th block is the basis $\mathbf{b}(\frac{i}{n})$ and zeros otherwise. Therefore, for all $j \leq b, b<i \leq n,$ we have
\begin{equation}\label{eq_bj}
\Big(\phi_j(\frac{i}{n})-\widehat{\phi}_j(\frac{i}{n}) \Big)^2 =\mathbb{B}_j^*(\frac{i}{n}) \Sigma^{-1} \frac{Y^* \bm{\epsilon}}{n} \frac{\bm{\epsilon}^* Y}{n} \Sigma^{-1} \mathbb{B}_j(\frac{i}{n})+o_{\mathbb{P}}(1).
\end{equation}
Denote $\mathbf{q}_{ij}^*=\mathbb{B}_j^*(\frac{i}{n}) \Sigma^{-1} \in \mathbb{R}^{bc}$ and $\mathbf{q}_{ijk}$ as the $k$-th block of $\mathbf{q}_{ij}$ of size $c.$ As a consequence,  we can write 
\begin{equation}\label{eq_summationminimax}
\mathbf{q}^*_{ij} \frac{Y^* \bm{\epsilon}}{n}=\frac{1}{n} \sum_{k=b+1}^n \mathbf{h}_k^* \widetilde{\mathbf{q}^{k}_{ij}},
\end{equation}
where we recall $\mathbf{h}_k=\epsilon_k \mathbf{x}_k, \widetilde{\mathbf{q}_{ij}^k} \in \mathbb{R}^b$ is denoted by $(\widetilde{\mathbf{q}_{ij}^k})_s=\mathbf{q}^*_{ijs} \mathbf{b}(\frac{k}{n}).$ Denote $\mathbf{Q}_{ij} \in \mathbb{R}^{(n-b)b \times (n-b)b}$ as a block matrix with size $b \times b$ whose $(k_1,k_2)$-th block is $\mathbf{q}^{\widetilde{k}_1}_{ij}(\mathbf{q}^{\widetilde{k}_2}_{ij})^*.$ Furthermore, we denote 
\begin{equation*}
Q(\frac{i}{n})=\sum_{j=1}^b \mathbf{Q}_{ij}, \  Q_{k_0}(\frac{i}{n})=\sum_{j=k_0}^b \mathbf{Q}_{ij}.
\end{equation*}
By (\ref{eq_summationminimax}) and Theorem \ref{prop_quadist}, it suffices to study the following quantity
\begin{equation*}
n^2T_1^{**} = (\Sigma_L^{1/2} \mathbf{z}_L)^* \Big( \int_0^1 Q(t) dt \Big) (\Sigma_L^{1/2} \mathbf{z}_L),
\end{equation*}
where $\Sigma_L$ is the covariance matrix of  $\mathbf{h}=(\mathbf{h}_{b+1}, \cdots, \mathbf{h}_n)^*$ and $\mathbf{z}_L \sim \mathcal{N}(\mathbf{0}, \mathbf{I})$,  $\mathbf{I} \in \mathbb{R}^{(n-b)b}.$ Similarly, we use the following statistic to study $\mathbf{H}_0^2$
\begin{equation*}
n^2T_2^{**}=(\Sigma_L^{1/2} \mathbf{z}_L)^* \Big( \int_0^1 Q_{k_0}(t) dt \Big) (\Sigma_L^{1/2} \mathbf{z}_L).
\end{equation*}

The above expressions are useful for our practical implementation as they provide us a way to separate the deterministic basis functions and the random part. Hence we only need to estimate the covariance  matrix $\Sigma_L$ for $\mathbf{h}.$ Next we will provide a nonparametric estimator for $\Sigma_L.$ Similar ideas have been employed to estimate the long-run covariance matrix in \cite{WZ2} for fixed dimensional random vectors.

We observe that the covariance matrix of $\mathbf{h}$ is a $(n-b) \times (n-b)$ block matrix with block size $b$.  We first consider the diagonal part, where each block $\Lambda_{k}$ is the covariance matrix of $\mathbf{h}_k, k=b+1, \cdots, n.$ Recall that  we can write $\{\mathbf{h}_k\}$ into a sequence of locally stationary time series $\{\mathbf{U}(\frac{k}{n}, \mathcal{F}_k)\}_{k=b+1}^n.$ Denote 
\begin{equation*}
\Lambda(t,j)=\operatorname{Cov}(\mathbf{U}(t, \mathcal{F}_0), \mathbf{U}(t, \mathcal{F}_j)).
\end{equation*}
The following lemma shows that $\Lambda_{kk},$ which is the $k$-th diagonal block of $\Sigma_L,$  can be well estimated by $\Lambda(\frac{k}{n},0)$ for any $k>b.$ We will put its proof in Appendix \ref{appendix_a}.
\begin{lem}\label{lem_longrundiag}
Under Assumption \ref{assu_phylip} and \ref{assu_smmothness}, we have 
\begin{equation*}
\sup_{k>b} \left | \left| \Lambda(\frac{k}{n},0)-\Lambda_{kk} \right| \right|=O(n^{-1+4/\tau}).
\end{equation*}
\end{lem} 

Next we consider the upper-off-diagonal blocks.  For any $b< k \leq n-b+1, $ we find that for $j>b+k,$ for some constant $C>0,$ we have 
\begin{equation} \label{eq_boundofflamada}
\left| \left| \Lambda_{kj} \right|\right| \leq C (j-b)^{-\tau+1},  
\end{equation}
where we use a similar discussion to Lemma \ref{lem_phy} and \ref{lem_disc}. As a consequence, we only need to estimate the blocks $\Lambda_{kj}$ for $k<j \leq k+b.$  Similar to Lemma \ref{lem_longrundiag}, we have 
\begin{equation*}
\left| \left| \Lambda(\frac{k}{n},j)-\Lambda_{kj} \right| \right|=O(n^{-1+4/\tau}).
\end{equation*} 
Hence, we need to estimate $ \Lambda(t,j), \ 0\leq j \leq b$ using the kernel estimators. For a smooth symmetric density function $K_{h}$ defined on $\mathbb{R}$ supported on $[-1,1],$ where $h \equiv h_n$ is the bandwidth such that $h \rightarrow 0, \ nh \rightarrow \infty.$  We write
\begin{equation*}
\widehat{\Lambda}(t,j)=\frac{1}{nh} \sum_{k=b+1}^{n-j} K\left(\frac{k/n-t}{h}\right)\mathbf{h}_{k} \mathbf{h}^*_{k+j}, \ 0 \leq j \leq b. 
\end{equation*}
Finally we define $\widehat{\Sigma}_L$ as the estimator by setting its blocks  
\begin{equation} \label{eq_longcov}
(\widehat{\Sigma}_L)_{kk}=\widehat{\Lambda}(\frac{b+k}{n},0), \ (\widehat{\Sigma}_L)_{kj}=\widehat{\Lambda}(\frac{k+b}{n},j),
\end{equation}
and zeros otherwise, where $k=1,2,\cdots, n-b, \ k<j \leq k+b.$ We can prove that our estimators are consistent under mild assumptions.

\begin{thm}\label{thm_bootstrap} Under Assumption \ref{assu_phylip} and  \ref{assu_smmothness}, let $ h \rightarrow 0$ and $nh \rightarrow \infty,$ for $ j=0,1,2,\cdots,b,$ we have 
\begin{equation}\label{eq_boundlamda}
\sup_{t} \left| \left| \Lambda(t,j)-\widehat{\Lambda}(t,j) \right| \right|=O_{\mathbb{P}} \Big(b  \Big( \frac{1}{\sqrt{nh}}+h^2 \Big) \Big).
\end{equation}
As a consequence, we have 
\begin{equation}\label{eq_sigmaest}
\left| \left| \Sigma_L-\widehat{\Sigma}_L \right| \right| =O_{\mathbb{P}} \Big(b^2 \Big( \frac{1}{\sqrt{nh}}+h^2 \Big) \Big).
\end{equation}
\end{thm}

In practice, the true $\epsilon_i$ is unknown and we have to use $\widehat{\epsilon}_i$ defined in (\ref{eq_hatepsilon}). We then define 
\begin{equation*}
\widetilde{\Lambda}(t,j)=\frac{1}{nh} \sum_{k=b+1}^{n-j} K\left(\frac{k/n-t}{h}\right)\mathbf{\widehat{h}}_{k} \mathbf{\widehat{h}}^*_{k+j}, \ 0 \leq j \leq b. 
\end{equation*}
where $\mathbf{\widehat{h}}_k:=\mathbf{x}_k\widehat{\epsilon}_{k}.$  Similarly, we can define the estimation $\widetilde{\Sigma}_L.$ The analog of Theorem \ref{thm_bootstrap} is the following result. 
\begin{thm} \label{thm_bootstrap1}
Under the assumptions of Theorem \ref{thm_bootstrap} and Assumption \ref{assu_basisregularity}, \ref{assu_derivativebasis} and \ref{assu_parameter}, we have 
\begin{equation*}
\sup_{t} \left| \left| \Lambda(t,j)-\widetilde{\Lambda}(t,j) \right| \right|=O_{\mathbb{P}} \Big(b  \Big( \frac{1}{\sqrt{nh}}+h^2+\theta_n \Big) \Big),
\end{equation*}
where $\theta_n$ is defined as
\begin{equation*}
\theta_n= \sqrt{\frac{b}{nh}}\left( \zeta_c\sqrt{\frac{\log n}{n}}+n^{-d \alpha_1} \right).
\end{equation*}
As a consequence, we have 
\begin{equation*} 
\left| \left| \Sigma_L-\widetilde{\Sigma}_L \right| \right|=O_{\mathbb{P}} \Big( b^2 \Big( \frac{1}{\sqrt{nh}}+h^2+\theta_n \Big) \Big).
\end{equation*}
\end{thm}

By Theorem \ref{prop_quadist}, \ref{thm_bootstrap} and \ref{thm_bootstrap1}, we now propose the following practical procedure to test $\mathbf{H}_0^1$ (the implementation for $\mathbf{H}_0^2$ is similar): \\

1.  For $j=1,2,\cdots,b,  i=b+1,\cdots, n,$ estimate $\Sigma^{-1}$ using $n(Y^*Y)^{-1}$ and calculate $\mathbf{Q}_{ij}$ by the definitions. \\

2. Choose the tuning parameters $b$ and $c$ according to Section \ref{sec_parachoice}. \\

3. Estimate $\Sigma_L$ using (\ref{eq_longcov}) from the samples $\{\mathbf{\widehat{h}}_k\}_{k=b+1}^n.$ \\

4.  Generate B (say 2000) i.i.d copies of Gaussian random vectors $\mathbf{z}_i, \ i=1,2,\cdots, B.$ Here $\mathbf{z}_i \sim \mathcal{N}(\mathbf{0}, \mathbf{I}).$ For each $k=1,2,\cdots, B,$ calculate the following Riemann summation 
\begin{equation*}
T^1_k=\frac{1}{n^2} \sum_{j=1}^b \sum_{i=b+1}^n (\widehat{\Sigma}_L \mathbf{z}_k)^* \mathbf{Q}_{ij} (\widehat{\Sigma}_L \mathbf{z}_k).
\end{equation*}

5. Let $T_{(1)}^1 \leq T_{(2)}^1 \leq \cdots \leq T_{(B)}^1 $ be the order statistics of $T^1_k, k=1,2,\cdots, B.$ Reject $\mathbf{H}_0^1$ at the level $\alpha$ if $T_1^{*}>T^1_{(\lfloor B(1-\alpha) \rfloor)},$ where $\lfloor x \rfloor$ stands for the largest integer smaller or equal to $x.$ Let $B^*=\max\{k: T^1_{(k)} \leq T_1^{*}\},$ the $p$-value can be denoted as $1-B^*/B.$ \\

\subsection{Choices of tuning parameters} \label{sec_parachoice}

In this subsection, we briefly discuss the practical choices of the key parameters, i.e. the lag $b$ of the auto-regression in Cholesky decomposition, the number of basis functions in sieve estimation, choice of $k_0$ in the bandedness test and the bandwith selection in the nonparametric estimation of covariance matrix. 

Similar to the discussion in Section \ref{sec_subtest}, by Proposition  \ref{lem_borderapproximation}, Lemma \ref{lem_approxphi} and Theorem \ref{thm_timevaryingcoeff}, for any given sufficiently large $b_0 \equiv b_0(n), $ the following statistic should be small enough 
\begin{equation*}
\mathcal{T}_b=\sum_{j=b_1}^{b_0} \int_0^1 \widehat{\phi}^2_j(t)dt, \ b<b_1<b_0.
\end{equation*}
By Theorem \ref{prop_quadist}, $\mathcal{T}_b$ is normally distributed. Hence, we can follow the procedure    described in the end of Section \ref{sec4_subtest}. For each fixed $b_1<b_0$, we can formulate the null hypothesis as $\mathbf{H}_0^b: b_1>b. $
Given the level $\alpha,$ denote 
\begin{equation*}
b^*=\max_{b_1}\{b_1<b_0: \mathbf{H}_0^b \ \text{is rejected}  \}.
\end{equation*}
Then we can choose $b=b^*.$  Note that $b^*+1$ is the first off diagonal where all its entries are effectively zeros in terms of statistical significance.

The number of basis functions can be chosen using model selection methods for nonparametric sieve estimation. However, due to non-stationarity, the classic Akaike information criterion (AIC) may fail under heteroskedasticity.  In the present paper, we use  the cross-validation method described in \cite[Section 8]{BH} where the cross-validation criterion is defined as 
\begin{equation*}
\text{CV}(c)=\frac{1}{n}\sum_{i=2}^n \frac{\widehat{\epsilon}^2_{ic}}{(1-\upsilon_{ic})^2},
\end{equation*}
where  $\{\widehat{\epsilon}_{ic}\}$ are the estimation residuals using sieve method with order of $c$ and $\upsilon_{ic}$ is the leverage defined as $\upsilon_{ic}=\mathbf{y}_i^*(Y^*Y)\mathbf{y}_{i},$
where we recall (\ref{y_kronecker}). Hence, we can choose 
\begin{equation*}
\widehat{c}=\argmin_{1 \leq c \leq c_0} \text{CV}(c),
\end{equation*} 
where $c_0$ is a pre-chosen large value.

Finally the bandwidth can be chosen using the standard leave-one-out cross-validation criterion for nonparametric estimation. Denote 
 \begin{equation*}
 \widehat{J}(h):=\sup_j \left|\left|\int_0^1 \widetilde{\Lambda}(t,j) \circ   \widetilde{\Lambda}(t,j)dt- \frac{2}{n} \sum_{k=b+1}^n \widetilde{\Lambda}_{-k}(t_k,j)          \right|\right|,
 \end{equation*}
where $t_i=\frac{i}{n}$, $\circ$ is the Hadamard (entrywise) product for matrices and $\widetilde{\Lambda}_{-k}$ is the estimation excluding the sample $\widehat{\mathbf{h}}_k\widehat{\mathbf{h}}_{k+j}^*$. Therefore, the selected bandwidth is 
\begin{equation*}
\widehat{h}=\argmin_{h} \widehat{J}(h).
\end{equation*}

\noindent{\bf Acknowledgments.} The authors are grateful for the suggestions of the referee, the associated editor and the editor, which have improved the paper significantly.

\clearpage

\begin{appendix}
\section{Simulation Studies}\label{sec_examsimu}
In this section, we design Monte Carlo experiments to study the finite sample accuracy and sensitivity of our estimation and testing procedure. First of all, we list the choices of tuning parameters $b$ and $c$ 

We mention that the state-of-the-art technique for choosing the number and values of $\{a_{jk}\}$ under certain sparsity assumption is the LASSO method \cite{lassopaper}. We record the choices of tuning parameters using our method, two-step CV from Section \ref{sec_parachoice} and the LASSO method in Table  \ref{table_tuning} for a few non-stationary processes considered in (\ref{eq_ma1})--(\ref{eq_ar2}). We find that the LASSO method is on one hand a little bit overestimated and on the other hand ignore the information of $b,$ which stands for the decay of temporal dependence since we indeed have a structure for our model.  In Section \ref{sec5_simu1} and \ref{sec5_simupower}, we will use such estimates for the estimation of precision matrices and hypothesis testing. Overall, we find that the two-step procedure, even though it will employ the CV twice,  has a better performance than simply using LASSO.  Our results perform better due to the fact that the coefficients $a_{jk}$ are overallly decreasing. While LASSO is more suitable for choosing parameters which are not ordered.   

\begin{table}[!ht]
\begin{center}
\begin{tabular}{@{}l*5{>{}l}%
  l<{}l@{}}
\toprule[1.5pt]
&  & &  Two-step CV  & & & LASSO \\ 
  \midrule[1pt]
  & \normal{\head{}} & \normal{\head{n=200}}
  & \normal{\head{n=500}} & \head{n=800} & \head{n=200}  &\head{n=500} &  \head{n=800}  \\
  \midrule[1pt]
  \multirow{3}{*}{MA(1)} & Fourier Basis & (2,2)  & (2,2)  & (1,2) &  8 & 6 & 6 
  \\
    & Polynomial Basis & (2,4) &  (2,6)  & (2,6) & 16 & 18 & 18 \\
  & Wavelet Basis & (1,4) & (2,4)  & (1,4) & 24 & 24 & 32 \\
  \midrule[1pt]
    \multirow{3}{*}{MA($2$)} & Fourier Basis &  (2,2) & (2,4) & (2,2) & 18 & 24 & 24
  \\
   & Polynomial Basis & (3,8) &  (3,7)  & (4,7) & 24 & 24 & 28 \\
  & Wavelet Basis & (2,8)  & (3,8)  & (2,16) & 24 & 24 & 24 \\
  \midrule[1pt]
  \multirow{3}{*}{AR(1)} & Fourier Basis & (4,6) &  (4,8) &  (4,8) & 28 & 36 & 32 
  \\
   & Polynomial Basis & (5,8) &  (4,10)  & (4,8) & 28 & 28 & 32 \\
  & Wavelet Basis & (6,8)  & (4,8)  &  (4,8) & 48 & 48 & 36 \\
  \midrule[1pt]
      \multirow{3}{*}{AR($2$)} & Fourier Basis &  (6,6) & (6,8) & (6,8) & 42 & 42 & 32
  \\
   & Polynomial Basis & (6,10) & (7,12)  & (6,10) & 36 & 42& 42 \\
  & Wavelet Basis & (5,8) & (6,8)  & (5,8) & 48 & 48 & 48 \\
\bottomrule[1.5pt]
\end{tabular}
\end{center}
  \vspace*{-1mm}
\caption{Choices of $b$ and $c$ based on Two-step CV and LASSO. In our two-step CV method,  we record the choices of $(b,c)$ as a pair and in LASSO we record the length of $b \times c.$ } \label{table_tuning}
\end{table}

\subsection{Accuracy of precision matrix estimation}\label{sec5_simu1} In this subsection, we show by simulations the finite sample performance of our estimation. For i.i.d $\mathcal{N}(0,1)$ random variables $\{\epsilon_i\}$, we investigate the non-stationary MA(1), MA(2), AR(1) and AR(2) processes respectively, i.e., 
\begin{equation} \label{eq_ma1}
x_i=0.6 \cos(\frac{2\pi i}{n})\epsilon_{i-1}+\epsilon_i.
\end{equation}
\begin{equation}\label{eq_ma2}
x_i=0.6\cos(\frac{2\pi i}{n})\epsilon_{i-1}+0.3\sin(\frac{2\pi i}{n})\epsilon_{i-2}+\epsilon_i.
\end{equation}
\begin{equation}\label{eq_ar1}
x_i=0.6 \cos(\frac{2\pi i}{n})x_{i-1}+\epsilon_i.
\end{equation}
\begin{equation}\label{eq_ar2}
x_i=0.6\cos(\frac{2\pi i}{n})x_{i-1}+0.3\sin(\frac{2i\pi}{n})x_{i-2}+\epsilon_i.
\end{equation} 

It is easy to compute the true precision matrices
of the above models. In the following simulations, we report the average estimation errors in terms of operator norm and their standard deviations based on 1000 repetitions. We use the methods from Section \ref{sec_parachoice} to choose the parameters and the Epanechnikov kernel \cite[Section 4.2]{WL} for the nonparametric estimation. We also record the results when we use LASSO for estimating the coefficients $\{a_{jk}\}$ in (\ref{defn_forest}). 
We compare the results for three different types of sieves, the Fourier basis functions (i.e. trigonometric polynomials), the  Legendre polynomials and  Daubechies orthogonal wavelet basis functions of order $16$ \cite{DI}. 

We observe from Table \ref{table_preest} that our estimators for the precision matrices are reasonably accurate. Due to the consistency of our estimators, they are more accurate when $n$ becomes larger. Furthermore, as we can see from the estimation of MA(1) and MA(2) processes, our estimators can still be quite accurate even when the underlying precision matrices are not sparse. Due to the sparsity structure, we find that LASSO can provide us an accurate estimates.  However, overall, the two-step CV method has a better performance than LASSO. 

\begin{table}[!ht]
\begin{center}
\begin{tabular}{@{}l*5{>{}l}%
  l<{}l@{}}
\toprule[1.5pt]
&  & &  Two-step CV & & & LASSO \\ 
  \midrule[1pt]
  & \normal{\head{}} & \normal{\head{n=200}}
  & \normal{\head{n=500}} & \head{n=800} & \head{n=200}  &\head{n=500} &  \head{n=800}  \\
  \midrule[1pt]
  \multirow{3}{*}{MA(1)} & Fourier Basis & 1.17 (0.18)  & 1.09 (0.18)  & 0.96 (0.14)&  1.43 (0.13) & 1.48 (0.23) & 1.24 (0.21)
  \\
    & Polynomial Basis & 1.48 (0.12) & 1.46 (0.19)  & 1.37 (0.21) & 1.63 (0.17) & 1.64 (0.19) & 1.34 (0.24) \\
  & Wavelet Basis & 1.5 (0.21) & 1.31 (0.21)  & 1.1 (0.23) & 1.83 (0.27) & 1.84 (0.25) & 1.74 (0.26)\\
  \midrule[1pt]
    \multirow{3}{*}{MA($2$)} & Fourier Basis &  1.35 (0.12) & 1.28 (0.16) & 1.18 (0.18) & 1.68 (0.1) & 1.6 (0.17) & 1.44 (0.14)
  \\
   & Polynomial Basis & 1.47 (0.13) & 1.43 (0.14)  & 1.32 (0.13) & 1.73 (0.19) & 1.86 (0.17) & 1.56 (0.21) \\
  & Wavelet Basis & 1.55 (0.21)  & 1.37 (0.19)  & 1.19 (0.22) & 1.62 (0.13) & 1.67 (0.14) & 1.54 (0.2) \\
  \midrule[1pt]
  \multirow{3}{*}{AR(1)} & Fourier Basis & 0.53 (0.18) &  0.46 (0.17) &  0.4 (0.18) & 0.54 (0.1) & 0.42 (0.11) & 0.4 (0.14)
  \\
   & Polynomial Basis & 0.61 (0.13) & 0.56 (0.12)  & 0.54 (0.14) & 0.6 (0.17) & 0.64 (0.19) & 0.44 (0.24) \\
  & Wavelet Basis & 0.68 (0.21)  & 0.62 (0.23)  &  0.57 (0.24) & 0.69 (0.17) & 0.62 (0.19) & 0.6 (0.2)  \\
  \midrule[1pt]
      \multirow{3}{*}{AR($2$)} & Fourier Basis &  0.78 (0.21) & 0.71 (0.24) & 0.64 (0.24) & 0.79 (0.2) & 0.68 (0.17) & 0.66 (0.18)
  \\
   & Polynomial Basis & 0.82 (0.15) & 0.76 (0.11)  & 0.75 (0.1) & 0.89 (0.18) & 0.86 (0.17) & 0.8 (0.22) \\
  & Wavelet Basis & 0.9 (0.22) & 0.83 (0.24)  & 0.78 (0.24) & 0.89 (0.27) & 0.9 (0.18) & 0.85 (0.25) \\
\bottomrule[1.5pt]
\end{tabular}
\end{center}
  \vspace*{-1mm}
\caption{Operator norm error for estimation of precision matrices. The standard deviations are recorded in the bracket. We use the trigonometric polynomials for Fourier basis, the Legendre polynomials for Polynomial basis and Daubechies wavelet of order $16$ for Wavelet basis. } 
\label{table_preest}
\end{table}


\subsection{Accuracy and power of tests}\label{sec5_simupower}
In this subsection, we design simulations to study the finite sample performance for the white noise and bandedness tests of precision matrices using the procedure described in the end of Section \ref{sec4_subtest}. At the nominal levels $0.01, 0.05$ and $0.1$, the simulated Type I error rates are listed below for the null hypothesis of $\mathbf{H}_0^1$ and $\mathbf{H}_0^2$ based on 1000 simulations, where for $\mathbf{H}_0^{2}$ we use the time varying AR(2) model (i.e. $k_0=2$).  From Table \ref{table_whitenoisealpha} and \ref{table_bandalpha}, we see that the performance of our proposed tests are reasonably accurate for all the above basis functions. We also record the results when we use LASSO for the estimation. We find that overall, our two-step CV method provides more accurate results.

\begin{table}[ht]
\begin{center}
\begin{tabular}{@{}l*5{>{}l}%
  l<{}l@{}}
\toprule[1.5pt]
&  & &  Two-step CV & & & LASSO \\ 
  \midrule[1pt]
  &  &\head{n=200}
  & \head{n=500}& \head{n=800} & \head{n=200}  &\head{n=500} &  \head{n=800}  \\
  \midrule[1pt]
  \multirow{3}{*}{$\alpha=0.01$} & Fourier Basis & 0.008  &  0.01 & 0.009 & 0.006 & 0.006 & 0.007  \\
     & Polynomial Basis & 0.009  &  0.0098 & 0.011 & 0.007 & 0.007 & 0.007  \\
  & Wavelet Basis  & 0.008 & 0.008 & 0.01 & 0.12 & 0.12 & 0.13 \\
  \midrule[1pt]
    \multirow{3}{*}{$\alpha=0.05$} & Fourier Basis &  0.057 & 0.046 & 0.045
& 0.039 & 0.041 & 0.047  \\
    & Polynomial Basis & 0.059  &  0.048 & 0.052 & 0.056 & 0.061 & 0.054 \\
  & Wavelet Basis & 0.053 & 0.048  & 0.047 & 0.038 & 0.039 & 0.058 \\
  \midrule[1pt]
  \multirow{3}{*}{$\alpha=0.1$} & Fourier Basis & 0.11 &  0.097 &  0.1 & 0.09 & 0.092 & 0.098
  \\
    & Polynomial Basis & 0.087  &  0.093 & 0.12 & 0.14 & 0.15 & 0.152 \\
  & Wavelet Basis & 0.091 & 0.087 &  0.088 & 0.089 & 0.094 & 0.088 \\
\bottomrule[1.5pt]
\end{tabular}
\end{center}
  \vspace*{-1mm}
\caption{Simulated type I error rates under $\mathbf{H}_0^1.$}
\label{table_whitenoisealpha}
\end{table}

\begin{table}[!ht]
\begin{center}
\begin{tabular}{@{}l*5{>{}l}%
  l<{}l@{}}
\toprule[1.5pt]
&  & &  Two-step CV &  & & LASSO \\ 
  \midrule[1pt]
  & \normal{\head{}} & \normal{\head{n=200}}
  & \normal{\head{n=500}} & \head{n=800} & \head{n=200}  &\head{n=500} &  \head{n=800}  \\
  \midrule[1pt]
  \multirow{3}{*}{$\alpha=0.01$} & Fourier Basis & 0.009  &  0.013 & 0.009 & 0.008 & 0.008 & 0.006 
  \\
 & Polynomial Basis & 0.008  &  0.011 & 0.009 & 0.013 & 0.016 & 0.014
  \\
  & Wavelet Basis  & 0.011 & 0.014 & 0.008 & 0.018 & 0.018 & 0.016 \\
  \midrule[1pt]
    \multirow{3}{*}{$\alpha=0.05$} & Fourier Basis &  0.052 & 0.05 & 0.049 & 0.056 & 0.058 & 0.058
  \\
   & Polynomial Basis & 0.051  &  0.048 & 0.052 & 0.056 & 0.054 & 0.052
  \\
  & Wavelet Basis & 0.052 & 0.05  & 0.05 & 0.041 & 0.042 & 0.045 \\
  \midrule[1pt]
  \multirow{3}{*}{$\alpha=0.1$} & Fourier Basis & 0.096 &  0.097 &  0.11 & 0.13 & 0.12 & 0.13
  \\
   & Polynomial Basis & 0.089  &  0.098 & 0.092 & 0.12 & 0.12 & 0.11
  \\
  & Wavelet Basis & 0.091 & 0.101 & 0.095 & 0.088 & 0.087 & 0.084 \\
\bottomrule[1.5pt]
\end{tabular}
\end{center}
  \vspace*{-1mm}
\caption{Simulated type I error rates under $\mathbf{H}_0^2$ for $k_0=2.$}
\label{table_bandalpha}
\end{table}

Next we consider the statistical power of our tests under some given alternatives. For the test of white noise, we choose the four examples considered in Section \ref{sec5_simu1} as our alternatives. For the testing of bandedness of the precision matrices, for the null hypothesis, we choose $k_0=2$ and consider the alternatives of AR(3) and MA(3) processes respectively, i.e.,  for $\delta \in (0,0.3),$
\begin{equation*}
x_i=0.6\cos(\frac{2\pi i}{n})x_{i-1}+0.3\sin(\frac{2i\pi}{n})x_{i-2}+\delta \sin(\frac{2i \pi }{n}) x_{i-3}+\epsilon_i, 
\end{equation*}
\begin{equation}\label{eq_ma3}
x_i=0.6\cos(\frac{2\pi i}{n})\epsilon_{i-1}+0.3\sin(\frac{2i\pi}{n})\epsilon_{i-2}+\frac{i}{n}\epsilon_{i-3}+\epsilon_i.
\end{equation}

In all of our simulations, we choose the Daubechies wavelet basis functions of order $16$ as our sieve basis functions and the Epanechnikov kernel for the nonparametric estimation (\ref{eq_longcov}). For the choices of the parameters, we follow the discussion of Section \ref{sec_parachoice}.
Figure \ref{fig1} shows that our testing procedures are quite robust and have strong statistical power for both tests.

\begin{figure}[!ht] \label{fig_beta}
\begin{center}
  \includegraphics[height=6cm, width=6cm]{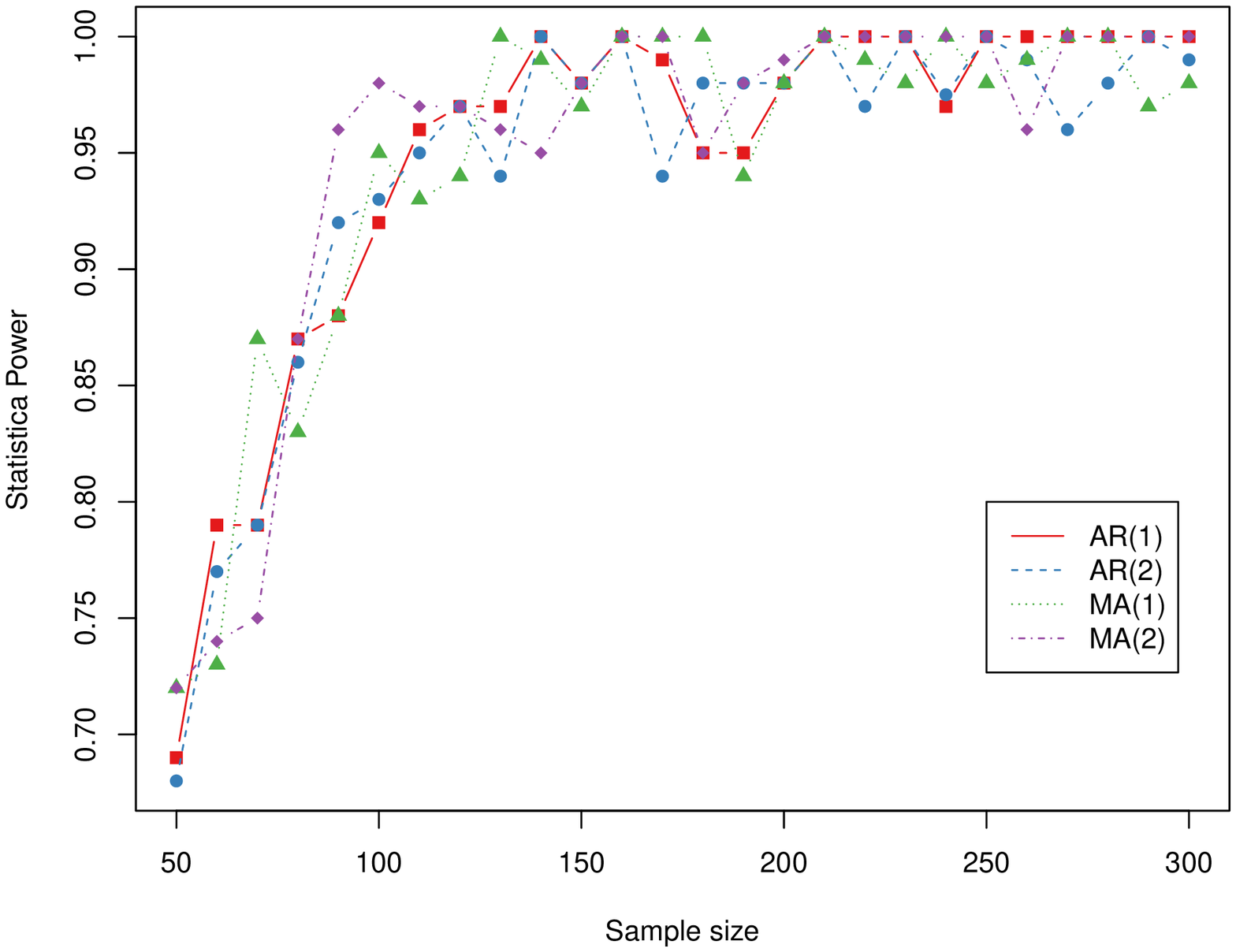} 
    \includegraphics[height=6cm, width=6cm]{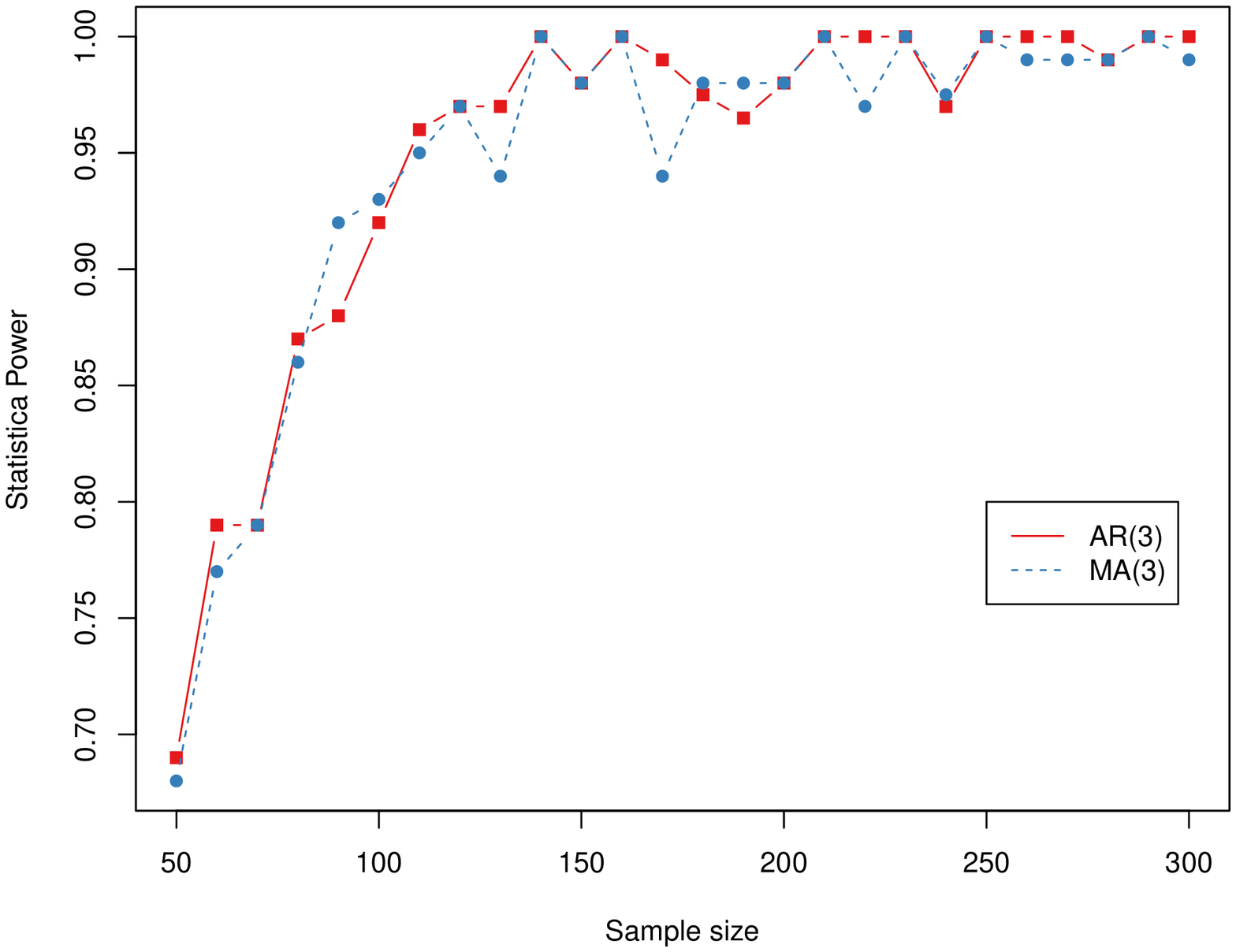} 
  \end{center}
    \vspace*{-1mm}
\caption{Left panel: power of White noise testing  under nominal level 0.05. Right panel: power of bandedness testing under nominal level 0.05. For the AR(3) process we choose $\delta=0.2.$}
\label{fig1}
\end{figure}


Finally, we simulate the statistical power for various choices of $\delta$ in the AR(3) process for the sample size $n=200, 300$ respectively in Figure \ref{fig3}, we find that our method is quite robust.  

\begin{figure}[!ht] \label{fig_beta1}
\begin{center}
  \includegraphics[height=6cm, width=8.5cm]{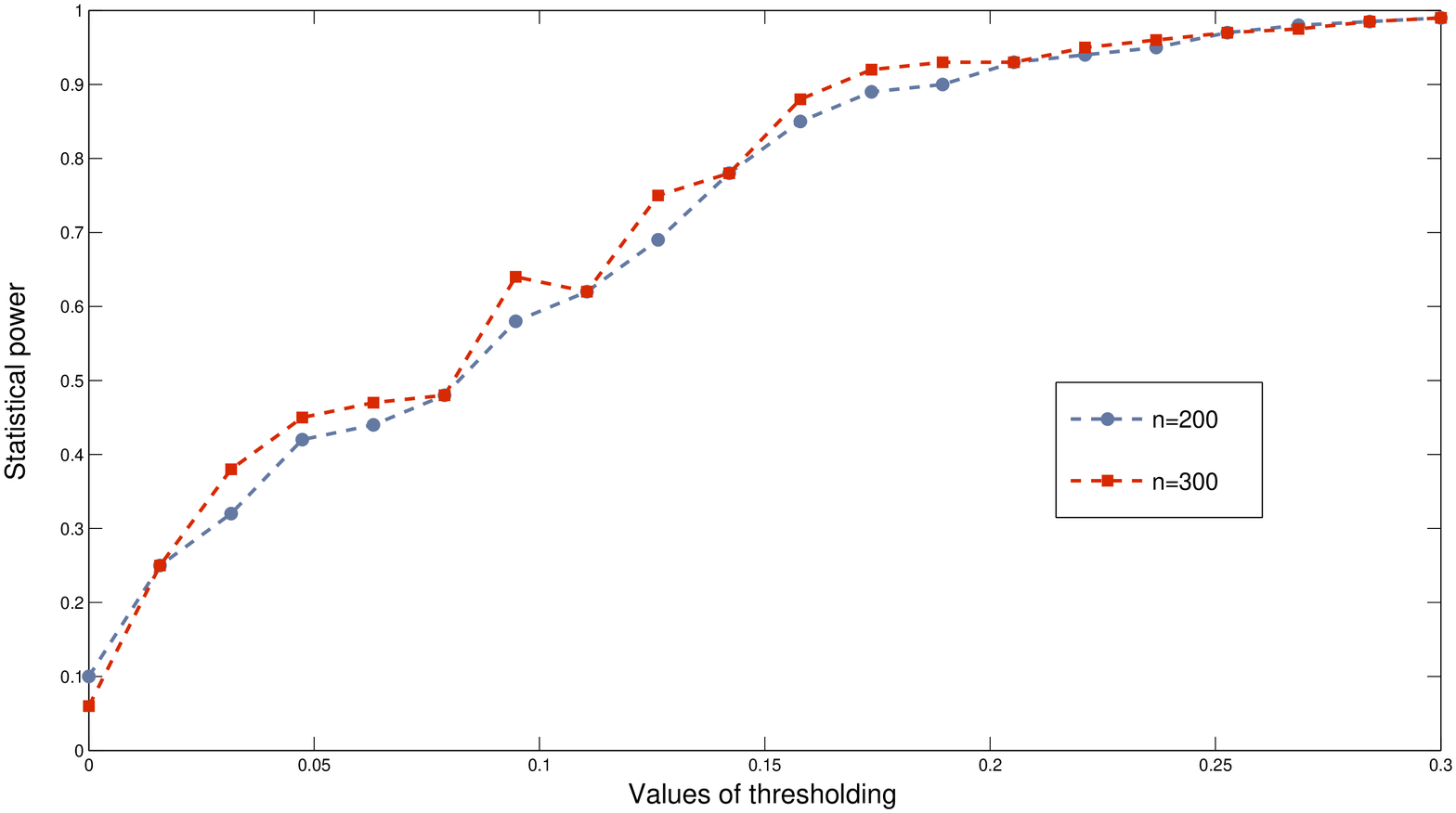} 
  \end{center}
  \vspace*{-1mm}
\caption{Power of bandedness testing under nominal level 0.05 for different values of $\delta.$}
\label{fig3}
\end{figure}

\vspace{4pt}

%

\section{Proofs}\label{sec_proofs} In this section, we prove the main theorems of this  paper.

\begin{proof}[Proof of Theorem \ref{thm_timevaryingcoeff}] We will follow the proof strategy of \cite[Lemma 2.3]{CC}. The key difference is that our projection matrix $S=\frac{1}{n} \sum_{i=b+1}^n \mathbf{y}_i \mathbf{y}_i^*$  will converge to some deterministic matrix other than identity. We therefore need to provide an analog of \cite[Lemma 2.2]{CC}, where they derive the convergence rate for $\beta$-mixing processes and use Berbee's lemma. Here, in our paper, we will use the trick of $m$-dependent sequence to prove our results. The proof contains three main steps: (i). Find the convergent limit $\Sigma$ for $S$; (ii). Find the optimal rate for the norm of $\Sigma-S$; (iii). Follow the proof of \cite[Lemma 2.3]{CC} to conclude our proof. We start with the first step.  
For any $1 \leq i \leq bc, $ denote 
\begin{equation*}
i^{\prime}:=i-ci^{\prime \prime}, \ i^{\prime \prime}:=\lfloor \frac{i}{c}\rfloor+1.
\end{equation*}
With the above definitions, for $1 \leq i,j \leq bc,$ we can write $S_{ij}$ as 
 \begin{equation*}
 S_{ij}=\frac{1}{n} \sum_{k=b+1}^n x_{k-i^{\prime \prime}} x_{k-j^{\prime \prime}} \alpha_{i^{\prime}}(\frac{k}{n}) \alpha_{j^{\prime}}(\frac{k}{n}).
 \end{equation*}
 For $1 \leq k_1, k_2 \leq b, $ denote $$S_n^{k_1,k_2}=\sum_{i=b+1}^n \left( x_{i-k_1} x_{i-k_2}-\mathbb{E}(x_{i-k_1} x_{i-k_2}) \right),$$ 
and $U_i^{k_1,k_2}=x_{i-k_1} x_{i-k_2}-\mathbb{E}(x_{i-k_1} x_{i-k_2}),$ it is easy to check that under the assumption of (\ref{assum_phy}), $\delta^u(j,q) \leq Cj^{-\tau}$ by Cauchy-Schwarz inequality.  As a consequence, by Lemma \ref{lem_con}, we have 
\begin{equation} \label{eq_snbound}
\left| \left| S_n^{k_1,k_2} \right| \right|_q^2 \leq C_q \sum_{j=-n}^{\infty} (\zeta^u_{j+n}-\zeta^u_j)^2, 
\end{equation}
where $\zeta^u_j=\sum_{k=0}^j \delta^u(k,q).$ Therefore, we have  
\begin{equation*}
\left| \left| S_n^{k_1,k_2} \right| \right|_q=O(n^{1/2}).
\end{equation*} 
Under the assumption (\ref{assum_lip}), as $|k_1-k_2| \leq b$ (for instance, we assume $k_1\leq k_2$), we then have 
\begin{equation*}
\left| \mathbb{E}(x_{i-k_1}x_{i-k_2})- \gamma(\frac{i-k_1}{n}, k_2-k_1) \right| \leq \frac{k_2-k_1}{n},
\end{equation*}
where we use (\ref{assum_moment}) and Jensen's inequality. This implies that 
\begin{equation*}
\frac{1}{n} \sum_{i=b+1}^n \mathbb{E}(x_{i-k_1} x_{i-k_2})=\frac{1}{n} \sum_{i=b+1}^n \gamma(\frac{i-k_1}{n},k_2-k_1)+O\left((k_2-k_1)n^{-2}\right).
\end{equation*}
Under Assumption \ref{assu_smmothness}, using \cite[Theorem 1.1]{HT}, we have 
 \begin{equation*}
\frac{1}{n} \sum_{i=b+1}^n \mathbb{E}(x_{i-k_1} x_{i-k_2})=\int_0^1 \gamma(t, k_2-k_1)dt+O((k_2-k_1)n^{-2}).
\end{equation*}
Therefore, combine with (\ref{eq_snbound}), for some $\alpha_3 \in (0,1),$ with $1-O(n^{2\alpha_3-1})$ probability, we have 
\begin{equation*}
\frac{1}{n} \sum_{i=b+1}^n x_{i-k_1} x_{i-k_2}=\int_0^1 \gamma(t, k_2-k_1)dt+O(n^{-\alpha_3}).
\end{equation*}
Similarly, we can show that for $1 \leq c_1, c_2 \leq c,$
\begin{equation}\label{eq_lawoflargenumberwithtime}
\frac{1}{n}\sum_{i=b+1}^{n-1}x_{i-k_1} x_{i-k_2} \alpha_{c_1}(\frac{i}{n}) \alpha_{c_2}(\frac{i}{n})=\int_0^1 \widetilde{\gamma}(t, k_2-k_1, c_1, c_2)+O(n^{-\alpha_3}),
\end{equation}
where $ \widetilde{\gamma}(t, k_2-k_1, c_1, c_2)$ is defined as 
\begin{equation*}
\widetilde{\gamma}(\frac{i-k_1}{n}, k_2-k_1, c_1, c_2)=\gamma(\frac{i-k_1}{n}, k_2-k_1) \alpha_{c_1}(\frac{i-k_1}{n}) \alpha_{c_2}(\frac{i-k_1}{n}).
\end{equation*}
Denote the matrix $\Sigma^{b}(t) \in \mathbb{R}^{b \times b}$ with the entries $\Sigma^{b}_{kl}(t)=\gamma(t, |k-l|), \ 1 \leq k,l \leq b $ and further denote
\begin{equation}\label{eq_sigma}
\Sigma= \int_0^1 \Sigma^{b}(t) \otimes \left( \mathbf{b}(t) \mathbf{b}^*(t) \right) dt.
\end{equation}
Using Lemma \ref{lem_disc}, with $1-O(n^{4/\tau+2\alpha_1+2\alpha_3-1})$ probability, we have
\begin{equation*}
\lambda_{\max}(S-\Sigma) \leq Cn^{2/\tau+\alpha_1-\alpha_3}.
\end{equation*}
Next, we will derive the optimal rate using the concentration inequality for random matrices Lemma \ref{lem_berstein} and the trick of $m$-dependent sequence. In order to deal with the issue of independence, we use the approximation of $m$-dependent sequence. For convenience, for $\mathbf{x}_i, i=b+1, \cdots,n,$ we denote by
\begin{equation*}
\mathbf{x}_i=\mathcal{G}(\frac{i}{n}, \mathcal{F}_i) \in \mathbb{R}^b.
\end{equation*}
We then denote the $m$-approximation sequence by 
\begin{equation}\label{eq_mdependent}
\mathbf{x}_i^M =\mathbb{E} (\mathbf{x}_i | \eta_{i-M}, \cdots, \eta_i ), \ i=b+1,\cdots,n,
\end{equation}
where $\mathbf{x}_i, \mathbf{x}_j$ are independent when $|i-j|>M.$ Under Assumption \ref{assu_phylip}, by the discussion of \cite[Remark 2.3]{CZ}, we conclude that for any $i,$
\begin{equation}\label{eq_mcontrol1}
\mathbb{P}\Big(\sup_j |\mathbf{x}_{ij}-\mathbf{x}_{ij}^M| \geq t \Big) \leq C \frac{(\log b)^{q/2} M^{-q\tau}}{t^q}.
\end{equation}
Recall that $\mathbf{y}_i=\mathbf{x}_i \otimes \mathbf{b}(\frac{i}{n}),$ we now denote $\mathbf{y}_i^M=\mathbf{x}_i^M \otimes \mathbf{b}(\frac{i}{n}),$ by choosing $t=M^{-\tau+2},$ we conclude that with $1-\frac{n (\log b)^{q/2}}{M^{2q}}$ probability, we have 
\begin{equation*}
 \Big| \Big| \frac{1}{n} \sum_{i=b+1}^n \mathbf{y}_i \mathbf{y}_i^* - \frac{1}{n}\sum_{i=b+1}^n \mathbf{y}_i^M (\mathbf{y}_i^M)^* \Big| \Big| \leq C\zeta_c M^{-\tau/2+1}. 
\end{equation*} 
By Assumption \ref{assu_parameter} and the fact $\tau>10$, we can choose $M$ such that $\zeta_c M^{-\tau/2+1}=O(n^{-1}).$ Therefore, it suffices to control the $m$-dependent approximation. Observe that for some constant $C>0$
\begin{align*}
R_n=\frac{1}{n}\sup_i \left| \left| \Big( (\mathbf{y}^M_i)^* \mathbf{y}_i^M-\Sigma \Big) \right| \right|&=\frac{1}{n} \sup_i \left| \left| \Big( (\mathbf{x}_i^* \mathbf{x}_i)\Big(\mathbf{b}^*(\frac{i}{n})\mathbf{b}(\frac{i}{n})\Big)-\Sigma \Big) \right| \right| \\
& \leq C \frac{\zeta_c^2}{n},
\end{align*}
where we use Assumption \ref{assu_phylip} and Lemma \ref{lem_con}. Define $k_0=\lfloor \frac{n-b}{M} \rfloor$ and the index set sequences by 
\begin{equation*}
\mathcal{I}_i=
\begin{cases}
\{b+i+kM: k=0,1,\cdots, k_0 \}, & \text{if} \ b+i+k_0M \leq n, \\
\{b+i+kM: k=0,1,\cdots, k_0-1 \}, & \text{otherwise}   
\end{cases}
\ i=1, 2, \cdots, M. 
\end{equation*}
 By triangle inequality and for some constant $C>0$, we have 
\begin{align*}
\mathbb{P} \Big( \Big| \Big| \frac{1}{n}\sum_{i=b+1}^n \mathbf{y}^M_i(\mathbf{y}_i^M)^*-\Sigma \Big| \Big| \geq t \Big) & \leq  CM \sup_i \mathbb{P} \Big( \Big|\Big| \frac{1}{n} \sum_{k \in \mathcal{I}_i} \mathbf{y}_k^M (\mathbf{y}_k^M)^*-\Sigma \Big|\Big| \geq t/M \Big) \\
& \leq CMbc \exp\Big(\frac{-t^2/(2M^2)}{\sigma_M^2+R_Mt/3M} \Big),
\end{align*}
where we apply Lemma \ref{lem_berstein}. To conclude our proof, we need to choose $M$ properly.   By definition, it is elementary to see that
\begin{equation*}
\sigma_M^2 \leq C \frac{\zeta_c^2}{n} \frac{M}{n}.
\end{equation*}
Now we choose $M=O(n^{1/3}),$ then $\sigma_M^2=O(\frac{\zeta_c^2}{n} n^{-2/3}).$ Hence, by choosing $t=O(\frac{\zeta_c}{\sqrt{n}} \sqrt{\log n}),$ we conclude that
\begin{equation} \label{eq_optimalrate}
\Big| \Big | S-\Sigma \Big| \Big|=O_{\mathbb{P}}(\frac{\zeta_c}{\sqrt{n}} \sqrt{\log n}).
\end{equation} 
Once we have (\ref{eq_optimalrate}), we can almost take the varbartim of the proof of  of \cite[Lemma 2.3]{CC} to finish proving our theorem. Denote 
\begin{equation*}
\check{\phi}_j(t)=\phi_{j}(t)-\widehat{\phi}_j(t). 
\end{equation*}
For $t, t^* \in [0,1],$ by mean value theorem and Assumption \ref{assu_derivativebasis}, for some constant $C,$ we have 
\begin{align*}
|\check{\phi}_j(t)-\check{\phi}_j(t^*)|&=\Big|(\mathbb{B}_j(t)-\mathbb{B}_j(t^*))^* \Big( Y^*Y/n \Big)^{-1} \frac{Y^* \bm{\epsilon}}{n} \Big| \\
& \leq C n^{\omega_1}c^{\omega_2} | t-t^*| \Big|  \frac{Y^* \bm{\epsilon}}{n} \Big|,
\end{align*}
where $\mathbb{B}_j(t)$ is defined in (\ref{eq_bj}). For some constant $\overline{M}>0,$ denote the event $\mathcal{B}_n$ as the event such that $C\Big|  \frac{Y^* \bm{\epsilon}}{n} \Big| \leq \overline{M},$ where it can be easily checked that $\mathbb{P}(\mathcal{B}_n^c)=o(1).$
By Assumption \ref{assu_phylip} and \ref{assu_basisregularity}, Lemma \ref{lem_con}, (\ref{eq_optimalrate}) and a similar discussion to \cite[equations (42) and (43)]{CC}, we conclude that on $\mathcal{B}_n,$ for some constant $C_1>0,$ there exists some positives  $\eta_1, \eta_2$ such that
\begin{equation*}
Cn^{\omega_1} c^{\omega_2}|t-t^*|\Big|  \frac{Y^* \bm{\epsilon}}{n} \Big| \leq C_1 \zeta_c \sqrt{\frac{\log n}{n}},
\end{equation*}
whenever $|t-t^*| \leq \eta_1 n^{-\eta_2}.$ Denote $\mathcal{S}_n$ be the smallest subset of $[0,1]$ such that for each $t \in [0,1],$ there exists a $t_n \in \mathcal{S}_n$ with $|t_n-t| \leq \eta_1 n^{-\eta_2}.$ For any $t \in [0,1],$ let $t_n(t)$ denote as the distance of $t_n \in \mathcal{S}_n$ to $t$. Then using a similar discussion to equations (44)-(47) of \cite{CC}, we conclude that 
\begin{equation*}
\mathbb{P} \left( \sup_t |\check{\phi}_j(t)| \geq 4C \zeta_c\sqrt{(\log n)/n} \right) \leq \mathbb{P} \Big( \Big\{ \max_{t_n \in \mathcal{S}_n} |\check{\phi}_j(t_n)| \geq 2C \zeta_c \sqrt{(\log n)/n}  \Big\}\cap \mathcal{B}_n\Big)+o(1).
\end{equation*}
The rest of the work leaves to control the above probability using  (\ref{eq_optimalrate}), Assumption \ref{assu_phylip} and Lemma \ref{lem_con}. We first observe that 
\begin{align}
& \mathbb{P} \Big( \Big\{ \max_{t_n \in \mathcal{S}_n} |\check{\phi}_j(t_n)| \geq 2C \zeta_c \sqrt{(\log n)/n}  \Big\}\cap \mathcal{B}_n\Big) \nonumber \\
& \leq \mathbb{P} \Big( \max_{t_n \in \mathcal{S}_n} \left| \mathbb{B}_j^*(t_n) (S^{-1}-\Sigma^{-1}) Y^* \bm{\epsilon}/n \right| \geq C \zeta_c \sqrt{(\log n)/n} \Big)\label{eq_finalcontrol1} \\
&+\mathbb{P} \Big( \max_{t_n \in \mathcal{S}_n}| \mathbb{B}_j(t_n)\Sigma^{-1} Y^* \bm{\epsilon}/n| \geq C \zeta_c \sqrt{(\log n)/n} \Big). \label{eq_finalcontrol2}
\end{align}
By (\ref{eq_optimalrate}), (\ref{eq_finalcontrol1}) can be controlled easily using the fact that $\Big|\frac{Y^* \bm{\epsilon}}{n} \Big|=O_{\mathbb{P}}(\sqrt{bc/n}).$ To control (\ref{eq_finalcontrol2}), we adopt the truncation from \cite{CC}. Denote $\mathcal{A}_n$ as the event on which $|| S-\Sigma || \leq \frac{1}{2}$ and (\ref{eq_optimalrate}) implies that $\mathbb{P}(\mathcal{A}_n^c)=o(1).$ Denote $\{M_n: n \geq 1 \}$ be an increasing sequence diverging to $+\infty$ and define
\begin{equation*}
\epsilon_{1,i,n}:=\epsilon_i\mathbf{1}(|\epsilon_i| \leq M_n)-\mathbb{E}[\epsilon_i \mathbf{1}(|\epsilon_i| \leq M_n)| \mathcal{F}_{i-1}],
\end{equation*}
\begin{equation*}
\epsilon_{2,i,n}=\epsilon_i-\epsilon_{1,i,n}, \ g_{i,n}(t_n)=\mathbb{B}_j(t_n)^* \Sigma^{-1} \mathbb{B}_j(\frac{i}{n}) \mathbf{1}(\mathcal{A}_n).
\end{equation*}
As a consequence,   we have
\begin{align} \label{eq_ysbound}
& \mathbb{P} \Big( \max_{t_n \in \mathcal{S}_n}| \mathbb{B}_j(t_n)\Sigma^{-1} Y^* \bm{\epsilon}/n| \geq C \zeta_c \sqrt{(\log n)/n} \Big) \nonumber \\
& \leq (\# \mathcal{S}_n) \max_{t_n \in \mathcal{S}_n} \mathbb{P} \Big( \Big\{ \Big| \frac{1}{n} \sum_{i=b+1}^n g_{i,n} \epsilon_{1,i,n} \Big| \Big\} >\frac{C}{2} \zeta_c \sqrt{(\log n)/n} \cap \mathcal{A}_n \Big) \nonumber \\
&+ \mathbb{P} \Big( \max_{t_n \in \mathcal{S}_n} \Big| \frac{1}{n} \sum_{i=b+1}^n g_{i,n} \epsilon_{2,i,n} \Big| \geq \frac{C}{2} \zeta_c \sqrt{(\log n)/n} \Big)+o(1). 
\end{align}
By the discussion of equations (65b) and (70) of \cite{CC}, we can choose $M_n=O\left(\zeta_c^{-1} \sqrt{n/(\log n)}\right),$ then the above bound can be controlled by $o(1).$ We can conclude our proof using Assumption \ref{assu_parameter}. 
\end{proof}

\begin{proof}[Proof of Theorem \ref{thm_phiileqb}] For each fixed $i \leq b,$ denote 
\begin{equation}\label{eq_di}
\mathbf{d}_i=(d_{11}, \cdots, d_{1c}, d_{21}, \cdots, d_{2c}, \cdots, d_{i-1,1}, \cdots, d_{i-1, c}) \in \mathbb{R}^{(i-1)c}.
\end{equation}
Hence, the OLS estimator of $\mathbf{d}_i$ can be written as 
\begin{equation} \label{eq_hatdi}
\widehat{\mathbf{d}}_i=(\mathbf{Y}_i^* \mathbf{Y}_i)^{-1} \mathbf{Y}_i^* \mathbf{x}^i, \ i \leq b,
\end{equation}
where $\mathbf{Y}_i^* $ is a $(i-1)c \times (n+1-i)$ rectangular matrix whose columns are $\mathbf{y}_i^k:=(x_{k-1}, \cdots, x_{k-i+1})^* \otimes \mathbf{b}(\frac{k}{n}) \in \mathbb{R}^{(i-1)c}, \ k=i, \cdots, n,$ with  $\mathbf{x}^i=(x_i, \cdots, x_n).$ Therefore, similar to (\ref{phi_diff}), we have 
\begin{equation*}
f^i_j(\frac{i}{n})-\widehat{f}^i_j(\frac{i}{n})=(\mathbf{d}_{i,j}-\widehat{\mathbf{d}}_{i,j})^* \mathbf{b}(\frac{i}{n}),
\end{equation*} 
where $\mathbf{d}_{i,j} \in \mathbb{R}^c$ is the $j$-th block of $\mathbf{d}_i, $ similarly for $\widehat{\mathbf{d}}_{i,j}.$  The rest of the proof relies on the following equation and Assumption \ref{assu_basisregularity} 
\begin{equation*}
\mathbf{d}_i-\widehat{\mathbf{d}}_i=(\mathbf{Y}_i^* \mathbf{Y}_i)^{-1} \mathbf{Y}_i^* \bm{\epsilon}^i,
\end{equation*}
where $\bm{\epsilon}^i=(\epsilon_i, \cdots, \epsilon_n).$ For the rest of the proof, we can almost take the verbatim as that of  Theorem \ref{thm_timevaryingcoeff}.

\end{proof}

\begin{proof}[Proof of Theorem \ref{thm_varaincefunction}]  Note
\begin{equation*}
\widehat{\bm{\alpha}}=\bm{\alpha}+\left(\frac{W^*W}{n} \right)^{-1} \frac{W^* \bm{\omega}}{n}+O_{\mathbb{P}}\Big(n^{2/\tau}\Big(\zeta_c \sqrt{\frac{\log n}{n}}+n^{-d\alpha_1} \Big) \Big),
\end{equation*}
where $\bm{\omega}=(\omega_{b+1}, \cdots, \omega_n)^*$ and the error is entrywise. Therefore, the only difference from that of Theorem \ref{thm_timevaryingcoeff} is that $W$ is a deterministic matrix. We ignore the further detail here. 
\end{proof}

\begin{proof}[Proof of Theorem \ref{thm_variancefunctionileqb}] The proof is similar to that of Theorem \ref{thm_timevaryingcoeff}, except that we need to analyze the residual (\ref{eq_residualileqb}).  However, it is easy to see that $\widehat{r}^i_k$ is a locally stationary time series with polynomial decay physical dependence measure. Hence,  we can almost take the verbatim except for some constants.

\end{proof}

\begin{proof}[Proof of Theorem \ref{prop_covest}] Using the fact that any two compatible matrices $A, B$, $AB$ and $BA$ have the same non-zero eigenvalues, for some constant $C>0,$ we have 
 \begin{equation*}
\left | \left| \Omega-\widehat{\Omega} \right| \right| \leq C ||E||,
 \end{equation*}
 where $E$ has the following form of decomposition 
\begin{align*}
E&=E_1+E_2+E_3 \\
& =\left[ \mathbf{\widetilde{D}}-\mathbf{\widehat{\widetilde{D}}} \right]+\left[ \mathbf{\widehat{\widetilde{D}}} \left( \widehat{\Phi}^{-1}-\Phi^{-1} \right)^* \Phi^*+\Phi \left( \widehat{\Phi}^{-1}-\Phi^{-1} \right) \mathbf{\widehat{\widetilde{D}}} \right] \\
&+\left[ \Phi\left(\widehat{\Phi}^{-1}-\Phi^{-1}\right) \mathbf{\widehat{\widetilde{D}}} \left(\widehat{\Phi}^{-1}-\Phi^{-1}\right)^* \Phi^* \right]. 
\end{align*} 

Denote $B:=\mathbf{\widehat{\widetilde{D}}}(\Phi^{-1})^*(\Phi-\widehat{\Phi})^*(\widehat{\Phi}^{-1})^* \Phi^*, $ we therefore have $||E_2|| \leq 2 ||B||.$ We further
denote $R_{\Phi}:=\Phi-\widehat{\Phi},$ we first observe that $R_{\Phi}=0, \ i \leq j.$ Then by Lemma \ref{lem_borderapproximation}, Lemma \ref{lem_approxphi},  Theorem \ref{thm_timevaryingcoeff} and \ref{thm_phiileqb}, for $ i \leq b$ or $j \leq
 b \leq i,$ $(R_{\Phi})_{ij}=O_{\mathbb{P}}(\zeta_c \sqrt{(\log n)/n}+n^{-d\alpha_1}).$ And for $i>b, j>b,|(R_{\Phi})_{ij} | \leq j^{-\tau}.$ This implies that
\begin{equation*}
 \lambda_{\max} \left( (\Phi-\widehat{\Phi})(\Phi-\widehat{\Phi})^* \right)=O_{\mathbb{P}}\Big( n^{4/\tau}\zeta^2_cn^{-1} \log n + n^{-2d\alpha_1+\frac{4}{\tau}} \Big),
\end{equation*}
where we use Lemma \ref{lem_disc}.     As a consequence, by submultiplicaticity, for some constant $C>0,$ we have that 
\begin{equation*}
|| E_2 ||=O_{\mathbb{P}}\Big(  \zeta_c \sqrt{\frac{\log n}{n}}+n^{-d\alpha_1} \Big).
\end{equation*} 
Similarly, we can show that
\begin{equation*}
||E_3||=O_{\mathbb{P}}\Big( n^{4/\tau}\zeta^2_cn^{-1} \log n + n^{-2d\alpha_1+\frac{4}{\tau}} \Big).
\end{equation*}
Denote the centered random variables $\omega_i^b=r_i^b-(\sigma_i^b)^2,$ by Lemma \ref{lem_sigmaigeqb} and (\ref{eq_rapprox}), we have 
\begin{equation}\label{eq_rhatbound}
\widehat{r}_i^b=g(\frac{i}{n})+\omega_i^b+O_{\mathbb{P}}\Big(n^{2/\tau}\Big(\zeta_c \sqrt{\frac{\log n}{n}}+n^{-d\alpha_1} \Big) \Big).
\end{equation}
By (\ref{eq_rhatbound}),  Theorem \ref{thm_varaincefunction} and \ref{thm_variancefunctionileqb} and Assumption \ref{assu_parameter}, we conclude that 
\begin{equation*}
||E_1||=O_{\mathbb{P}}\left( \zeta_c n^{4/\tau} \sqrt{\frac{\log n}{n}}+n^{-d \alpha_1+4/\tau}\right).
\end{equation*}
Hence, we have finished our proof. 
\end{proof}

\begin{proof}[Proof of Theorem \ref{thm_gaussian}] By \cite[Lemma 7.2]{XZW}, we find that 
\begin{equation} \label{lem_gap}
\sup_{x \in \mathbb{R}} \mathbb{P}(x \leq R^u \leq x+\psi^{-1})=O(\psi^{-1/2}). 
\end{equation}
Denote 
\begin{equation*}
g_0(u)=(1-\min(1, \max(u,0))^4)^4,
\end{equation*}
it is easy to check that (see the proof of \cite[Proposition 2.1 and Theorem 2.2]{XZW})
\begin{equation} \label{indicatorfunctionbound}
\mathbf{1}(y \leq x) \leq g_{\psi,x}(y) \leq \mathbf{1}(y \leq x+\psi^{-1}),
\end{equation}
\begin{equation}\label{gderivativebound}
\sup_{y,x} |g^{\prime}_{\psi,x}(y)| \leq g_* \psi, \  \sup_{y,x} |g^{\prime \prime}_{\psi,x}(y)| \leq g_* \psi^2, \ \sup_{y,x} |g^{\prime \prime \prime}_{\psi,x}(y)| \leq g_* \psi^3,
\end{equation}
where $g_{\psi, x}(y):=g_0(\psi(y-x))$ and 
$$ g_*=\max_y \left[|g^{\prime}_0(y)|+|g^{\prime \prime}_0(y)|+|g^{\prime \prime \prime}_0(y)|\right]<\infty.$$
 By (\ref{lem_gap}), (\ref{indicatorfunctionbound}) and a similar discussion to equations (7.5) and (7.6) of \cite{XZW}, $\rho$ can be well controlled if we let $\psi \rightarrow \infty$ and bound 
\begin{equation}\label{key_proof}
\sup_x\left|\mathbb{E}g_{\psi,x}(R^u)-\mathbb{E}g_{\psi,x}(R^z)\right|.
\end{equation}
The rest of the proof leaves to control (\ref{key_proof}). The proof relies on two main steps: (i). an $m$-dependent sequence approximation for the locally stationary time series; (ii). a leave-one-block out argument to control the bounded $m$-dependent time series. We start with step (i) and control the error between the $m$-dependent sequence approximation and the original time series. Recall (\ref{eq_mdependent}), we denote by
\begin{align}\label{M_dependence_const}
\mathbf{z}_i^M=(z_{i1}^M, \cdots, z_{ip}^M)&=\mathbb{E}[\mathbf{z}_i \vert \eta_{i-M},\cdots, \eta_i], \nonumber \\
&=\mathbf{x}_i^M \epsilon_i^M \otimes \mathbf{b}(\frac{i}{n}), 
\end{align}
be an $m$-dependent approximation for $\mathbf{z}_i,$ where $ \epsilon_i^M$ are defined using $\mathbf{x}_i^M.$ Similarly, we can define $R^{zM}$ by replacing $\mathbf{Z}$ with $\mathbf{Z}^M.$ Therefore, by (\ref{gderivativebound}) and the definition of $g_{\psi, x}$, there exists some constant $C>0,$ for some small $\Delta_M>0,$ we have  
\begin{equation}\label{mapprox_tri}
\left| \mathbb{E}[g_{\psi,x}(R^z)-g_{\psi,x}(R^{zM})] \right| \leq Cp\psi \Delta_M+C \mathbb{E}[1-\mathcal{I}_M],
\end{equation}
where $\mathcal{I}_M:=\mathbf{1}\{ \max_{1 \leq j \leq p}|\mathbf{Z}_j-\mathbf{Z}_j^M| \leq \Delta_M\}$. Here we use Lemma \ref{lem_con}, mean value theorem and Cauchy-Schwartz inequality.  We use the following lemma to control  the right-hand side of (\ref{mapprox_tri}) by suitably choosing $\Delta_M. $ Recall (\ref{eq_zis}), we denote the physical dependence measure of $z_{kl}$ as $\delta^z_{kl}(s,q)$ and 
\begin{equation*}
\theta_{s,l,q}:=\sup_{k} \delta_{kl}^z(s,q), \ \Theta_{s,l, q}=\sum_{o=s}^{\infty} \theta_{o,l,q}.
\end{equation*}
By Assumption \ref{assu_phylip} and Lemma \ref{lem_epsilon}, the above physical dependence measures satisfy
\begin{equation}\label{eq_phytheta}
\sup_{ 1 \leq l \leq p} \Theta_{s, l, q}<\xi_c,  \ \sum_{s=1}^{\infty} \sup _{1 \leq l \leq p} s \theta_{s,l,3} < \xi_c. 
\end{equation}

Armed with the above preparation, we now control the right-hand side of (\ref{mapprox_tri}) using the following lemma. We put its proof into Appendix \ref{appendix_a}.

\begin{lem}\label{lem_m_rough} Under Assumption \ref{assu_phylip},  for some constant $C_1>0,$ we have 
\begin{equation}\label{lem_m_rough_eq}
p\psi \Delta_M+ \mathbb{E}[1-\mathcal{I}_M] \leq C_1 \Big( p \psi \Big)^{\frac{q}{q+1}} \Big( \sum_{k=1}^p \Theta_{M,k,q}^q \Big)^{\frac{1}{q+1}}.
\end{equation}
\end{lem}

By choosing a sufficiently large $M$, the right-hand side of (\ref{lem_m_rough_eq}) will be of order $o(1).$ Next we will use the leave-one-block out argument to show that the difference between two $m$-dependent sequences can be well controlled.  Its proof relies on Stein's  method. 

We now introduce the dependency graph strictly following \cite[Section 2.1]{ZC}. For the sequence of $p$-dimensional random vectors $\{\mathbf{z}_k \}_{i=b+1}^n$, we call it dependency graph $G_n=(V_n, E_n),$ where $V_n=\{b+1,\cdots,n\}$ is a set of vertices and $E_n$ is the corresponding set of undirected edges. For any two disjoint subsets of vertices $S,T \subset V_n,$ if there is no edge from any vertex in $S$ to any vertex in $T,$ the collections of the corresponding $\mathbf{z}_k$ will be independent. We further denote $D_{\max,n}$ as the maximum degree of $G_n$ and $D_n=1+D_{\max, n}.$ Next we provide a rough bound for $R^z, R^u$ in terms of the the maximum degree of $G_n.$ Denote 
\begin{equation}\label{eq_F}
F(x)=g_{\psi,x} \circ f, \ \text{where} \ f(x)=x^*Ex, \ x \in \mathbb{R}^{p}.
\end{equation}
We further define the bounded random variables $\widetilde{\mathbf{z}}_{kl}=(\mathbf{z}_{kl} \wedge M_x) \vee (-M_x)-\mathbb{E}[(\mathbf{z}_{kl} \wedge M_x) \vee (-M_x)]$ and $\widetilde{\mathbf{u}}_{kl}=(\mathbf{u}_{ij} \wedge M_y) \vee (-M_y)-\mathbb{E}[(\mathbf{u}_{ij} \wedge M_y) \vee (-M_y)]$ for some $M_x, M_y>0.$  For some small $\Delta>0,$ denote 
\begin{equation}\label{not_delta}
\mathcal{I}:=\mathcal{I}_{\Delta}=\mathbf{1}\{ \max_{1 \leq j \leq p} |\mathbf{Z}_j-\widetilde{\mathbf{Z}}_j|<\Delta, |\max_{1 \leq j \leq p} |\mathbf{U}_j-\widetilde{\mathbf{U}}_j|<\Delta\},
\end{equation}
where $\widetilde{\mathbf{Z}}_j, \widetilde{\mathbf{U}}_j$ are defined using $\widetilde{\mathbf{z}}_{kl}, \widetilde{\mathbf{u}}_{kl}.$
We next denote $N_k=\{l:\{k,l \in E_n\}\}$ and $\widetilde{N}_k=\{k\} \cup N_k.$ Let $\phi(M_x)$ be a constant depending on the threshold parameter $M_x$ such that 
\begin{equation*}
\max_{1 \leq \alpha,\beta \leq p} \frac{1}{n} \sum_{k=b+1}^n \left|\sum_{s \in \widetilde{N}_k}(\mathbb{E}\mathbf{z}_{k\alpha}\mathbf{z}_{s\beta}-\mathbb{E}\widetilde{\mathbf{z}}_{k\alpha}\widetilde{\mathbf{z}}_{s\beta}) \right| \leq \phi(M_x).
\end{equation*}
Analogous quantity $\phi(M_y)$ can be defined for $\{\mathbf{u}_i\}.$ Set $\phi(M_x, M_y)=\phi(M_x)+\phi(M_y)$ and define 
\begin{equation*}
m_{z,k}=(\bar{\mathbb{E}}\max_{1 \leq l \leq p}|\mathbf{z}_{sl}|^k)^{1/k}, \ m_{u,k}=(\bar{\mathbb{E}}\max_{1 \leq l \leq p}|\mathbf{u}_{sl}|^k)^{1/k},
\end{equation*}
\begin{equation*}
\bar{m}_{z,k}=(\max_{1 \leq l \leq p}\bar{\mathbb{E}}|\mathbf{z}_{sl}|^k)^{1/k}, \ \bar{m}_{u,k}=(\max_{1 \leq l \leq p}\bar{\mathbb{E}}|\mathbf{u}_{sl}|^k)^{1/k},
\end{equation*}
where $\bar{\mathbb{E}}(\mathbf{z}_{sl})=\frac{\sum_{s=b+1}^n \mathbb{E} \mathbf{z}_{sl}}{n}.$ The following lemma provides a rough bound for (\ref{key_proof}) in terms of $\Delta,$ which can be improved later for the $m$-dependent sequence. Its proof can be found in Appendix \ref{appendix_a}.

\begin{lem}\label{lem_priorbound} For any $\Delta>0$ defined in (\ref{not_delta}),  denote $M_{xy}=\max\{M_x, M_y\},$ for some constant $C>0,$  we have
\begin{align*}
\sup_x & \Big | \mathbb{E}[g_{\psi,x}(R^u)-g_{\psi,x}(R^z)] \Big | \\
 & \leq C p\psi \Delta+C\mathbb{E}(1-\mathcal{I})+C  \phi(M_x, M_y)\psi^2 p^4+\frac{CD_n^3}{\sqrt{n}}(m^3_{x,3}+m^3_{y,3}) \psi^3 p^6.
\end{align*}
\end{lem} 

For the $m$-dependent sequence, we can easily obtain the following result, whose proof will be put into Appendix \ref{appendix_a}. 
\begin{cor} \label{cor_boundmdependent} Assuming that $\{\mathbf{z}_i\}$ and $\{\mathbf{u}_i\}$ are $m$-dependent sequences, for some constant $C>0,$ we have 
\begin{align*} 
& \sup_x  \Big | \mathbb{E}[g_{\psi,x}(R^u)-g_{\psi,x}(R^z)] \Big | \\
 & \leq Cp \psi \Delta+C\mathbb{E}(1-\mathcal{I})+C \phi(M_x, M_y) \psi^2 p^4  +C\frac{(2M+1)^2}{\sqrt{n}}(\bar{m}^3_{x,3}+\bar{m}^3_{y,3})\psi^3 p^6. 
\end{align*}
\end{cor}

As we can see from the above corollary, we need to control the first two items.  Denote $\varphi(M_x)$ be the smallest finite constant which satisfies that uniformly for $i$ and $j,$
\begin{equation*}
\mathbb{E}(A_{ij}-\check{A}_{ij})^2 \leq N \varphi^2(M_x), \  \mathbb{E}(B_{ij}-\check{B}_{ij})^2 \leq M \varphi^2(M_x),
\end{equation*}
where $\check{A}_{ij}, \check{B}_{ij}$ are the truncated blocked summations of $A_{ij}, B_{ij}$ and defined as 
\begin{equation*}
\check{A}_{ij}=\sum_{l=iN+(i-1)M-N+1}^{iN+(i-1)M}(\mathbf{z}_{lj} \wedge M_x) \vee (-M_x),
\end{equation*}
\begin{equation*}
\check{B}_{ij}=\sum_{l=i(N+M)-M+1}^{i(N+M)} (\mathbf{z}_{lj} \wedge M_x) \vee (-M_x).
\end{equation*}
Similarly, we can define $\varphi(M_y)$ for the Gaussian sequence $\{\mathbf{u}_i\}$. Set $\varphi(M_x, M_y)=\varphi(M_x) \vee \varphi(M_y).$ Furthermore, we let $u_x(\gamma)$ and $u_y(\gamma)$ be the smallest quantities such that 
\begin{equation*}
\mathbb{P} \left( \max_{b+1 \leq i \leq n} \max_{1 \leq j \leq p}|\mathbf{z}_{ij}| \leq u_x(\gamma)  \right) \geq 1-\gamma,
\end{equation*}
\begin{equation*}
\mathbb{P} \left( \max_{b+1 \leq i \leq n} \max_{1 \leq j \leq p}|\mathbf{u}_{ij}| \leq u_y(\gamma) \right) \geq 1-\gamma.
\end{equation*}
Then we have the following control on the first two items of Corollary \ref{cor_boundmdependent}, whose proof can be found in Appendix \ref{appendix_a}. Its analog is \cite[Proposition 4.1]{ZC}.
\begin{lem} \label{lem_boundmdependenttwo} 
Assuming that $M_x>u_x(\gamma)$ and $M_y>u_y(\gamma)$ for some constant $C>0$, $\gamma \in (0,1),$    we have  
\begin{equation*} 
p\psi\Delta+ \mathbb{E}[1-\mathcal{I}] \leq C\left(p \psi \varphi(M_x, M_y) \sqrt{\log (p/\gamma)}+\gamma \right).  
\end{equation*} 
\end{lem}

In a final step, we will control $\rho$ defined in (\ref{eq_rhoij}) by quantifying the parameters in the above bounds.
Recall (\ref{M_dependence_const}), by (\ref{lem_m_rough_eq}), Corollary \ref{cor_boundmdependent} and Lemma \ref{lem_boundmdependenttwo},  for some constant $C>0,$ we have 
\begin{align*}
& \left|\mathbb{E}g_{\psi,x}(R^z)-\mathbb{E}g_{\psi,x}(R^{u})\right| \\
& \leq C\Big( p \psi \Big)^{\frac{q}{q+1}} \Big( \sum_{k=1}^p \Theta_{M,k,q}^q \Big)^{\frac{1}{q+1}}+C \phi(M_x, M_y) \psi^2 p^4  +C\frac{(2M+1)^2}{\sqrt{n}}(\bar{m}^3_{x,3}+\bar{m}^3_{y,3})\psi^3 p^6 \\
&+C\left(p \psi \varphi(M_x, M_y) \sqrt{\log (p/\gamma)}+\gamma \right).  
\end{align*}
By (\ref{eq_phytheta}) and  the proof of \cite[Theorem 2.1]{ZC}, we conclude that 
\begin{equation*}
\varphi(M_x)=C(\xi_c^{1/2}/M_x^{5/6}+\sqrt{N}/M_x^3), \ \varphi(M_y)=C\xi_c/M_y^2, 
\end{equation*}
\begin{equation*}
\phi(M_x, M_y)=C\xi_c(1/M_x+1/M_y).
\end{equation*}
By Assumption \ref{assu_parameter}, we should choose that $M_x, M_y \gg \xi_c.$ 
Furthermore, following the arguments of the proofs of \cite[Theorem 2.1]{ZC} and the orthogonality of the basis functions, we find that 
\begin{equation*}
 \bar{m}^3_{x,3}+\bar{m}^3_{y,3}<\infty.
\end{equation*}
Therefore, we can also ignore all the constants. By (\ref{eq_phytheta}) and Assumption \ref{assu_phylip}, we have that 
\begin{equation*}
\Big( \sum_{k=1}^p \Theta^q_{M,k,q} \Big)^{\frac{1}{q+1}} \leq C\xi_c p^{\frac{1}{q+1}}(M^{-\tau+1})^{\frac{q}{q+1}}.
\end{equation*}
We impose the condition that $M_x=M_y$ and $N=M$. This concludes our proof.  
\end{proof}

\begin{proof}[Proof of Theorem \ref{prop_quadist}] We first observe that $\Delta$ is the limiting covariance of $\frac{Y^* \bm{\epsilon}}{\sqrt{n}}.$ Denote the eigenvalues of $\Delta^{1/2} \Sigma^{-2} \Delta^{1/2}$ as $d_1 \geq d_2 \geq \cdots \geq d_{p} \geq 0.$ Using Theorem \ref{thm_gaussian} with $E=\Sigma^{-2}$ and Lindeberg's central limit theorem, we conclude that 
\begin{equation*}
\frac{nT_1^*-\sum_{i=1}^{p} d_i}{\sqrt{\sum_{i=1}^{p} d_i^2}} \Rightarrow \mathcal{N}(0,2),
\end{equation*}     
provided that the following equation holds true 
\begin{equation}\label{eq_cltcondition}
\frac{d_1}{\sqrt{\sum_{i=1}^{p} d_i^2}} \rightarrow 0.
\end{equation}
Therefore, the rest of the proof is devoted to analyzing the spectrum of $\Delta^{1/2} \Sigma^{-2} \Delta^{1/2}.$  Using a similar discussion to the proof of Theorem \ref{thm_timevaryingcoeff} (i.e. equations (\ref{eq_lawoflargenumberwithtime}) and (\ref{eq_sigma})), we can show that $\lambda_{\min}(\Delta)=O(1)$ and conclude that   $\lambda_{\max}(\Delta)=O(1)$ using Lemma \ref{lem_epsilon} and \ref{lem_disc}. Therefore, we conclude that $d_p=O(1)$ and $d_1=O(1)$ using the fact that $\lambda_{\min}(A) \lambda_{\min}(B) \leq \lambda_{\min}(AB) \leq \lambda_{\max}(AB) \leq \lambda_{\max}(A) \lambda_{\max}(B)$ for any positive definite matrices $A,B.$ Hence, (\ref{eq_cltcondition}) holds true immediately. Similarly, denote the eigenvalues of $\Delta^{1/2} \Sigma^{-1} A^* A \Sigma^{-1} \Delta^{1/2}$ as $\rho_1 \geq \rho_2 \geq \rho_{(b-k_0)c}.$ Similarly, it  can be shown that $\rho_k=O(1), k=1,2,\cdots, (b-k_0)c.$ It is notable that in this case we use Theorem \ref{thm_gaussian} by setting $E=\Sigma^{-1}A^* A \Sigma^{-1}.$

\end{proof}

\begin{proof}[Proof of Proposition \ref{pro_power}]  By the smoothness of $\gamma(t,j)$ and Lemma \ref{lem_phy}, we find that $\mathbf{H}_a$ is equivalent to 
\begin{equation*}
\mathbf{H}_a^{\prime}: \ \frac{n\sum_{j=1}^{b} \int_0^1 \gamma^2(t,j)dt}{\sqrt{bc}} \rightarrow \infty,
\end{equation*} 
 Using (\ref{defn_bestcoeff}), Assumption \ref{assu_phylip} and Lemma \ref{lem_borderapproximation} and \ref{lem_approxphi}, we have 
\begin{equation*}
\frac{n\sum_{j=1}^{b} \int_0^1 \gamma^2(t,j)dt}{\sqrt{bc}} \rightarrow \infty \Leftrightarrow \frac{n\sum_{j=1}^{b} \int_0^1 \phi_j^2(t)dt}{\sqrt{bc}} \rightarrow \infty.
\end{equation*}
As a consequence, we will consider the following alternative
\begin{equation*}
\mathbf{H}_a^{*}: \ \frac{n\sum_{j=1}^{b} \int_0^1 \phi_j^2(t)dt}{\sqrt{bc}} \rightarrow \infty.
\end{equation*}
%
%
%
%
We first observe that 
\begin{align*}
n T_1^*&=n \sum_{j=1}^b \int_0^1 \Big(\phi_j(t)-\widehat{\phi}_j(t) \Big)^2dt+n \sum_{j=1}^b \int_0^1 \phi_j^2(t)dt \\
&+2n \sum_{j=1}^b \int_0^1 \Big( \widehat{\phi}_j(t)-\phi_j(t) \Big) \phi_j(t) dt.
\end{align*}
By Theorem \ref{prop_quadist}, we have that 
\begin{equation*}
\frac{nT_1^*-f_1-n\sum_{j=1}^b \int_0^1 \phi^2_j(t)dt}{f_2}-\frac{2n\sum_{j=1}^b \int_0^1 \Big( \widehat{\phi}_j(t)-\phi_j(t) \Big) \phi_j(t) dt }{f_2} \Rightarrow \mathcal{N}(0,2).
\end{equation*} 
One one hand, we have that  
\begin{equation*}
\sum_{j=1}^b \int_0^1 \Big( \widehat{\phi}_j(t)-\phi_j(t)\Big) \phi_j(t) dt =o\Big(\frac{\sqrt{bc}}{n}\Big) \ \text{in probability}.
\end{equation*}
This is because by Lemma \ref{lem_sievebasisbound}, we can write 
\begin{align*}
\sum_{j=1}^b \int_0^1 \Big( \widehat{\phi}_j(t)-\phi_j(t)\Big) \phi_j(t) dt = \sum_{j=1}^b \sum_{k=1}^c (\widehat{a}_{jk}-a_{jk})a_{jk}=\bm{\beta}^*(\widehat{\bm{\beta}}-\bm{\beta})+O(c^{-d}). 
\end{align*}
Recall (\ref{beta_est}), from the proof of Theorem \ref{thm_timevaryingcoeff} (for instance see (\ref{eq_ysbound})), we find that 
\begin{equation*}
\Big|\bm{\beta}^* (\widehat{\bm{\beta}}-\bm{\beta}) \Big|=O_{\mathbb{P}}\Big( || \bm{\beta} || \sqrt{\frac{\log n}{n}}   \Big).
\end{equation*}  
Hence, under $\mathbf{H}_a, $ we have 
\begin{equation*}
\sum_{j=1}^b \int_0^1 \Big( \widehat{\phi}_j(t)-\phi_j(t)\Big) \phi_j(t) dt=O_{\mathbb{P}} \Big( \sqrt{\log n} \frac{(bc)^{1/4}}{n} \Big).
\end{equation*}
On the other hand, as $\mathbf{H}_a$ is equivalent to $\mathbf{H}_a^*,$ we can conclude our proof using Assumption \ref{assu_parameter}. 
\end{proof}

\begin{proof}[Proof of Theorem \ref{thm_bootstrap}] We only discuss the case when $j=0,$ the other cases can be proved similarly. Denote $R_i:=\widehat{\Lambda}(t_i,0)-\Lambda(t_i,0),$ where $t_i=\frac{i+b}{n}.$ We will focus on the case when $i=1$ and analyze the first entry of $(R_i)_{11}.$ The other cases can be analyzed similarly. By definition, we have that
\begin{align*}
(R_1)_{11}=\frac{1}{nh} \sum_{k=b+1}^n K\left(\frac{k/n-t_1}{h}\right)(x_{k-1} \epsilon_k)^2-\mathbb{E}(\mathbf{U}_1^2(t_1,\mathcal{F}_0)),
\end{align*}
where $\mathbf{U}_1$ is the first entry of $\mathbf{U}(t_1, \mathcal{F}_0).$ On one hand, using Lemma \ref{lem_con}, we have that
\begin{align*}
\left| \left| \frac{1}{nh}\sum_{k=b+1}^n K\left(\frac{k/n-t_1}{h}\right)\Big [(x_{k-1} \epsilon_k)^2-\mathbb{E}(x_{k-1} \epsilon_k)^2 \Big] \right| \right| \leq C(nh)^{-1/2},
\end{align*}
where we recall that $K$ is supported on $[-1,1].$ On the other hand, using the stochastic Lipschitz continuity and the elementary property of kernel estimation, we conclude that 
\begin{align*}
\left| \left| \frac{1}{nh}\sum_{k=b+1}^n K\left(\frac{k/n-t_1}{h}\right)\mathbb{E}(x_{k-1} \epsilon_k)^2-\mathbb{E}(\mathbf{U}_1^2(t_1, \mathcal{F}_0))  \right| \right| \leq C\Big( (nh)^{-1/2}+ h^2 +n^{-1} \Big ).
\end{align*}
This implies that 
\begin{equation*}
||(R_1)_{11}|| \leq C \Big( \frac{1}{\sqrt{nh}}+h^2+n^{-1} \Big).
\end{equation*}
The other entries can be discussed in the same way. As a consequence, using Lemma \ref{lem_disc}, we conclude that
\begin{equation*}
||R_1 || \leq Cb  \Big( \frac{1}{\sqrt{nh}}+h^2+n^{-1} \Big).
\end{equation*} 
Hence, we can conclude the proof of (\ref{eq_boundlamda}). The proof of (\ref{eq_sigmaest}) follows from  (\ref{eq_boundlamda}) and Lemma \ref{lem_blockge}.
\end{proof}

\begin{proof}[Proof of Theorem \ref{thm_bootstrap1}] Following the proof of Theorem \ref{thm_bootstrap}, we consider the case when $i=1.$ Notice that
\begin{equation*}
\left(\widetilde{\Lambda}(t_1, j)-\widehat{\Lambda}(t_1,j) \right)_{11}=O \left(\frac{1}{nh} \sum_{k=b+1}^n K\left(\frac{k/n-t_1}{h}\right)x^2_{k-1} \epsilon_k \Big[ \epsilon_k-\widehat{\epsilon}_k \Big] \right).
\end{equation*}
By Lemma \ref{lem_approxphi} and Theorem \ref{thm_timevaryingcoeff}, we conclude that 
\begin{equation*}
\left|\left(\widetilde{\Lambda}(t_1, j)-\widehat{\Lambda}(t_1,j) \right)_{11} \right| = O_{\mathbb{P}} \left( \sqrt{\frac{b}{nh}} \Big( \zeta_c \sqrt{\frac{\log n}{n}}+n^{-d \alpha_1} \Big) \right). 
\end{equation*}
This finishes our proof. 
\end{proof}

\section{Additional Proofs}\label{appendix_a} This appendix is devoted to providing the additional technical proofs of the lemmas and theorems of this paper. 

\begin{proof}[Proof of Lemma \ref{lem_phy}] Denote $\mathcal{P}_j(\cdot)=\mathbb{E}(\cdot \vert \mathcal{F}_j)-\mathbb{E}(\cdot \vert \mathcal{F}_{j-1}),$ we can write
\begin{equation} \label{project_repres}
x_i=\sum_{k=-\infty}^i\mathcal{P}_k(x_i).
\end{equation}
Denote $t_i=\frac{i}{n},$ with the convenience of (\ref{project_repres}), we have 
\begin{equation}\label{controlgamma}
|\gamma(t_i,j)|  \leq \sum_{k=-\infty}^{\infty} \delta(i-k,2)\delta(i+j-k,2),
\end{equation}
where we use \cite[Theorem 1]{WW} (also see the proof of \cite[Proposition 4]{ZZ2}).  Therefore, there exists a constant $C>0,$ such that $|\gamma(t_i, j)| \leq C j^{-\tau}$ by the  assumption of (\ref{assum_phy}). This concludes our proof. 
\end{proof}
\begin{proof}[Proof of Proposition \ref{lem_borderapproximation}] We start with the proof of (\ref{eq_phijbound}). For $i>b^2,$ denote the $(i-1) \times (i-1)$ symmetric banded matrix $\Gamma_i^T$  by
\begin{equation*}
(\Gamma_i^T)_{kl}=
\begin{cases} 
(\Gamma_i)_{kl}, & |k-l| \leq b^2; \\
0, & \text{otherwise}.
\end{cases}
\end{equation*}
where $\Gamma_i=\Omega_i^{-1}$ is the covariance matrix of $\mathbf{x}_{i-1}^i.$
Using Assumption \ref{assu_phylip} and a simple extension of \cite[Proposition 5.1.1]{BD}, we find that $\lambda_{\min}(\Gamma_i)>C,$ for some constant $C>0.$ Therefore, by Weyl's inequality and Lemma \ref{lem_phy}, we have $\lambda_{\min}(\Gamma_i^T) \geq C-n^{-4+\frac{4}{\tau}}.$ A direct application of Cauchy-Schwarz inequality yields that 
\begin{equation}\label{eq_bandedapprox}
\left| \left| \Omega_i\bm{\gamma}_i-\Omega_i^T \bm{\gamma}_i \right| \right| \leq n^{-4+5/\tau}.
\end{equation}
When $n$ is large enough, $\Gamma_i^T$ is a $b^2$-banded positive definite bounded matrix, then by Lemma \ref{lem_band}, we conclude that for some $\kappa \in (0,1)$ and some constant $C>0,$ we have 
\begin{equation}\label{eq_bandin}
\left| (\Omega_i^T)_{kl}\right| \leq C \kappa^{2|k-l|/b^2}.
\end{equation}  
Therefore, by (\ref{eq_bandedapprox}), (\ref{eq_bandin}) and Lemma \ref{lem_phy}, we conclude our proof when $i \geq b^2$. Similarly we can prove the case when $b<i \leq b^2.$ The second part is due to the Yule-Walker's equation and  Lemma \ref{lem_numerical}. 

\end{proof}

\begin{proof}[Proof of Lemma \ref{lem_approxphi}] For any fixed $i$ and $j=1,2,\cdots,b,$ we have 
\begin{equation}\label{yulewalkerdecom}
\phi_{ij}-\phi_j(\frac{i}{n})= \mathbf{e}_j^* \Omega_i^b\left(\gamma_i^b-\widetilde{\gamma}^b_i\right)+\mathbf{e}_j^*\Omega_i^b\left( \widetilde{\Gamma}^b_i-\Gamma^b_i \right) \widetilde{\Omega}_i^b\widetilde{\gamma}_i^b,
\end{equation}
where $\Gamma_i^b, \widetilde{\Gamma}_i^b$ are the covariance matrices of $\mathbf{x}_i$ and $\widetilde{\mathbf{x}}_i$ respectively. For the first item of the right-hand side of (\ref{yulewalkerdecom}), using Cauchy-Schwarz inequality, we have  
\begin{align*}
\left| \mathbf{e}_j^* \Omega_i^b(\gamma_i^b-\widetilde{\gamma}_i^b) \right|^2 \leq \lambda_{\max}((\Omega_i^b)^*\Omega_i^b)||\gamma^b_i-\widetilde{\gamma}^b_i ||_2^2.
\end{align*}
On one hand, by Lemma \ref{lem_phy}, it is easy to check that  
$ \lambda_{\max}((\Omega_i^b)^*\Omega_i^b)=1/\lambda_{\min}((\Gamma_i^b)*\Gamma^b_i)>\kappa,$ where $\kappa>0$ is some constant; on the other hand, similar to (\ref{controlgamma}), we have 
\begin{equation*}
||\gamma^b_i-\widetilde{\gamma}^b_i ||_2^2=\sum_{k=1}^{b}(\gamma_i(k)-\widetilde{\gamma}_i(k))^2 \leq n^{-2+3/\tau},
\end{equation*} 
where we use the assumption of (\ref{assum_lip}). For the second item,  a similar discussion yields that, for some constant $C>0,$
\begin{equation*}
\left| \mathbf{e}_j^*\Gamma_i^{-1}\left( \widetilde{\Gamma}^b_i-\Gamma^b_i \right) \widetilde{\Gamma}_i^{-1}\widetilde{\gamma}_i \right|^2 \leq C \lambda_{\max}\left(\left( \widetilde{\Gamma}^b_i-\Gamma^b_i \right)^* \left( \widetilde{\Gamma}^b_i-\Gamma^b_i \right) \right) \leq Cn^{-2+4/\tau}.
\end{equation*}
where we use Lemma \ref{lem_disc}.  Hence, the proof follows from Proposition \ref{lem_borderapproximation}.
\end{proof}
\begin{proof}[Proof of Lemma \ref{lem_epsilon}]  Using (\ref{estimationeq}), we have 
\begin{equation*}
\sup_i \sigma_i^2=\lambda_{\max} \left( \Phi \Gamma \Phi^* \right) \leq \lambda_{\max}(\Phi \Phi^*) \lambda_{\max} (\Gamma)<\infty,
\end{equation*} 
where we use Lemma \ref{lem_phy} and Proposition \ref{lem_borderapproximation}, and Lemma \ref{lem_disc}. Then the proof follows from (\ref{eq_reduceepsilon}) . For the control of physical dependence measure, for some constant $C>0,$ we have 
\begin{align} \label{eq_bound1}
\Big| \Big|G(\frac{i}{n}, \mathcal{F}_i)-& G(\frac{i}{n}, \mathcal{F}_{i,j})-\sum_{k=1}^j \phi_{ik}(G(\frac{i-k}{n},\mathcal{F}_{i-k})-G(\frac{i-k}{n}, \mathcal{F}_{i-k,j-k})) \Big |\Big|_q \nonumber \\  
& \leq Cj^{-\tau}+ \sum_{k=1}^{j-1} \phi_{ik} (j-k)^{-\tau}. 
\end{align}
By Proposition \ref{lem_borderapproximation}, $j \leq b$ and 
\begin{equation*}
\sum_{k=1}^{j-1} ((j-k))^{-\tau} \sim 2\int_1^{j/2} ((j-x))^{-\tau} dx,
\end{equation*}  
we can conclude our proof.
\end{proof}

\begin{proof}[Proof of Lemma \ref{lem_smoothphijt}] Recall we use $\widetilde{\Gamma}^b_i$ to stand for the covariance matrix of $\widetilde{\mathbf{x}}_i.$ 
By definition, we have 
$\widetilde{\Gamma}_i^b \bm{\phi}^b(t)=\widetilde{\bm{\gamma}}_i^b.$ Using Cramer's rule, we have 
\begin{equation*}
\phi_j(t)=\frac{\det (\widetilde{\Gamma}_i^b)_j}{\det \widetilde{\Gamma}^b_i}, \ j=1,2,\cdots, b,
\end{equation*}
where $(\widetilde{\Gamma}_i^b)_j$ is the matrix formed by replacing the $j$-th column of $\widetilde{\Gamma}^b_i$ by the column vector $\widetilde{\bm{\gamma}}_i^b$. As $\widetilde{\Gamma}_i^b$ is non-singular, it suffices to show that $\det (\widetilde{\Gamma}_i^b)_j, \det \widetilde{\Gamma}^b_i$ are $C^d$ functions of $t$ on $[0,1].$  We employ the definition of determinant, where 
\begin{equation*}
\det \widetilde{\Gamma}^b_i=\sum_{\sigma \in S_b} \left(\operatorname{sign}(\sigma) \prod_{k=1}^b (\widetilde{\Gamma}_i^b)_{k, \sigma(k)}  \right), 
\end{equation*}
where $S_b$ is the permutation group. Similarly for $(\widetilde{\Gamma}_i)_j.$ Due to Lemma \ref{lem_borderapproximation}, the possible maximal item in the expansion of $\det \widetilde{\Gamma}_i^b $ is the non-zero item $\Delta=\prod_{k=1}^b (\widetilde{\Gamma}_i^b)_{k, k}.$ Therefore, it suffices to show that $\det (\widetilde{\Gamma}_i^b)_j/\Delta, \ \det \widetilde{\Gamma}_i^b /\Delta$ are in $C^d.$ Now we reorder the items to write 
\begin{equation*}
\det \widetilde{\Gamma}_i^b=\sum_{j=1}^{b !} \omega_j \gamma_j, 
\end{equation*}
where $\gamma_j$ is some entry of $\widetilde{\Gamma}_i$ and $w_j$ contains the items $\operatorname{sign}(\sigma)$ and $\prod_{k=1}^{b-1} (\widetilde{\Gamma}_i^b)_{k, \sigma(k)}.$ It is easy to check that $\sum_{j=1}^{b!} |\omega_j|<\infty,$ we can therefore conclude our proof using Lemma \ref{lem_phy}.
\end{proof}

\begin{proof}[Proof of Lemma \ref{lem_checkinginvert}]  
Since $\mathbf{b}(t) \mathbf{b}^*(t)$ is a rank-one matrix with positive nontrivial largest eigenvalue and $\int_0^1 \mathbf{b}(t) \mathbf{b}^*(t) dt=I,$ Assumption \ref{assu_basisregularity} will be satisfied if $\Sigma^k(t)$ is positive definite for all $t \in [0,1]. $ We consider the case when $b>q$ and the other case can be proved similarly (actually easier).  

By the definition of $\Sigma^k(t),$ we find that the diagonal entries are $1+\sum_{j=1}^q a_j^2(t).$ The off-diagonal entries can be computed easily. Therefore, by Lemma \ref{lem_disc}, when (\ref{eq_mainvert}) holds true, $\Sigma^k(t), k=1,2,\cdots,b$ will be strongly diagonally dominant and hence positive definite for all $t \in [0,1]$. This concludes our proof for $b \leq q$.  Similar discussion holds for $b <q.$ 

%

\end{proof}

\begin{proof}[Proof of Lemma \ref{lem_coeffsmalli}] For the proof that $f$ is smooth, it is similar to that of Lemma \ref{lem_smoothphijt}, we omit the detail here. And for the  proof of (\ref{eq_phiapproileqb}), we have 
\begin{equation*}
\widetilde{\phi}_{ij}-f_j(\frac{i}{n})= \mathbf{e}_j^* \Omega_i^i\left(\bm{\gamma}^i_i-\widetilde{\gamma}^i_i\right)+\mathbf{e}_j^*(\Gamma_i^i)^{-1}\left( \widetilde{\Gamma}^i_i-\Gamma^i_i \right) \widetilde{\Omega}_i^i\widetilde{\bm{\gamma}}_i^i.
\end{equation*}
Then the rest of the proof is similar to that of Lemma \ref{lem_approxphi},  where we use Lemma \ref{lem_sievebasisbound}.
\end{proof}
\begin{proof}[Proof of Lemma \ref{lem_sigmaigeqb}] First of all, for some constant $C>0,$ we have that 
\begin{equation*}
\left| \sigma_i^2-(\sigma_i^b)^2 \right| \leq Cn^{-2+\frac{1}{\tau}}, 
\end{equation*}
where we use Assumption \ref{assu_phylip} and Lemma \ref{lem_borderapproximation}.  Similarly, by (\ref{assum_lip}), we can show that 
\begin{equation*}
\left| (\sigma_i^b)^2-g(\frac{i}{n}) \right| \leq Cn^{-1+4/\tau}.
\end{equation*}
$g(t) \in C^d([0,1])$ is due to Assumption \ref{assu_smmothness} and the fact that $\phi_{ij}$ is absolutely summable. 
\end{proof}


%

\begin{proof}[Proof of Lemma \ref{lem_longrundiag}]  Similar to the proof of  Lemma \ref{lem_approxphi}, we have 
\begin{equation*}
\left| \left| \Lambda(\frac{k}{n})-\Lambda_{kk} \right| \right|_{\infty} \leq n^{-1+2/\tau}.
\end{equation*}
Hence, the proof follows from Lemma \ref{lem_disc}. 

\end{proof}

\begin{proof}[Proof of Lemma \ref{lem_m_rough}]
By \cite[Lemma A.1]{LL} (see the discussion below equation (14) of \cite{ZC}), we find that 
\begin{equation*}
\left( \mathbb{E}|\mathbf{Z}_k-\mathbf{Z}_k^M|^q \right)^{2/q} \leq C \Theta^{2}_{M,k,q},
\end{equation*}
where  $C$ is some positive constant. As a consequence, we have 
\begin{align}\label{projection_control}
\mathbb{E}[1-\mathcal{I}_M] & \leq \sum_{j=1}^p \mathbb{P}(|\mathbf{Z}_j-\mathbf{Z}_j^M| \geq \Delta_M) \leq  \frac{1}{\Delta_M^q} \sum_{j=1}^p \mathbb{E}|\mathbf{Z}_j-\mathbf{Z}_j^M|^q  \nonumber \\
& \leq  \frac{C}{\Delta_M^q} \sum_{j=1}^p \Theta^q_{M,j,q}.
\end{align}
Optimizing the bound with respect to $\Delta_M,$  we finish our proof. 
\end{proof}

\begin{proof}[Proof of Lemma \ref{lem_controlderivative}] 
By a direct computation, we have 
\begin{equation}\label{eq_partialf}
\partial_jf(z)=2 \sum_{i=1}^{p}e_{ij} z_i, \ \partial^2_{jk}f(z)=2, \ \partial^3_{jkl}f(z)=0.
\end{equation}
Using (\ref{gderivativebound}) and chain rule, there exists some constant $C>0,$  we have
\begin{equation*}
|\partial_{jk}^2 F(z)| \leq C(\psi^2 |\partial_j f \partial_k f| +\psi |\partial^2_{jk}f|),
\end{equation*} 
\begin{equation*}
|\partial^3_{jkl}F(z)| \leq C[\psi^3 |\partial_j f  \partial_k f  \partial_l f|+\psi^2(|\partial^2_{jk}f \partial_lf|+|\partial^2_{jl}f \partial_kf|+|\partial^2_{kl}f \partial_jf|)].
\end{equation*}
We then conclude our proof using (\ref{eq_partialf}).

\end{proof}
\begin{proof}[Proof of Lemma \ref{lem_priorbound}] 
Similar to (\ref{mapprox_tri}), for some constant $C>0,$ we have 
\begin{equation*}
\sup_x \left| \mathbb{E}[g_{\psi,x}(R^z)-g_{\psi,x}(\widetilde{R}^z)] \right| \leq \psi p \Delta+C \mathbb{E}(1-\mathcal{I}),
\end{equation*}
Hence, it suffices to control $\sup_x \left| \mathbb{E}\left[g_{\psi,x}(\widetilde{R}^{z})-g_{\psi,x}(\widetilde{R}^{u})\right]\right|.$  Denote 
\begin{equation*}
\Psi(t)=\mathbb{E}F(Z(t)), \ Z(t)=\sum_{k=b+1}^n Z_i(t),
\end{equation*}
where $F$ is defined in (\ref{eq_F}) and  
\begin{equation*}
Z_i(t)=\frac{\sqrt{t}\widetilde{\mathbf{z}}_i+\sqrt{1-t}\widetilde{\mathbf{u}}_i}{\sqrt{n}}.
\end{equation*}
Similar to the discussion of equations (25) and (26) in \cite{ZC}, we have 
\begin{align*}
\mathbb{E}\left[g_{\psi,x}(\widetilde{R}^{z})-g_{\psi,x}(\widetilde{R}^{u})\right]=\int_0^1 \Psi^{\prime}(t)dt&=\frac{1}{2}\sum_{i=b+1}^n \sum_{j=1}^{p} \int_0^1 \mathbb{E}[\partial_jF(Z(t))\dot{Z}_{ij}(t)]dt \nonumber \\
& =\frac{1}{2}(I_1+I_2+I_3),
\end{align*}
where $ \dot{Z}_{ij}(t)=\frac{t^{-1/2}\widetilde{\mathbf{z}}_{ij}-(1-t)^{-1/2}\widetilde{\mathbf{u}}_{ij}}{\sqrt{n}}$ and $I_k, k=1,2,3$ are defined as 
\begin{equation*}
I_1=\sum_{i=b+1}^n \sum_{j=1}^{p} \int_0^1 \mathbb{E}[\partial_j F(Z^{(i)}(t))\dot{Z}_{ij}(t)]dt, 
\end{equation*}
\begin{equation*}
I_2=\sum_{i=b+1}^n \sum_{k,j=1}^{p} \int_0^1 \mathbb{E}[\partial_k \partial_j F(Z^{(i)}(t))\dot{Z}_{ij}(t)V_k^{(i)}(t)]dt,
\end{equation*}
\begin{equation*}
I_3=\sum_{i=b+1}^n \sum_{k,l,j=1}^{p} \int_0^1 \int_0^1 (1-\tau)\mathbb{E}[\partial_l\partial_k \partial_j F(Z^{(i)}(t)+\tau V^{(i)}(t))\dot{Z}_{ij}(t)V_k^{(i)}(t)V_l^{(i)}(t)]dt d\tau,
\end{equation*}
where $ V^{(i)}(t)=\sum_{j \in \widetilde{N}_i} Z_j(t), \ Z^{(i)}(t)=Z(t)-V^{(i)}(t)$. We first use the following lemma to control the derivatives of $F,$ which are one of the key differences between the max norm in \cite{ZC}  and $L^2$ norm in the present paper. We will prove it later. 
 
\begin{lem}\label{lem_controlderivative} For $z=(z_1,\cdots, z_{p}) \in \mathbb{R}^{p}$ and some constant $C>0,$ we have 
\begin{equation*} 
\left|\partial_{jk}^2 F(z) \right| \leq C \left( \psi^2 \Big| \sum_{s=1}^p e_{js} z_s \Big| \Big| \sum_{s=1}^p e_{ks} z_s \Big|+\psi \right), 
\end{equation*}
\begin{equation*}  
  \left| \partial_{jkl}^3 F(z) \right| \leq C\left[ \psi^3 \Big| \sum_{s=1}^p e_{js} z_s \Big| \Big| \sum_{s=1}^p e_{ks} z_s \Big| \Big| \sum_{s=1}^p e_{ls} z_s \Big|+\psi^2 \Big |\sum_{s=1}^p e_{k} z_s \Big| \right],
\end{equation*}
where $\{e_{kl}\}$ are the entries of the matrix $E.$
\end{lem}
Next we will follow the strategy of the proofs of \cite[Proposition 2.1]{ZC} to control $I_k, k=1,2,3.$ Using the fact that $Z^{(i)}(t)$ and $\dot{Z}_{ij}(t)$ are independent and $\mathbb{E}(\dot{Z}_{ij}(t))=0,$  we conclude that $I_1=0.$ For the control of $I_2, $ define the expanded neighborhood around $N_i$ by 
\begin{equation*}
\mathcal{N}_i:=\{j: \{j,k\} \in E_n \ \text{for some} \ k \in N_i \}, 
\end{equation*}
and $\mathcal{Z}^{(i)}(t)=Z(t)-\sum_{l \in \mathcal{N}_i \cup \widetilde{N}_i}Z_{l}(t)=Z^{(i)}(t)-\mathcal{V}^{(i)}(t),$ where $\mathcal{V}^{(i)}(t)=\sum_{l \in \mathcal{N}_i/\widetilde{\mathcal{N}}_i} Z_l(t)$ with $\mathcal{N}_i/\widetilde{\mathcal{N}}_i:=\{k \in \mathcal{N}_i: k \notin \widetilde{\mathcal{N}}_i \}.$  Using the decomposition in \cite{ZC} (see the discussion below equation (26) of \cite{ZC}), we can write $I_2=I_{21}+I_{22},$ where $I_{21}, I_{22}$ are defined as 
\begin{equation*}
I_{21}=\sum_{i=b+1}^n \sum_{k,j=1}^{p} \int_0^1 \mathbb{E}[\partial_k \partial_j F(\mathcal{Z}^{(i)}(t))] \mathbb{E}[\dot{Z}_{ij}(t)V_k^{(i)}(t)]dt,
\end{equation*}
\begin{equation*}
I_{22}=\sum_{i=b+1}^n \sum_{k,j,l=1}^{p} \int_0^1 \int_0^1 \mathbb{E}[\partial_k \partial_j \partial_l F(\mathcal{Z}^{(i)}(t)+\tau \mathcal{V}^{(i)}(t))\dot{Z}_{ij}(t)V_k^{(i)}(t)\mathcal{V}_l^{(i)}(t)]dtd\tau.
\end{equation*}
We start with the control of $I_{21}.$ 
Using Lemma \ref{lem_controlderivative} and \ref{lem_con} and Assumption \ref{assu_basisregularity}, we conclude that 
\begin{equation*}
\sup_t \mathbb{E}\left | \partial_k \partial_j F(\mathcal{Z}^{(i)}(t)) \right| \leq C \psi^2 p^2.
\end{equation*}
By the equation (28) of \cite{ZC}, we have 
\begin{equation*}
\max_{ 1 \leq j, k \leq p} \sum_{i=b+1}^n \left| \mathbb{E}\dot{Z}_{ij}(t) V_k^{(i)}(t) \right| \leq \phi(M_x, M_y).
\end{equation*}
As a consequence, we can control $I_{21}$ using  
\begin{align*}
|I_{21}| & \leq \sum_{i=b+1}^n \sum_{k,j=1}^p  \int_{0}^1 \mathbb{E}|\partial^2_{jk}F(\mathcal{Z}^{(i)}(t))| \left| \mathbb{E}[\dot{Z}_{ij}(t)V_k^{(i)}(t)] \right|dt \\
& \leq C \phi(M_x, M_y)  \psi^2 p^4. 
\end{align*}
For the control of $I_{22}, $ 
by Lemma \ref{lem_controlderivative}, for some constant $C>0,$ we have that 
\begin{align*}
\sup_t \mathbb{E}\left | \partial_k \partial_j \partial_l F(\mathcal{Z}^{(i)}(t)+\tau \mathcal{V}^{(i)}(t)) \right| \leq C \psi^3 p^3.
\end{align*}
Next, we can bound 
\begin{align*}
& \int_0^1 \mathbb{E} \max_{1 \leq k,j,l \leq p} \sum_{i=b+1}^n |\dot{Z}_{ij}(t) V_k^{(i)}(t)\mathcal{V}_l^{(i)}(t)| dt \\
& \leq \int_0^1 w(t) \left( \mathbb{E} \max_{1 \leq j \leq p} \sum_{i=b+1}^n |\dot{Z}_{ij}(t)/w(t)|^3 \mathbb{E} \max_{1 \leq k \leq p} \sum_{i=b+1}^n|V_k^{(i)}(t)|^3 \mathbb{E} \max_{1 \leq l \leq p} \sum_{i=b+1}^n |\mathcal{V}_l^{(i)}(t)|^3 \right)^{1/3} dt, \\
& \leq C\frac{D_n^3}{\sqrt{n}}(m_{x,3}^3+m_{y,3}^3),
\end{align*}
where $w(t):=1/(\sqrt{t} \wedge \sqrt{1-t})$ and we use the following bounds (see the equations below (28) of \cite{ZC}) 
\begin{equation}\label{eq_bound31}
\mathbb{E} \max_{1 \leq j \leq p} \sum_{i=b+1}^n \left| \dot{Z}_{ij}(t)/w(t)\right|^3 \leq \frac{C}{\sqrt{n}} (m_{x,3}^3+m_{y,3}^3),
\end{equation} 
\begin{equation}\label{eq_bound32}
\mathbb{E} \max_{1 \leq k \leq p} \sum_{i=b+1}^n \left| V^{(i)}_k(t)\right|^3 \leq \frac{CD_n^3}{\sqrt{n}} (m_{x,3}^3+m_{y,3}^3),
\end{equation}
\begin{equation}\label{eq_bound33}
\mathbb{E} \max_{1 \leq l \leq p} \sum_{i=1}^n \left| \mathcal{V}^{(i)}_l(t)\right|^3 \leq \frac{CD_n^6}{\sqrt{n}} (m_{x,3}^3+m_{y,3}^3).
\end{equation}
For the control of $I_3,$ using a similar argument (see the discussion below equation (28) in \cite{ZC}), we find that $|I_3| \leq C \psi^3p^6 I_{31},$ where $I_{31}$ satisfies that
\begin{equation*}
|I_{31}| \leq \int_0^1 w(t) \left( \mathbb{E} \max_{1 \leq j \leq p} \sum_{i=b+1}^n |\dot{Z}_{ij}(t)/w(t)|^3 \mathbb{E} \max_{1 \leq k \leq p} \sum_{i=b+1}^n|V_k^{(i)}(t)|^3 \mathbb{E} \max_{1 \leq l \leq p} \sum_{i=b+1}^n |V_l^{(i)}(t)|^3 \right)^{1/3} dt.
\end{equation*}
Combining with (\ref{eq_bound31}) and (\ref{eq_bound32}), we conclude our proof. 

\end{proof}

\begin{proof}[Proof of Corollary \ref{cor_boundmdependent}] We first  notice that $D_n=2M+1, |\widetilde{N}_i| \leq 2M+1$ and $|\mathcal{N}_i \cup \widetilde{N}_i| \leq 4M+1.$ Define $\Upsilon_i=\{j: \{j,k\} \in E_n \ \text{for some } \ k \in \mathcal{N}_i\},$ then we have $|\Upsilon_i \cup \mathcal{N}_i \cup \widetilde{N}_i| \leq 6M+1. $  Following the arguments of the proof of Lemma \ref{lem_priorbound}, it can be shown that (see the equations above (29) of \cite{ZC}), $I_3$ can be bounded by the following quantity
\begin{align*}
Cn \psi^3 p^6 \times \int_0^1 w(t) \left( \mathbb{\bar{E}} \max_{1 \leq j \leq p} \sum_{i=1}^n |\dot{Z}_{ij}(t)/w(t)|^3 \mathbb{\bar{E}} \max_{1 \leq k \leq p} \sum_{i=1}^n|V_k^{(i)}(t)|^3 \mathbb{\bar{E}} \max_{1 \leq l \leq p} \sum_{i=1}^n |V_l^{(i)}(t)|^3 \right)^{1/3} dt.
\end{align*}
Using a similar discussion to (\ref{eq_bound31}) and (\ref{eq_bound32}), we conclude that 
\begin{equation*}
|I_3| \leq  C\psi^3 p^6 \frac{(2M+1)^2}{\sqrt{n}}(\bar{m}^3_{x,3}+\bar{m}^3_{y,3}). 
\end{equation*}
Similarly, we can bound $I_{22}$ by slightly modifying (\ref{eq_bound33}). We omit further detail and refer to the proof of \cite[Proposition 2.1]{ZC}. This concludes our proof.   
\end{proof}

\begin{proof}[Proof of Lemma \ref{lem_boundmdependenttwo}] 
First of all, by Assumption \ref{assu_basisregularity}, using a similar discussion to (\ref{eq_sigma}), we find that there exist constants $0<c_1<c_2$ such that $c_1<\min_{1 \leq j \leq p} \sigma_{j,j} \leq \max_{1 \leq j \leq  p} \sigma_{j,j} <c_2$ uniformly, where $\sigma_{j,k}=\text{Cov}(\mathbf{Z}_j, \mathbf{Z}_k).$ Under the assumptions that $M_x > u_x(\gamma), M_y >
 u_y(\gamma),$ we can directly use the bounds from the proof of \cite[Corollary 2.1 and Proposition 4.1]{ZC}, where we have that with $1-\gamma$ probability, 
\begin{equation*}
\sup_j |\mathbf{Z}_j-\widetilde{\mathbf{Z}}_j| \leq C\varphi(M_x) \sqrt{8 \log (p/\gamma)}.  
\end{equation*}
As a consequence, we can choose 
\begin{equation*}
 \Delta=C \varphi(M_x, M_y) \sqrt{\log(p/\gamma)}.
\end{equation*}
This finishes our proof. 
\end{proof}

\section{Preliminary lemmas}\label{app_d}

In this section, we collect some preliminary lemmas which will be used in Appendix \ref{sec_proofs} and \ref{appendix_a}.  First of all, we collect a result which provides a deterministic bound of the spectrum of a  square matrix.  Let  $A=(a_{ij})$ be a complex $ n\times n$ matrix. For  $1 \leq i \leq n,$ let  $R_{i}=\sum _{{j\neq {i}}}\left|a_{{ij}}\right| $ be the sum of the absolute values of the non-diagonal entries in the  $i$-th row. Let  $ D(a_{ii},R_{i})\subseteq \mathbb {C} $ be a closed disc centered at $a_{ii}$ with radius  $R_{i}$. Such a disc is called a Gershgorin disc.
\begin{lem}[Gershgorin circle theorem]\label{lem_disc} Every eigenvalue of $ A=(a_{ij})$ lies within at least one of the Gershgorin discs  $D(a_{ii},R_{i})$, where $R_i=\sum_{j\ne i}|a_{ij}|$.
\end{lem}

The above lemma can be extended to the block matrices. We record it as the follow lemma, whose proof can be found in \cite[Section 1.13]{CT}. It will be used in the proof of Theorem \ref{thm_bootstrap}. 

\begin{lem} \label{lem_blockge} For an $b(n-b) \times b(n-b)$ block matrix $\mathcal{A}$ with each diagonal block $A_{ii}, i=1,2,\cdots, n-b$  being symmetric, denote $G_i$ as the region contains the eigenvalue of $\mathcal{A}$ of the $i$-th block,  we then have that 
\begin{equation*}
G_i=\sigma(A_{ii}) \cup  \left\{ \bigcup_{k=1}^{b} R\left( \lambda_k (A_{ii}), \sum_{j=1, j \neq i}^{n-b} ||A_{ij}|| \right) \right\}, \ i=1,2,\cdots, n-b,
\end{equation*}
where $R(\cdot, \cdot)$ denotes the disk 
\begin{equation*}
R(c,r)=\{\lambda: |\lambda-c| \leq r \}.
\end{equation*}
\end{lem}

The next lemma provides a classic result from numerical analysis. It will be used to control the solution of perturbed Yule-Walker equation and utilized in the proof of Proposition  \ref{lem_borderapproximation}.

\begin{lem}\label{lem_numerical} Let $Ax=w, \ x, w \in \mathbb{R}^n $ be the original linear system. And for the perturbed system, 
\begin{equation*}
\left(A+\Delta A\right) \widetilde{x}=w+\Delta w,
\end{equation*}
we have the following control
\begin{equation*}
\frac{||\widetilde{x}-x||}{||x||} \leq ||A|| ||A^{-1}|| \frac{|| \Delta w||}{||w||}.
\end{equation*}
\end{lem}

We then collect the Bernstein's inequality  \cite[Theorem 6.1.1]{JT} for summation of independent random matrices. It will be used in the proof of Theorem \ref{thm_timevaryingcoeff}. 
 \begin{lem}\label{lem_berstein} Let $Y_i, \ i=1,2,\cdots,n$ be a sequence of centered independent random matrices with dimensions $d_1, d_2$. Assume that for each $i,$ we have $\max_i ||Y_i|| \leq R_n$ and define 
\begin{equation*}
\sigma_n^2=\max \left\{ \left| \left| \sum_{i=1}^n \mathbb{E} Y_iY_i^* \right| \right|,  \left| \left| \sum_{i=1}^n \mathbb{E} Y_i^* Y_i \right| \right| \right\},
\end{equation*}
where the norm stands for the largest singular value. The for all $t \geq 0,$ we have 
\begin{equation*}
\mathbb{P} \Big( \Big| \Big | \sum_{i=1}^n Y_i \Big| \Big| \geq t \Big) \leq (d_1+d_2)\exp\Big(\frac{-t^2/2}{\sigma_n^2+R_nt/3} \Big).
\end{equation*}
 \end{lem}

The following lemma indicates that, under suitable condition, the inverse of a band matrix can also be approximated by another band matrix. It will be used in the proof of Proposition \ref{lem_borderapproximation} and can be found in \cite[Proposition 2.2]{DMS}.  We say that $A$ is $m$-banded if $$A_{ij}=0, \ \text{if} \ |i-j|>m/2.$$
\begin{lem}\label{lem_band} Let $A$ be a positive definite, $m$-banded, bounded and bounded invertible matrix.   Let $[a,b]$ be the smallest interval containing the spectrum of $A.$ Set $r=b/a, q=(\sqrt{r}-1)/(\sqrt{r}+1)$ and set $C_0=(1+r^{1/2})^2/(2ar)$ and $\lambda=q^{2/m}.$ Then we have 
\begin{equation*}
|(A^{-1})_{ij}| \leq C \lambda^{|i-j|},
\end{equation*}
where 
\begin{equation*}
C:=C(a,r)=\max\{a^{-1}, C_0\}. 
\end{equation*}

\end{lem}

Finally, we collect the concentration inequalities for  non-stationary process using the physical dependence measure.  It is the key ingredient for the proof of most of the theorems and lemmas. 
It can be found in  \cite[Lemma 6]{ZZ1}. Recall Definition \ref{defn_physical}. 

\begin{lem}\label{lem_con} Let $x_i=G_i(\mathcal{F}_i),$ where $G_i(\cdot)$ is a measurable function and  $\mathcal{F}_i=(\cdots, \eta_{i-1}, \eta_i)$ and $\eta_i, \ i  \in \mathbb{Z}$ are i.i.d  random variables. Suppose that $\mathbb{E}x_i=0$ and $\max_i \mathbb{E}|x_i|^q<\infty$ for some $q>1.$ For some $k>0,$ let $\delta_x(k):=\max_{ 1 \leq i \leq n} \norm{G_i(\mathcal{F}_i)-G_i(\mathcal{F}_{i,i-k})}_q.$ We further let $\delta_x(k)=0$ if $k<0.$ Write $\gamma_k=\sum_{i=0}^k \delta_x(i).$ Let $S_i=\sum_{j=1}^i x_j.$ \\
(i). For $q'=\min(2,q),$
\begin{equation*}
\norm{S_n}_q^{q'} \leq C_q \sum_{i=-n}^{\infty} (\gamma_{i+n}-\gamma_i)^{q'}.
\end{equation*}
(ii). If $\Delta:=\sum_{j=0}^{\infty} \delta_x(j) <\infty,$ we then have 
\begin{equation*}
\norm{\max_{1 \leq i \leq n}|S_i|}_q \leq C_q n^{1/q'} \Delta. 
\end{equation*}
In (i) and (ii), $C_q$ are generic finite constants which only depend on $q$ and can vary from place to place.  
\end{lem}

\end{appendix}


\end{document}